\documentclass[12pt,amstex]{amsart}

\usepackage[top=25mm, bottom=25mm, left=30mm, right=30mm]{geometry}

\usepackage{mathptmx}
\usepackage{mathrsfs}

\usepackage{verbatim}
\usepackage{url}
\usepackage[all]{xy}
\usepackage{color}

\usepackage[colorlinks=true,citecolor=blue]{hyperref}

\usepackage{stmaryrd}
\usepackage{epsfig}
\usepackage{amsmath}
\usepackage{amssymb}

\usepackage{amscd}
\usepackage{graphicx}
\usepackage{pstricks}
\usepackage{tocvsec2}

\newtheorem{prob}{Problem}
\newtheorem{conj}{Conjecture}
\newtheorem{ques}{Question}

\theoremstyle{plain}
\newtheorem{thm}{Theorem}[section]
\newtheorem{lem}[thm]{Lemma}

\newtheorem{prop}[thm]{Proposition}
\newtheorem{cor}[thm]{Corollary}

\theoremstyle{definition}
\newtheorem{de}[thm]{Definition}
\newtheorem{exam}[thm]{Example}
\theoremstyle{remark}
\newtheorem{rem}[thm]{Remark}

\numberwithin{equation}{section}

\def \N {\mathbb N}
\def \C {\mathbb C}
\def \Z {\mathbb Z}
\def \R {\mathbb R}
\def \E {\mathbb{E}}

\def \A {\mathcal{A}}
\def \B {\mathcal{B}}
\def \F {\mathcal{F}}

\def \I {\mathcal{I}}

\def \X {\mathcal{X}}
\def \Y {\mathcal{Y}}
\def \ZZ {\mathcal{Z}}
\def \O {\mathcal{O}}

\def \id {{\rm id}}
\def \h {\hat }

\def \HK {\interleave}

\def \a {\alpha }

\def \ep {\epsilon}
\def \d {\delta}
\def \D {\Delta}

\def\w {\omega}

\begin{document}
\title{Polynomial Furstenberg joinings and its applications}

\author{Wen Huang}
\author{Song Shao}
\author{Xiangdong Ye}

\address{School of Mathematical Sciences, University of Science and Technology of China, Hefei, Anhui, 230026, P.R. China}

\email{wenh@mail.ustc.edu.cn}
\email{songshao@ustc.edu.cn}
\email{yexd@ustc.edu.cn}


\thanks{This research is supported by National Natural Science Foundation of China (12031019, 12090012, 11971455, 11731003).}


\begin{abstract}
In this paper, a polynomial version of Furstenberg joining is introduced and its structure is investigated. Particularly, it is shown that if all  polynomials are non-linear, then almost every ergodic component of the joining is a direct product of an infinity-step pro-nilsystem and a Bernoulli system.

As applications, some new convergence theorems 
are obtained. Particularly, it is proved that
if $T$ and $S$ are ergodic measure preserving transformations on a probability space $(X,\X,\mu)$ and $T$ has zero entropy, then for all $c_i\in \Z\setminus \{0\}$, all integral polynomials $p_j$ with $\deg {p_j}\ge 2$, and for all $f_i, g_j\in L^\infty(X,\mu)$, $1\le i\le m$ and
$1\le j\le d$,
$$\lim_{N\to\infty} \frac{1}{N}\sum_{n=0}^{N-1}f_1(T^{c_1n}x)\cdots f_m(T^{c_mn}x)\cdot g_1(S^{p_1(n)}x)\cdots g_d(S^{p_d(n)}x),$$
exists in $L^2(X,\mu)$,
which extends the recent result by Host and Frantzikinakis. 

Moreover, it is shown that for an ergodic measure-preserving system $(X,\X,\mu,T)$, a non-linear integral polynomial $p$ and $f\in L^\infty(X,\mu)$,  the Furstenberg systems of $\big(f(T^{p(n)})x\big)_{n\in \Z}$ are ergodic and isomorphic to direct products of infinite-step pro-nilsystems and Bernoulli systems for almost every $x\in X$, which answers a problem  by Frantzikinakis.
\end{abstract}

\maketitle



\tableofcontents \settocdepth{section}


\section{Introduction}

In the article, the set of integers (resp. natural numbers $\{1,2,\ldots\}$) is denoted by $\Z$ (resp. $\N$).
An {\em integral polynomial} is a polynomial taking integer values at the
integers. Two polynomials $p(n)$ and $q(n)$ are {\em essentially distinct} if $p(n)-q(n)$ is not a constant function.
If $n<m\in\Z$, then $[n,m]$ denotes the set $\{n,n+1,\ldots,m\}$.

\medskip

Throughout this paper, all probability spaces $(X,\X,\mu)$ considered are Lebesgue, meaning $X$ can be given the structure of
a compact metric space and $\X$ is its Borel $\sigma$-algebra. By a {\em measure-preserving system} (m.p.s. for short), we mean a quadruple
$(X,\X, \mu, T)$, where $(X,\X,\mu )$ is a Lebesgue probability space and $T: X \rightarrow X$ is an invertible measure preserving transformation.

\subsection{Furstenberg Joining}\
\medskip

The notion of joinings was introduced by Furstenberg in his seminal paper \cite{F67}, and now the theory of joinings is an
important part of ergodic theory \cite{Glasner}. Some joinings play a key role in the study of multiple recurrence and
multiple ergodic averages. To give a dynamical proof of the well-known theorem of Szemer\'{e}di \cite{Sz}, Furstenberg proved the following multiple recurrence theorem \cite{F77}, which says that for a m.p.s. $(X,\X,\mu,T)$, $d \in \N$ and  $A\in \mathcal{X}$ with positive measure, there is some $n\in \N$ such that
	\begin{equation*}
	\mu(A\cap T^{-n}A\cap T^{-2n}A\cap \cdots \cap T^{-dn}A)>0.
	\end{equation*}
In the same paper, Furstenberg introduced the following measure (now called {\em Furstenberg joining}):
\begin{equation}\label{a2}
  \mu^{(d)}(\prod_{i=1}^dA_i)=\lim_{N\to\infty} \frac{1}{N} \sum_{n=1}^N \mu(T^{-n}A_1\cap T^{-2n}A_2\cap \cdots \cap T^{-dn}A_d),
\end{equation}
for all $A_1,\ldots, A_d\in \X$. Note that the existence of the limit was proved in \cite{HK05, Z},
and in \cite{F77} $\mu^{(d)}$ was in fact defined as a limit of a subsequence.
It is easy to verify that $\mu^{(d)}$ is invariant under the actions of $T^{(d)}=T\times \cdots \times T$ ($d$ times) and $\tau_d=T\times T^2\times \cdots \times T^d$. Thus we have a m.p.s. $(X^d,\X^d,\mu^{(d)}, \tau_d )$.

In \cite{F77}, it was shown that
$\mu^{(d)}$ is a conditional product measure in $X^d$ relative to $(Y_{d-1}, \mu_{d-1})$, where $\{(Y_n,\mu_n)\}_n$ is the distal sequence of $(X,\X,\mu,T)$ (\cite{F}). That is,
\begin{equation}\label{a3}
  \mu^{(d)}=\int_{Y_{d-1}^d}\mu_{y_1}\times \mu_{y_2}\times \cdots \times \mu_{y_d}d \mu_{d-1}^{(d)} (y_1,y_2,\ldots, y_d),
\end{equation}
where $\mu=\int \mu_{y}\ d\mu_{d-1}(y)$ is the disintegration of $\mu$ with respect to $\mu_{d-1}$, and
$\mu_{d-1}^{(d)}$ is defined as (\ref{a2}) related to $Y_{d-1}$.
This leads to the notion of ``characteristic factors'', which was first explicitly defined by Furstenberg and Weiss in \cite{FW96}. Let $(X,\X,\mu, T)$ be a m.p.s. and $(Y,\Y,\nu, T)$ be a factor of $X$. For $d\in \N$, we say that $Y$ is a {\em characteristic factor of $X$ for $\tau_d$} if for all $f_1,\ldots,f_d\in L^\infty(X,\mu)$,
		\begin{equation*}
		\begin{split}
		\frac{1}{N}\sum_{n=0}^{N-1} \prod_{i=1}^d f_i(T^{in}x)- \frac{1}{N}\sum_{n=0}^{N-1} \prod_{i=1}^d \E(f_i|\Y)(T^{in}x) \to 0,\ N\to\infty
		\end{split}
		\end{equation*}
in $L^2(X,\mu)$. By (\ref{a3}) it is easy to see $Y_{d-1}$ is a characteristic factor of $X$ for $\tau_d$.
This enables one to give a reduction of the problem on the multiple recurrence to its distal factors.

In \cite{HK05} and in \cite{Z} the authors improved the above result significantly by showing that for an ergodic m.p.s. $(X,\X,\mu,T)$, its $(d-1)$-step pro-nilfactor is a characteristic factor for $\tau_d$, and using this result the authors proved that
$$\displaystyle \lim_{N\to\infty} \frac{1}{N} \sum_{n=0}^{N-1}  f_1(T^{n}x)  f_2(T^{2n}x) \cdots f_d(T^{dn}x)$$
exists in $L^2(X,\mu)$. 

\subsection{Polynomial version of the Furstenberg joining}\
\medskip

Let $(X,\X,\mu,T)$ be a m.p.s. and let $\A=\{p_1,p_2,\ldots, p_d\}$ be a family of non-constant integral polynomials.
There is a natural way to generalize the Furstenberg joining to the polynomial setting as follows:
\begin{equation}\label{}
  \mu_\A^{(d)}(\prod_{i=1}^d A_i)=\lim_{N\to\infty} \frac{1}{N} \sum_{n=1}^N \mu(T^{-p_1(n)}A_1\cap T^{-p_2(n)}A_2\cap \cdots \cap T^{-p_d(n)}A_d),
\end{equation}
for all $A_1,\ldots, A_d\in \X$, where the limit exists \cite{HK051, Leibman05-Isr}. Of course, $\mu_\A^{(d)}$ is a probability measure on $X^d$ and it is $T^{(d)}$-invariant (in fact, it can be proved to be a self-joining). Unlike Furstenberg joining, we can not find some transformation $\tau$ related to $\A$ like $\tau_d$ such that $(X^d,\X^d,\mu_\A^{(d)}, \tau )$ becomes a m.p.s. Thus to study the dynamical behavior related to polynomials, we need to find another  way to generalize the Furstenberg joining.

\medskip

To overcome the difficulty, when studying the dynamical behavior concerning polynomials $\{p_1,\ldots,p_d\}$, we will consider  infinitely
many polynomials $p_i^{[k]}, 1\le i\le d, k\in \Z$ defined by $p_i^{[k]}(n)=p_i(n+k)-p_i(k), n\in\Z$ at the same time. For example, for $\A=\{n^2, n^3\}$, we define a measure $\mu^{(\infty)}_\A$ on $(X^2)^\Z$ as follows: for all $l\in \N$, $A_i^{(1)}, A_i^{(2)}\in \X$ and $-l\le i\le l$, put
\begin{equation*}
  \begin{split}
    & \mu^{ (\infty)}_{\A}\Big((\prod_{j=-\infty}^{-l-1} X^2)\times \prod_{i=-l}^l (A^{(1)}_i\times A^{(2)}_i)\times (\prod_{j=l+1}^\infty X^2)\Big)\\
&\quad \quad \quad \quad \quad \quad =\lim_{N\to\infty} \frac{1}{N} \sum_{n=1}^N \mu\Big(\bigcap_{i=-l}^l (T^{-(n+i)^2}A^{(1)}_i\cap T^{-(n+i)^3}A^{(2)}_i)\Big),
   \end{split}
\end{equation*}
where the  existence of the limit is guaranteed by \cite{HK051, Leibman05-Isr}.

Similarly, for $\A=\{p_1,\ldots, p_d\}$, we define a measure $\mu^{(\infty)}_\A$ on $(X^d)^\Z$. This measure
$\mu^{(\infty)}_\A$ is invariant under the product map $T^{\infty}=\cdots \times T^{(d)}\times T^{(d)}\times \cdots$ and the shift map $\sigma: (X^d)^\Z\rightarrow (X^d)^\Z$. 
Thus for a given m.p.s. $(X,\X,\mu,T)$ and $\A=\{p_1,\ldots, p_d\}$, we associate two m.p.s. $$\big((X^d)^\Z,(\X^d)^\Z, \mu_\A^{(\infty)},  \sigma \big)\ \text{and}\ \big((X^d)^\Z,(\X^d)^\Z, \mu_\A^{(\infty)},
\langle T^\infty, \sigma\rangle \big),$$
where $\langle T^\infty, \sigma\rangle$ is the group generated by $\sigma$ and $T^\infty$. The first one is a $\Z$-system and the second one
is a $\Z^2$-system, as $\sigma\circ T^\infty=T^\infty\circ \sigma$.
The measure $\mu^{(\infty)}_\A$ is a joining on $(X^d)^\Z$, and we call it an {\em $\infty$-joining with respect to $\A$}, see Section \ref{section-furstenberg-joining} for detailed discussions. In this paper, we will focus on the dynamical properties of this joining and give its applications.

Note that in a companion paper by the same authors \cite{HSY22-1}, for a given topological system $(X,T)$ and $\A=\{p_1,\ldots,p_d\}$, we associate two topological dynamical systems $(W_x^\A,\sigma)$, and $(M_\infty(X,\A), \langle T^{\infty},\sigma\rangle)$. It is proved that if $(X,T)$ is minimal, then $(M_\infty(X,\A), \langle T^{\infty},\sigma\rangle)$ is transitive and has dense minimal points, which can be used to answer some open question.

\subsection{Main results}\
\medskip

Let $(X,\X,\mu, T)$ be an ergodic m.p.s. and let $\A=\{p_1,\ldots, p_d\}$ be a family of non-constant integral polynomials. First we will study some basic properties of  $((X^d)^\Z,(\X^d)^\Z, \mu_\A^{(\infty)}, \sigma)$ and $((X^d)^\Z,(\X^d)^\Z, \mu_\A^{(\infty)},
\langle T^\infty, \sigma\rangle)$.  It is shown that: 

\begin{itemize}
  \item[$\blacklozenge$] $\big((X^d)^\Z,(\X^d)^\Z, \mu_\A^{(\infty)}, \langle T^\infty, \sigma\rangle \big)$ is ergodic (Theorem \ref{thm-ergodic-measures}-(1)).
  \item[$\blacklozenge$]One can characterize the ergodic decomposition of $\mu_\A^{(\infty)}$ under $\sigma$
  (Theorem \ref{thm-ergodic-measures}-(2), (3)).
  \item[$\blacklozenge$] If the degree of each polynomial  in $\A$ is greater than 1, then for $\mu$ a.e. $x\in X$, $\big((X^d)^\Z, (\X^d)^\Z, {\mu}^{(\infty)}_{\A,x}, \sigma \big)$ is a direct product of an infinity-step pro-nilsystem
      and a Bernoulli system (Theorem \ref{measure-like}). 
   \item[$\blacklozenge$] If $(X,T)$ is a topological system, $\mu$ is an invariant Borel probability measure on $X$ with full support,
   then $\mu_\A^{(\infty)}$ is supported on $(N_\infty(X,\A), \langle T^{\infty},\sigma\rangle)$ (Theorem \ref{support}).
\end{itemize}

Note that $(N_\infty(X,\A), \langle T^{\infty},\sigma\rangle)$ is an equivalent form of $(M_\infty(X,\A), \langle T^{\infty},\sigma\rangle)$,
see \cite[Section 4]{HSY22-1} for details.

\medskip

Also we will use $\big((X^d)^\Z,(\X^d)^\Z, \mu_\A^{(\infty)}, \sigma\big)$ to study the convergence of multiple averages along polynomials.
Particularly it is proved that 
\begin{itemize}
  \item[$\blacktriangledown$] Let $T$ and $S$ be ergodic measure preserving transformations on a probability space $(X,\X,\mu)$ and with $T$  having zero entropy. Let $c_1,\ldots, c_m\in \Z\setminus \{0\}$ and $p_i$ be an integral polynomial with $\deg {p_i}\ge 2, 1\le i\le d$. Then for all $f_1,\ldots, f_m, g_1,\ldots, g_d\in L^\infty(X,\mu)$, 
\begin{equation*}
  \lim_{N\to\infty} \frac{1}{N}\sum_{n=0}^{N-1} f_1(T^{c_1n}x)\cdots f_m(T^{c_mn}x)\cdot g_1(S^{p_1(n)}x)\cdots g_d(S^{p_d(n)}x)
\end{equation*}
exists in $L^2(X,\mu)$ (Theorem \ref{thm-mean}).
\end{itemize}

When $m=d=1$, the result above was given in \cite{FranHost21}. The main point here is that $T$ and $S$ do not satisfy any commutativity assumptions.


\medskip

 Moreover, we will compare the system defined in this paper with Furstenberg systems of sequences defined in \cite{FranHost18,Fran22} (see Section \ref{section-Furst-systems} for definitions), and show the following result, which answers Problem 1 in \cite{Fran22}.
\begin{itemize}
  \item[$\blacktriangle$] Let $(X,\X,\mu, T)$ be an ergodic m.p.s., $p$ be a non-linear integral polynomial, and $f\in L^\infty(X,\mu)$.
  Then for any strictly increasing sequence of positive integers $\{N_k\}_{k\in \N}$ there is a subsequence $\{N'_k\}_{k\in \N}$  such that for a.e. $x\in X$ the Furstenberg systems of $\big(f(T^{p(n)})x\big)_{n\in \Z}$ along $\{N'_k\}_{k\in \N}$ are ergodic and isomorphic to the direct products of infinite-step
  pro-nilsystems and Bernoulli systems (Theorem \ref{thm-answer-Fran}).
\end{itemize}

\subsection{Organization of the paper}

We organize the paper as follows. In Section \ref{section-pre} we state some necessary notions and some known facts used in
the paper. In Section \ref{section-furstenberg-joining}, we define $\infty$-joinings $\mu^{(\infty)}_\A$ and $\widetilde{\mu}^{(\infty)}_\A$
with respect to a given m.p.s. and a family of integral polynomials, and explain them through certain examples. In Section \ref{section-reduing-nil},
we study the $\sigma$-algebra of invariant subsets under the shift map, and show how to relate it to the one in pro-nilfactors. In Section \ref{section-nilsystem} and Section \ref{section-ergodic-decom}, we give the ergodic decomposition theorem of $\mu_\A^{(\infty)}$ under the shift map $\sigma$. In the next two sections, we derive some applications of the theory we have developed. Namely, in Section \ref{section-applications}, we provide some new convergence theorems along integral polynomials, and in Section \ref{section-Furst-systems}, we compare our notion with Furstenberg systems related to sequences.
In the last section, we present some questions.

\bigskip

\noindent{\bf Acknowledgments:} We thank Hanfeng Li for useful discussions concerning Theorem~ \ref{top-model}.
Thank Hui Xu for very careful readings and corrections, and Thank Nikos Frantzikinakis for useful comments.

\section{Preliminaries}\label{section-pre}

In this section we give some necessary notions and some known facts used in the paper.
Note that, instead of just considering a single transformation $T$,
sometimes we study commuting
transformations $T_1$, $\ldots$ , $T_k$ of $X$. We only recall some basic
definitions and properties of systems for
one transformation, and the extensions to the general case are
straightforward.

\subsection{Some notion in ergodic theory and topological dynamics}

\subsubsection{Measurable systems}

Recall that in this paper we consider invertible m.p.s. \linebreak $(X,\X, \mu, T)$ defined on a Lebesgue probability space.
Thus $X$ can be given the structure of a compact metric space and $\X$ is its Borel $\sigma$-algebra.

For a m.p.s. $(X,\X,\mu,T)$, the {\em associated unitary operator (Koopman operator) of $T$}
 is $U_T: L^2(X,\mu)\rightarrow L^2(X,\mu), f\mapsto f\circ T$. In the sequel, we usually denote $U_Tf$ by $Tf$.

\medskip

For a m.p.s. $(X,\X, \mu, T)$ we write $\I(X,\X,\mu, T)$ for the $\sigma$-algebra
$\{A\in \X : T^{-1}A = A\}$ consisting of invariant sets. Sometimes we will use $\I$ or $\I (T)$ for short.
A m.p.s. is {\em ergodic} if all
the $T$-invariant sets have measures either $0$ or $1$. $(X,\X, \mu,
T)$ is {\em weakly mixing} if the product system $(X\times X,
\X\times \X, \mu\times \mu, T\times T)$ is erdogic.

\medskip
Let $X, Y$ be spaces and $\phi: X \to Y$ be a map.  For $n \geq 2$ let
$$\phi^{(n)}=\underbrace{\phi\times \cdots \times \phi }_{\text{($n$ times)}}: X^n\rightarrow Y^n.$$
For a m.p.s. $(X,\X,\mu,T)$, we write $(X^n,\X^n,\mu^d, T^{(n)})$ for the $n$-fold product system $(X\times	\cdots \times X,\X\times\cdots \times \X, \mu\times \cdots\times \mu, T\times \cdots \times T)$.
The diagonal of
$X^n$ is $$\Delta_n(X)=\{x^{\otimes n}=(x,\ldots,x)\in X^n: x\in X\}.$$
When $n=2$ we write	$\Delta(X)=\Delta_2(X)$.

\medskip

A {\em Bernoulli system} has the form $(Y^\Z,\Y^\Z,\nu^\Z,\sigma)$, where $(Y,\Y,\nu)$ is a probability space, $\sigma$ is the shift 
on $Y^\Z$, $\Y^\Z$ is the product $\sigma$-algebra of $Y^\Z$, and $\nu^\Z$ is the product measure.

\medskip

A {\em factor map} or {\em homomorphism} from  $(X,\X, \mu, T)$ to $(Y,\Y, \nu,
S)$ is a measurable map $\pi : X_0 \rightarrow  Y_0$, where $X_0\in \X$ is
$T$-invariant and $Y_0\in \Y$ is $S$-invariant, both of full measure, such that $\pi_*\mu \triangleq \mu\circ
\pi^{-1}=\nu$ and $S\circ \pi(x)=\pi\circ T(x)$ for $x\in X_0$. When
we have such a homomorphism we say that $(Y,\Y, \nu, S)$
is a {\em factor} of  $(X,\X, \mu , T)$. If the factor map
$\pi: X_0\rightarrow  Y_0$ can be chosen to be bijective, then we
say that $(X,\X, \mu, T)$ and $(Y,\Y, \nu, S)$ are {\em
(measure theoretically) isomorphic} (bijective maps on Lebesgue
spaces have measurable inverses), and denote it by $(X,\X, \mu, T)\cong (Y,\Y,\nu, S)$.

Let $(Y,\Y,\nu, S)$ be a factor of $(X,\X,\mu,T)$. A factor can be characterized
(modulo isomorphism) by $\pi^{-1}(\Y)$, which is a $T$-invariant
sub-$\sigma$-algebra of $\X$, and conversely any $T$-invariant
sub-$\sigma$-algebra of $\X$ defines a factor. By a classical abuse
of terminology we denote by the same letter the $\sigma$-algebra
$\Y$ and its inverse image by $\pi$, $\pi^{-1}(\Y)$. In other words, if $(Y,\Y, \nu,
S)$ is a factor of $(X,\X, \mu, T)$, we think of $\Y$ as a
sub-$\sigma$-algebra of $\X$. Thus we may regard $L^2(Y,\Y,\nu)$ as a closed subspace of $L^2(X,\X,\mu)$. When there is no danger of confusion, we may abuse the notation and denote the transformation $S$ on $Y$ by $T$.

\medskip

We say that $(X,\X, \mu, T)$ is an {\em inverse limit} of a sequence
of factors $\{(X_j,\X_j ,\mu_j, T)\}_{j\in \N}$ if $(\X_j)_{j\in\N}$ is an increasing
sequence of $T$-invariant sub-$\sigma$-algebras such that
$\bigvee_{j\in \N}\X_j=\X$ up to sets of measure zero.

\subsubsection{Topological dynamical systems}

A {\em topological dynamical system} (t.d.s. for short) is a pair $(X, T)$,
where $X$ is a compact metric space and $T$ is a homeomorphism from $X$ to itself.
A t.d.s. $(X, T)$ is {\em transitive} if there exists
some point $x\in X$ whose orbit $\O(x,T)=\{T^nx: n\in \Z\}$ is dense
in $X$ and we call such a point a {\em transitive point}. A t.d.s. $(X,T)$
is {\em minimal} if the orbit of any point is dense in $X$. This
property is equivalent to saying that X is the only closed non-empty invariant subset in $X$.
\medskip

A {\em topological factor} of a t.d.s. $(X, T)$ is another t.d.s. $(Y, S)$ such that there exists a continuous and
onto map $\phi: X \rightarrow Y$ satisfying $S\circ \phi = \phi\circ
T$. In this case, $(X,T)$ is called an {\em topological extension } of $(Y,S)$.
The map $\phi$ is called a {\em topological factor map}.


\subsubsection{${\mathcal M}(X)$ and ${\mathcal M}_T(X)$}

For a t.d.s. $(X,T)$, denote by ${\mathcal M}(X)$ the set of all Borel
probability measures on $X$. Let ${\mathcal M}_T(X)=\{\mu\in {\mathcal M}(X):
T_*\mu=\mu\circ T^{-1}=\mu\}$. 
With the weak$^*$-topology, ${\mathcal M}(X)$ and ${\mathcal M}_T(X)$ are compact convex spaces. By the Krylov-Bogolioubov
theorem ${\mathcal M}_T(X)\neq \emptyset$. Denote by ${\mathcal M}^{erg}_T(X)$ the set of ergodic measures of $(X,T)$, then ${\mathcal M}^{erg}_T(X)$ is the set of extreme points of ${\mathcal M}_T(X)$ and one can use the Choquet representation theorem to express each member
of ${\mathcal M}_T(X)$ in terms of the ergodic members of ${\mathcal M}_T(X)$. That is, for each $\mu\in {\mathcal M}_T(X)$ there is a unique measure
$\tau$ on the Borel subsets of the compact space ${\mathcal M}_T(X)$ such that $\tau\big({\mathcal M}^{erg}_T(X)\big)=1$ and $\displaystyle \mu=\int_{{\mathcal M}^{erg}_T(X)}m d\tau (m)$,
which is called the {\em ergodic decomposition} of $\mu$.
In particular, ${\mathcal M}_T^{erg}(X)\neq \emptyset$.

A t.d.s. $(X,T)$ is called {\em uniquely ergodic} if
there is a unique $T$-invariant Borel probability measure on $X$. It is
called {\em strictly ergodic} if it is uniquely ergodic and minimal.



\subsection{Topological models}

\subsubsection{Topological models}
Let $(X,\X, \mu, T)$ be an ergodic m.p.s. We say that  $(\h{X}, \hat{T})$
is a {\em topological model} (or just a {\em model}) for $(X,\X,
\mu, T)$ if $(\h{X}, \h{T})$ is a t.d.s. and there exists an invariant
probability measure $\h{\mu}$ on the Borel $\sigma$-algebra
$\mathcal{B}(\h{X})$ such that the systems $(X,\X, \mu, T)$ and
$(\h{X}, \mathcal{B}(\h{X}), \h{\mu}, \h{T})$ are measure theoretically
isomorphic.

\medskip

The well-known Jewett-Krieger's theorem  \cite{Jewett, Krieger}
states that every ergodic system has a strictly ergodic model.

\subsubsection{Weiss's Theorem}
We say that $\h{\pi}: \h{X}\rightarrow \h{Y}$ is a {\em topological
model} for a factor map $\pi: (X,\X, \mu, T)\rightarrow (Y,\Y, \nu, S)$ if
$\h{\pi}$ is a topological factor map and there exist measure
theoretical isomorphisms $\phi$ and $\psi$ such that the diagram
\[
\begin{CD}
X @>{\phi}>> \h{X}\\
@V{\pi}VV      @VV{\h{\pi}}V\\
Y @>{\psi }>> \h{Y}
\end{CD}
\]
is commutative, i.e. $\h{\pi}\phi=\psi\pi$. Weiss \cite{Weiss85}
generalized Jewett-Krieger Theorem to the relative case.
Namely, he proved that if $\pi: (X,\X, \mu, T)\rightarrow (Y,\Y, \nu,
S)$ is a factor map with $(X,\X, \mu, T)$ ergodic and
$(\h{Y},\h{\Y}, \h{\nu}, \h{S})$ is a uniquely ergodic model for $(Y,\Y,
\nu, T)$, then there is a uniquely ergodic model $(\h{X}, \h{\X},
\h{\mu}, \h{T})$ for $(X,\X, \mu, T)$ and a factor map $\h{\pi}:
\h{X}\rightarrow \h{Y}$ which is a model for $\pi: X\rightarrow Y$.
We will refer this theorem as {\it Weiss's Theorem}.
We note that in \cite{Weiss85} Weiss pointed that the relative case holds
for commutative group actions.

\subsection{The Rohlin skew-product theorem}\
\medskip

Let $(Y,\Y,\nu, S)$ be a m.p.s. and $(U, {\mathcal U},\rho)$ a probability space. Let $\a: Y\rightarrow {\rm Aut} (U,\rho)$ be a measurable map ($\a$ is called a {\em cocycle}), where ${\rm Aut} (U,\rho)$ is the Polish group of all invertible measure-preserving transformations of $(U,{\mathcal U},\rho)$. We define the {\em skew-product system} $(Y\times U, \Y\times {\mathcal U}, \nu\times \rho, T_\a)$ as follows:
$$T_\a(y,u)=(Sy, \a(y)u), \quad \forall (y,u)\in Y\times U.$$
The following is the well known Rohlin skew-product theorem (see \cite[Theorem 3.18]{Glasner}).

\begin{thm}[Rohlin skew-product theorem]\label{Rohlin}
Let $\pi: (X,\X,\mu,T)\rightarrow (Y,\Y,\nu,S)$ be a factor map of ergodic m.p.s. Then $(X,\X,\mu,T)$ is isomorphic to a skew-product over $(Y,\Y,\nu,S)$. That is, there exist a probability space $(U, {\mathcal U},\rho)$ and a measurable cocycle $\a: Y\rightarrow {\rm Aut} (U,\rho)$ such that
$$(X,\X,\mu,T)\cong (Y\times U, \Y\times {\mathcal U}, \nu\times \rho, T_\a).$$
\end{thm}

\subsection{Conditional expectations}

\subsubsection{Conditional expectations}
Let $\pi: (X,\X,\mu,T)\rightarrow (Y,\Y,\nu, T)$ be a factor map. As mentioned above, $\Y$ is also regarded as a $T$-invariant sub-$\sigma$-algebra of $\X$. For each $f\in
L^1(X, \mu)$, we write $\E(f|\Y)$, or $\E_\mu(f|\Y)$ if needed, for the
{\em conditional expectation} of $f$ with respect to $\Y$.
We will frequently make use of the
identities $$\int_X \E(f|\Y) \ d\mu = \int_X f \ d\mu \quad
\text{and}\quad T \E(f|\Y) = \E(Tf|\Y).$$ We say that a function $f\in L^2(X,\mu)$
is {\em orthogonal} to $\Y$, and we write $f \perp \Y$, when it has a zero
conditional expectation on $\Y$.



\subsubsection{}
The {\em  disintegration of $\mu$ over $\nu$}, written as $\displaystyle \mu=\int_Y \mu_y\ d \nu(y)$,
is given by a measurable map
$y \mapsto \mu_y$ from $Y$ to the space of Borel probability measures on
$X$  such  that
\begin{equation}
    \E(f|\Y)(y)=\int_X f d\mu_y
\end{equation}
$\nu$-almost everywhere.

\subsection{Joinings}

\subsubsection{Joinings and conditional product measures}

The notions of joining and conditional product measure were introduced by Furstenberg in \cite{F67} and \cite{F77}.
Let $(X_i, \X_i, \mu_i,T_i), i\in \{1,2,\ldots, k\}$ be m.p.s. A Borel probability measure $\lambda$ on
$X=\prod_{i=1}^k X_i$ defines a {\em joining} of the measures on $X_i$ if
it is invariant under $\prod_{i=1}^k T_i=T_1\times \cdots \times T_k$ and maps onto
$\mu_j$ under the natural map $\prod_{i=1}^k X_i\rightarrow X_j, \forall j\in \{1,2,\ldots, k\}$. When $(X_1,\X_1,\mu_1, T_1)=(X_2,\X_2,\mu_2, T_2)=\cdots=(X_k, \X_k, \mu_k, T_k)$, we then say that
$\lambda$ is a {\em $k$-fold self-joining}. One can generalize the notion to infinitely many m.p.s. easily.

Let $(X_i, \X_i, \mu_i,T_i)$ be a m.p.s., $(Y_i,\Y_i, \nu_i,S_i)$ be a corresponding factor, and
$\pi_i:X_i\rightarrow Y_i$ be a factor map, $1\le i\le k$.
Let $\nu$ be a joining of the measures on $Y_i, i=1,\cdots, k$, and
let $\mu_i=\int \mu_{X_i, y_i}\ d\nu_i(y_i)$ represent the
disintegration of $\mu_i$ with respect to $\nu_i$. Let $\mu$ be a
measure on $X=\prod_i X_i$ defined by
\begin{equation}\label{}
    \mu=\int_Y \mu_{X_1,y_1}\times \mu_{X_2,y_2}\times \cdots \times
    \mu_{X_k,y_k}\ d\nu(y_1,y_2,\cdots,y_k).
\end{equation}
Then $\mu$ is called the {\em conditional product measure with
respect to $\nu$}.

Equivalently, $\mu$ is a conditional product measure relative to $\nu$
if and only if for all $k$-tuple $f_i\in L^\infty(X_i,\mu_i),
i=1,\ldots, k$
\begin{equation}
\begin{split}
\int_X
   &f_1(x_1)f_2(x_2)\cdots f_k(x_k)\ d\mu(x_1,x_2,\cdots,x_k)\\
   &\quad = \int_Y \E(f_1|\Y_1)(y_1)\E(f_2|\Y_2)(y_2)\cdots \E(f_{k}|\Y_k)(y_{k})\
   d\nu (y_1,y_2,\cdots,y_{k}).
\end{split}
\end{equation}

\subsubsection{Relatively independent joining}

Let $(X_1,\X_1,\mu_1,T_1), (X_2,\X_2,\mu_2,T_1)$ be m.p.s. and let
$(Y,\Y,\nu, S)$ be a common factor of them with $\pi_i: X_i\rightarrow Y$ for
$i = 1, 2$ the factor maps. Let $\mu_i=\int_Y \mu_{i,y}\ d\nu(y)$
represent the disintegration of $\mu_i$ with respect to $\nu$. Let
$\mu_1\times_{Y} \mu_2$ denote the measure defined by
$$\mu_1\times_{Y} \mu_2(A)=\int_Y \mu_{1,y}\times \mu_{2,y}(A)\ d\nu(y),$$
for all $A\in \X_1\times \X_2$. The system $(X_1\times X_2,
\X_1\times \X_2,\mu_1\times_Y \mu_2, T_1\times T_2)$ is called the
{\em relative product} of $X_1$ and $X_2$ with respect to $Y$ and is
denoted $X_1\times_{Y} X_2$. $\mu_1\times_{Y} \mu_2 $ is also
called {\em a relatively independent joining} of $X_1$ and $X_2$ over
$Y$.

\subsection{Furstenberg Joinings}\

\subsubsection{Furstenberg Joinings}
Let $T: X\rightarrow X$ be a map and $d\in \N$. Set
$$\tau_{\vec{a}}=T^{a_1}\times T^{a_2} \times \cdots \times T^{a_d},$$
where $\vec{a}=(a_1,\ldots, a_d)\in \Z^d$.

\begin{de}\label{de-Furstenberg-selfjoining}
Let $(X,\X,\mu, T)$ be a m.p.s. For $d\ge 1$ let $\mu^{(d)}_{\vec{a}}$ be the measure on $(X^d,\X^d)$ defined by
$$\int_{X^d} f_1(x_1)\cdots f_d(x_d) d\mu^{(d)}_{\vec{a}}(x_1, \ldots, x_d)=\lim_{N\to\infty} \frac{1}{N} \sum_{n=0}^{N-1}\int_X \prod_{j=1}^d f_j(T^{a_jn}x) d\mu(x)$$
for $f_i\in L^{\infty}(X,\mu)$, $1\le j\le d$, where the limits exists by \cite[Theorem 1.1]{HK05}.

We call $\mu^{(d)}_{\vec{a}}$ the {\em Furstenberg self-joining}. Clearly, it is invariant under $\tau_{\vec{a}}$ and $T^{(d)}$.
\end{de}

When in addition we assume that $X$ is a compact metric space and $T$ is a homeomorphism, it is easy to see that
\begin{equation*}
\frac {1}{N} \sum_{n=0}^{N-1} (\tau_{\vec{a}})_*^n \mu_\D^d\longrightarrow
\mu^{(d)}_{\vec{a}},\ N\to \infty, \quad \text{weak$^*$ in ${\mathcal M}(X^d)$},
\end{equation*}
where $\mu_\D^d$ is the diagonal measure on $X^d$ as defined in
\cite{F77}, i.e. it is defined on $X^d$ as follows
\begin{equation*}
    \int_{X^d} f_1(x_1)\cdots f_d(x_d)\ d \mu_\D^d(x_1,\ldots,x_d)=
    \int_X f_1(x)\cdots f_d(x)\ d\mu(x),
\end{equation*}
where $f_1,\ldots, f_d\in C(X)$.

\subsubsection{The ergodic decomposition of $\mu_{\vec{a}}^{(d)}$}

\begin{thm}\cite[Theorem D.]{HSY19}\label{thm-HSY}
Let $(X,\X,\mu,T)$ be an ergodic system and $d\in \N$. Let $a_1,\ldots, a_d$ be distinct non-zero integers. Then there exists a family $\{\mu^{(d)}_{\vec{a},x}\}_{x\in X}$ of Borel probability measures on $(X^d,\X^d)$ such that
\begin{enumerate}
  \item the disintegrations of $\mu_{\vec{a}}^{(d)}$ over $\mu$ is $\displaystyle \mu_{\vec{a}} ^{(d)}=\int_X \mu_{\vec{a}, x}^{(d)} d\mu(x),$

  \item for $\mu$ a.e. $x\in X$, $\mu^{(d)}_{\vec{a},x}$ is ergodic under $\tau_{\vec{a}}=T^{a_1}\times T^{a_2}\times \cdots \times T^{a_d}$,

  \item for all $f_1,\ldots, f_d\in L^\infty(X, \mu)$,
  \begin{equation*}
\begin{split}
   \frac{1}{N} \sum_{n=0}^{N-1}\prod_{i=1}^df_i(T^{a_in}x) \overset{L^2(\mu)}\longrightarrow \int_{X^d} \prod_{i=1}^df_i(x_i)\
   d \mu^{(d)}_{\vec{a},x}(x_1,\ldots,x_{d}),\ N\to \infty.
\end{split}
\end{equation*}


\end{enumerate}
\end{thm}

\subsection{Host-Kra seminorms}\
\medskip

When $f_i$, $i\in I$, are functions on the set $X$, we define a function
$\bigotimes_{i\in I} f_i$ on $X^I$ by
\begin{equation}\label{de-otimes}
    \bigotimes_{i\in I} f_i ({\bf x})=\prod_{i\in I}
    f_i(x_i),
\end{equation}
where ${\bf x}=(x_i)_{i\in I}\in X^I$ and $I$ is an index set.

\subsubsection{The measure $\mu^{[k]}$}
Let $(X,\X, \mu,T)$ be an ergodic system and $k\in \N$. We define a probability measure $\mu^{[k]}$ on $X^{2^k}$ invariant under $T^{[k]}=T\times T\times \cdots \times T$ ($2^k$ times), by
\begin{equation*}
   \mu^{[1]}=\mu\times_{\I(T)}\mu=\mu\times \mu;
\end{equation*}
and for $k\ge 1$,
\begin{equation*}
 \mu^{[k+1]}=\mu^{[k]}\times_{\I ( T^{[k]} ) } \mu^{[k]}.
\end{equation*}

\subsubsection{Host-Kra seminorms}
Write ${\bf x}=(x_0,x_1,\ldots,x_{2^k-1})$ for a point of $X^{2^k}$, we define a seminorm $\HK \cdot \HK_k$ on $L^\infty(\mu)$ by
\begin{equation}\label{}
    \HK f\HK_k=\Big( \int_{X^{2^k}} \bigotimes_{i\in \{0,1,\ldots, 2^k-1\}}
f( {\bf x})d\mu^{[k]}({\bf x})\Big)^{1/2^k}=\Big( \int_{X^{2^k}} \prod_{i=0}^{2^k-1}
f(x_i)d\mu^{[k]}({\bf x})\Big)^{1/2^k}.
\end{equation}
That $\HK \cdot \HK_k$  is a seminorm \footnote{Here for simplicity we give the formula for real functions, and one can give the formula for complex functions similarly.} can be proved as in \cite[Lemma 3.2]{HK05}, and we call it
{\em Host-Kra seminorm} (HK seminorm for short).
As $X$ is assumed to be ergodic, the $\sigma$-algebra $\I^{[0]}$ is
trivial and $\mu^{[1]}=\mu \times \mu$. We therefore have
$$\HK f\HK_1=\Big(\int_{X^2}f(x_0){f(x_1)}d(\mu \times \mu)(x_0,x_1)\Big)^{1/2}
=\Big|\int fd\mu\Big|.$$


\subsubsection{}
It was shown in the proof of \cite[Theorem 12.1]{HK05} that
\begin{lem}\label{lemmaE1}
For every integer $k\ge 0$ and every $f\in L^\infty(\mu)$, one has
\begin{equation}\label{lem2.1}
    \HK f\HK_{k+1}=\Big(\lim_{N\to \infty} \frac{1}{N}\sum_{n=0}^{N-1}
    \HK f\cdot T^n {f}\HK_k^{2^k}\Big)^{1/2^{k+1}}.
\end{equation}
\end{lem}
Note that (\ref{lem2.1}) can be considered as an alternate definition of the seminorms. For each $f\in L^\infty(X,\mu)$, one has that \cite[Lemma 3.9]{HK05}
\begin{equation}\label{a6}
  \HK f\HK_k\le \HK f\HK_{k+1}, \ \forall k\in \N.
\end{equation}



\begin{lem}\label{lem-AP-vdc}\cite[Lemma 3.2]{HSY19}
Let $(X,\X,\mu,T)$ be an ergodic m.p.s. and $d\ge 1$ be an integer.
Suppose that $\lambda$ is a $d$-fold self-joining of $(X,\X,\mu,T)$.
Assume that $f_j\in L^\infty(X,\mu)$ with
$\|f_j\|_\infty\le 1$ for $j=1,\ldots,d$. Then
\begin{equation}\label{AP-VDC}
\limsup_{N\to\infty}\Big\|
\frac{1}{N}\sum_{n=0}^{N-1}f_1(T^nx_1)f_2(T^{2n}x_2)\cdots
f_d(T^{dn}x_d) \Big\|_{L^2(X^d, \lambda)}\le \min_{1\le l\le d}\{l\cdot
\interleave f_l\interleave_d \}
\end{equation}
\end{lem}

In particular, let $\lambda$ be the diagonal joining, and we have

\begin{lem}\label{lem-AP-HK}\cite[Theorem 12.1]{HK05}
Let $(X,\X,\mu,T)$ be an ergodic m.p.s. and $d\ge 1$ be an integer.
Assume that $f_j\in L^\infty(X,\mu)$ with
$\|f_j\|_\infty\le 1$ for $j=1,\ldots,d$. Then
\begin{equation}
\limsup_{N\to\infty}\Big\|
\frac{1}{N}\sum_{n=0}^{N-1}f_1(T^nx)f_2(T^{2n}x)\cdots
f_d(T^{dn}x) \Big\|_{L^2(X, \mu)}\le \min_{1\le l\le d}\{l\cdot
\interleave f_l\interleave_d \}
\end{equation}
\end{lem}

\subsection{Pro-nilsystems}

\subsubsection{Nilpotent groups} Let $G$ be a group. For $g, h\in G$, we write $[g, h] =
ghg^{-1}h^{-1}$ for the commutator of $g$ and $h$ and we write
$[A,B]$ for the subgroup spanned by $\{[a, b] : a \in A, b\in B\}$.
The commutator subgroups $G_j$, $j\ge 1$, are defined inductively by
setting $G_1 = G$ and $G_{j+1} = [G_j ,G]$. Let $k \ge 1$ be an
integer. We say that $G$ is {\em $k$-step nilpotent} if $G_{k+1}$ is
the trivial subgroup.

\subsubsection{Nilsystems}
Let $G$ be a $k$-step nilpotent Lie group and $\Gamma$ a discrete
cocompact subgroup of $G$. The compact manifold $X = G/\Gamma$ is
called a {\em $k$-step nilmanifold}. The group $G$ acts on $X$ by
left translations and we write this action as $(g, x)\mapsto gx$.
The Haar measure $\nu$ of $X$ is the unique probability measure on
$X$ invariant under this action. Let $t\in G$ and $T$ be the
transformation $x\mapsto t x$ of $X$. Then $(X, \nu, T)$ is
called a {\em $k$-step nilsystem}.

\begin{de}
Let $k\in \N$. A m.p.s. $(X,\X,\mu,T)$ is called a {\em $k$-step pro-nilsystem} or a {\em system of order $k$} if it is a (measure theoretic) inverse limit of $k$-step nilsystems.
\end{de}

\subsubsection{}
Here are some basic properties of nilsystems, and they also hold for pro-nilsystems.

\begin{thm}\cite{P, Leibman}\label{thm-ParryLeibman}
Let $(X = G/\Gamma, \nu , T )$ be a $k$-step nilsystem with $T$ the
translation by the element $t\in G$. Then:

\begin{enumerate}
\item $(X, T )$ is uniquely ergodic if and only if $(X, \nu , T )$ is
ergodic if and only if $(X, T )$ is minimal if and only if $(X, T )$
is transitive.

\item Let $Y$ be the closed orbit of some point $x\in X$. Then $Y$ can
be given the structure of a nilmanifold, $Y = H/\Lambda$, where $H$
is a closed subgroup of $G$ containing $t$, and $\Lambda$ is a closed
cocompact subgroup of $H$.

\end{enumerate}
\end{thm}

\begin{rem}
One can generalize the above results to the action of several translations. For example,
let $X= G/\Gamma$ be a nilmanifold with Haar measure $\nu$ and let
$t_1,\ldots , t_k$ be commuting elements of $G$. If the group
spanned by the translations $t_1, \ldots , t_k$ acts ergodically on
$(X,\nu)$, then $X$ is uniquely ergodic for this group, see \cite{Leibman} for details.
\end{rem}

\subsubsection{Nilfactors}
By the Host-Kra seminorm we can define factors  $(Z_{d-1},\mathcal{Z}_{d-1},\mu_{d-1},T)$.
\begin{de}
Let $(X,\X,\mu , T)$ be an ergodic m.p.s.
For $d\in \N$, there exists a $T$-invariant $\sigma$-algebra $\ZZ_{d-1}$  of $\X$ such that for $f\in L^\infty(\mu)$,
$$\HK f\HK_d=0 \ \text{ if and only if}\  \E(f|\ZZ_{d-1})=0.$$
Let $(Z_{d-1},\mathcal{Z}_{d-1},\mu_{d-1},T)$ be the factor of $X$ associated to the sub $\sigma$-algebra $\ZZ_{d-1}$.

\end{de}

By \eqref{a6}, the factors $Z_k, k\ge 1,$ form an increasing sequence of factors of $X$.
In \cite{HK05}, it was showed that $(Z_{d-1}, \ZZ_{d-1}, \mu_{d-1}, T)$ has a very nice structure.

\begin{thm}\cite[Theorem 10.1]{HK05}
Let $(X,\X,\mu , T)$ be an ergodic m.p.s. and $d\in \N$. Then the
system $(Z_{d-1}, \ZZ_{d-1}, \mu_{d-1}, T)$ is a $(d-1)$-step pro-nilsystem.
\end{thm}

\subsubsection{$\infty$-step pro-nilsystems}
An $\infty$-step pro-nilsystem was introduced and studied in \cite{DDMSY}.
\begin{de}
Let $(X,\X,\mu,T)$ be an ergodic m.p.s. The factors $\ZZ_k, k\in \N$, form an increasing sequence of $T$-invariant sub-$\sigma$-algebras of $\X$, and we let $\ZZ_\infty\triangleq \bigvee_{k\in \N} \ZZ_k$ and $(Z_\infty, \ZZ_\infty, \mu_\infty, T)$ be the factor associated with the $\ZZ_\infty$. Then this system is the inverse limit of the systems $(Z_k, \ZZ_k,\mu_k,T), k\in \N$. We say $(Z_\infty, \ZZ_\infty, \mu_\infty, T)$ is an {\em $\infty$-step pro-nilsystem}.
\end{de}

Let $(X,\X,\mu,T)$ be an ergodic m.p.s. and $f\in L^\infty(X,\mu)$.
We define $\HK f\HK_\infty \triangleq \sup_{k\in \N} \HK f\HK_k$. Note that
$\HK f\HK_\infty=0$ if and only if $\E(f|\ZZ_\infty)=0$.

\subsection{Topological pro-nilsystems}\

\subsubsection{}
One also has the topological version of this notion, i.e. the topological inverse limit of nilsystems. First recall the
definition of an inverse limit of t.d.s. If
$(X_i,T_i)_{i\in \N}$ are t.d.s. with $diam(X_i)\le 1$ and
$\pi_i: X_{i+1}\rightarrow X_i$ are factor maps, the {\em topological inverse
limit} of the systems is defined to be the compact subset of
$\prod_{i\in \N}X_i$ given by $\{ (x_i)_{i\in \N }: \pi_i(x_{i+1}) =
x_i, i\in \N\}$, and we denote it by
$\lim\limits_{\longleftarrow}(X_i,T_i)_{i\in\N}$. It is a compact
metric space endowed with the distance $\rho((x_{i})_{i\in\N}, (y_{i})_{i\in
\N}) = \sum_{i\in \N} 1/2^i \rho_i(x_i, y_i )$, where $\rho_{i}$ is the metric in
$X_{i}$. We note that the
maps $T_i$ induce naturally a transformation $T$ on the inverse
limit.

\begin{de}\cite{HKM}
Let $d\in \N$. A topological inverse limit of $d$-step minimal
nilsystems is called a {\em topological $d$-step pro-nilsystem} or a {\em topological system of order $d$}.
\end{de}

See \cite{HKM} and \cite{SY} for more about topological pro-nilsystems.

\subsubsection{}
A minimal t.d.s. is a {\em topological  $\infty$-step pro-nilsystem} if and only if it is
a topological inverse limit of minimal nilsystems \cite{DDMSY}.

By Theorem \ref{thm-ParryLeibman}, a topological $d$-step pro-nilsystem is uniquely ergodic for each $d\in \N\cup\{\infty\}$.

\begin{thm}\cite[Subsection 5.1]{HKM}\label{thm-HKM}
Let $d\in \N\cup\{\infty\}$. Any $d$-step pro-nilsystem is isomorphic in the measure theoretic sense to a topological $d$-step pro-nilsystem.
\end{thm}

\begin{lem}\cite[Lemma A.3]{DDMSY}\label{DDMSY}
Let $d\in \N\cup\{\infty\}$. Let $(X,T)$ be a topological $d$-step nilsystem,
then the maximal measurable and topological $j$-step pro-nilfactors of $(X,T)$
coincide, where $j\leq d$.
\end{lem}

\subsection{Nilsequencs}

\begin{de}\cite{BHK05}
Let $d\ge 1$ be an integer and let $X = G/\Gamma$ be a
$d$-step nilmanifold. Let $\phi$ be a continuous real (or complex)
valued function on $X$ and let $t \in G$ and $x \in X$. The sequence
$\{\phi(t^n\cdot x)\}_{n\in \mathbb{Z}}$ is called a {\em  basic $d$-step nilsequence}.
A {\em $d$-step nilsequence} is a uniform limit of basic $d$-step
nilsequences.
\end{de}

We will need the following theorem.

\begin{thm}\cite[Theorem 16.10]{HK18}\label{thm-CFH}
Let $(X,\X,\mu,T)$ be a m.p.s. (not necessarily ergodic) and let $k\in \N$. Then for all $p\in [1,\infty)$ and $\ep>0$, every $f\in L^\infty(X,\mu)$ admits a decomposition
$$f=f_{\rm unif}+f_{\rm nil}+f_{\rm sml},$$
where
\begin{enumerate}
  \item the function $f_{\rm unif}$ satisfies $\HK f_{\rm unif}\HK_{k+1}=0$;
  \item for $\mu$-almost every $x\in X$, the sequence $\big(f_{\rm nil}(T^n x)\big)_{n\in \Z}$ is a $k$-step nilsequence;
  \item the function $f_{\rm sml}\in L^\infty(X,\mu)$ satisfies $\|f_{\rm sml}\|_{L^p(X,\mu)}\le \ep$.
\end{enumerate}
Furthermore, $\|f_{\rm nil}\|_{L^\infty(X,\mu)}\le \|f\|_{L^\infty(X,\mu)}$ and $\|f_{\rm nil}+f_{\rm sml}\|_{L^\infty(X,\mu)}\le \|f\|_{L^\infty(X,\mu)}$.
\end{thm}

\begin{rem}\label{rem-CFH}
\begin{enumerate}
  \item Let $(X,\X,\mu,T)$ be a m.p.s. and let $k\in \N$. As a corollary of Theorem \ref{thm-CFH},
 for all $\ep>0$, every $f\in L^\infty(X,\mu)$ admits a decomposition
$$f=f_{\rm unif}+f_{\rm nil},$$
such that for $\mu$-almost every $x\in X$, the sequence $\big(f_{\rm nil}(T^n x)\big)_{n\in \Z}$ is a $k$-step nilsequence and the function $f_{\rm unif}\in L^\infty(X,\,u)$ satisfies $\HK f_{\rm unif}\HK_{k+1}<\ep$ (\cite[Theorem 16.11]{HK18}).

  \item Let $(X,\X,\mu,T)$ be a m.p.s. and let $k\in \N$. If $f\in L^\infty(\ZZ_k,\mu)$, then for every $\ep>0$ there is a function $\widetilde{f}\in L^\infty(X,\mu)$ such that \begin{itemize}
                  \item $\widetilde{f}\in L^\infty(\ZZ_k, \mu)$, $\|\widetilde{f}\|_{L^\infty(\mu)}\le \|f\|_{L^\infty(X,\mu)}$ and $\|f-\widetilde{f}\|_{L^1(\mu)}\le \ep$,
                  \item for $\mu$-almost every $x\in X$, the sequence $\big(\widetilde{f}(T^n x)\big)_{n\in \Z}$ is a $k$-step nilsequence.
                \end{itemize}
  (see \cite[Proposition 3.1]{CFH11})
\end{enumerate}
\end{rem}

\subsection{Characteristic factors}\
\medskip

In the study of multiple ergodic averages, the idea of characteristic factors plays a very important role. This idea was suggested by Furstenberg in \cite{F77}, and the notion of ``characteristic factors'' was
first introduced in a paper by Furstenberg and Weiss \cite{FW96}.

\begin{de}
Let $(X,\X,\mu, T)$ be a m.p.s. and $(Y,\Y,\mu, T)$ be a factor of $X$. Let $\{p_1,\cdots, p_d\}$ be a family of integral polynomials, $d\in \N$. We say that $Y$ is a {\em $L^2$(resp. a.e.)-characteristic factor} of $X$ for the scheme $\{p_1,\ldots, p_d\}$ if for all $f_1,\ldots,f_d\in L^\infty(X,\X,\mu)$,
\begin{equation*}
\begin{split}
  \frac{1}{N}\sum_{n=0}^{N-1} & f_1(T^{p_1(n)}x) f_2(T^{p_2(n)}x)\cdots f_d(T^{p_d(n)}x)\\
  -  &\frac{1}{N}\sum_{n=0}^{N-1} \E(f_1|\Y)(T^{p_1(n)}x) \E(f_2|\Y)(T^{p_2(n)}x)\cdots \E(f_d|\Y)(T^{p_d(n)}x)\to 0
\end{split}
\end{equation*}
in $L^2(X,\X,\mu)$ (resp. almost everywhere).
\end{de}

In \cite{HK05} the authors showed that the factor $Z_{d-1}$ is characteristic for the scheme $\{n,2n,\cdots, dn\}$, and
hence proved that the averages $ \frac 1 N\sum_{n=0}^{N-1}f_1(T^{n}x)f_2(T^{2n}x)\cdots f_d(T^{dn}x)$ converge in $L^2(X,\mu)$.

\subsection{Furstenberg's Theorem and Bourgain's Theorem}

\subsubsection{}
Let $(X,\X,\mu, T)$ be a m.p.s. Let
$${\mathcal H}_{\rm rat}=\overline{\{f\in L^2(X,\mu): \exists a \in \N \ s.t.\ T^a f=f\}},$$
and let $\X_{\rm rat}\subseteq \X$ such that
${\mathcal H}_{\rm rat}=L^2(X,\X_{\rm rat},\mu).$

\begin{thm}[Furstenberg]\cite[Lemma 3.14]{F}\label{Furstenberg}
Let $(X,\X,\mu, T)$ be a m.p.s. and let $q(n)$ be a non-constant integral polynomial. Then for all $f\in L^2(X, \mu)$,
\begin{equation*}
\frac{1}{N}\sum_{n=0}^{N-1}T^{q(n)}f- \frac{1}{N}\sum_{n=0}^{N-1}T^{q(n)} \E(f|\X_{\rm rat})\longrightarrow 0, \ N\to\infty ,
\end{equation*}
in $L^2(X,\mu)$.
In particular, if $(X,\X,\mu,T)$ is totally ergodic \footnote{A m.p.s. $(X,\X,\mu,T)$ is totally ergodic if $(X,\X,\mu,T^k)$ is ergodic for all $k\in \N$}, then
\begin{equation*}
\frac{1}{N}\sum_{n=0}^{N-1}T^{q(n)}f \longrightarrow \int_X f d\mu, \ N\to\infty ,
\end{equation*}
in $L^2(X,\mu)$.
\end{thm}

\begin{thm}[Bourgain]\cite[Theorem 1]{B89}\label{Bouragain}
Let $(X,\X,\mu, T)$ be a m.p.s. and let $q(n)$ be a non-constant integral polynomial and $p>1$. Then there exists a constant $C(p,q)>0$ such that for all $f\in L^p(X, \mu)$,
\begin{equation*}
\Big \| \sup_{N\in \N} \big| \frac{1}{N}\sum_{n=0}^{N-1}T^{q(n)}f\big|\Big \|_{L^p(\mu)}\le C(p,q) \|f\|_{L^p(\mu)}.
\end{equation*}
Moreover, $\displaystyle \frac{1}{N}\sum_{n=0}^{N-1}T^{q(n)}f$ converges almost surely, as $N\to\infty$. \end{thm}

\begin{rem}
By Furstenberg's theorem and Bourgain's theorem, if
$(X,\X,\mu,T)$ is totally ergodic (in particular, when $T$ is weakly mixing), then for all $f\in L^2(X,\mu)$
\begin{equation*}
\frac{1}{N}\sum_{n=0}^{N-1}T^{q(n)}f \longrightarrow \int_X f d\mu, \ N\to\infty,
\end{equation*}
converges almost surely.
\end{rem}

\subsubsection{}
By Theorem \ref{Bouragain} one can show that:

\begin{prop}\cite[Corollary 2.2]{DL1996}\label{prop-DL}
Let $T_1,T_2, \cdots, T_d$ be invertible measure preserving transformations acting on a probability space $(X,\X,\mu)$, and let
$p_1, \ldots, p_d$ be integral polynomials.
Let $q_1,\ldots, q_d \in \R$ be positive real numbers such that $\frac{1}{q}=\sum_{i=1}^d \frac{1}{q_i}<1$. Then all the $d$-tuples $(f_1,\ldots, f_d)$ for which the limit of the averages
\begin{equation}\label{MEA-pol}
  \frac{1}{N}\sum_{n=0}^{N-1} f_1(T_1^{p_1(n)}x)f_2(T_2^{p_2(n)}x)\cdots f_d(T_2^{p_d(n)}x)
\end{equation}
exits a.e. is closed in $L^{q_1}(X,\mu)\times \cdots \times L^{q_d}(X,\mu)$.
\end{prop}

This proposition says that if $\{(f_1^{(k)}, \ldots, f_d^{(k)})\}_{k\in \N}\subseteq \prod_{i=1}^dL^{q_i}(X,\mu)$ and the limit of the averages $(\ref{MEA-pol})$ exists a.e. for each $(f_1^{(k)}, \ldots, f_d^{(k)}), k\in\N$, then for any limit point $(f_1,\ldots,f_d)$ of $\{(f_1^{(k)}, \ldots, f_d^{(k)})\}_{k\in \N}$ in $\prod_{i=1}^dL^{q_i}(X,\mu)$, $\displaystyle \frac{1}{N} \sum_{n=0}^{N-1} f_1(T_1^{p_1(n)}x)\cdots f_d(T_d^{p_d(n)}x)$ exists a.e. Hence by this result, to study $(\ref{MEA-pol})$ in $\prod_{i=1}^dL^{q_i}(X,\mu)$, one may assume that $f_1,\ldots,f_d\in L^\infty(X,\mu)$ or other dense subsets.

\subsection{Multiple polynomial ergodic theorems}\
\medskip

The following theorem was proven under some supplementary assumption about the polynomials in \cite{HK051} and
was proven in its general form in \cite{Leibman05-Isr}. See \cite{Walsh} for a more general result.

\begin{thm}[Multiple Polynomial Ergodic Theorem]\label{thm-Leibman-main}
Let $(X,\X,\mu, T)$ be an ergodic m.p.s. and $d\in \N$. Let $p_1,p_2,\ldots, p_d$ be integral polynomials. Then for any F{\o}lner sequence $\{\Phi_N\}_{N=1}^\infty$ in $\Z$ and for any $f_1,f_2,\ldots, f_d\in L^\infty(X)$, the averages
$$\frac{1}{|\Phi_N|}\sum_{{n}\in \Phi_N} T^{p_1({n})}f_1\cdots T^{p_d({ n})}f_d$$
converge in $L^2(X,\mu)$ as $N\to\infty$.
\end{thm}

The final result we state here is the following

\begin{thm}\cite[Theorem A$_0$]{BL96}\label{thm-mul-rec}
Let $(X,\X,\mu)$ be a probability space, let $T_1,\ldots, T_d$ be commuting measure preserving transformations of $X$, let $p_1,p_2,\cdots, p_d$ be integral polynomials with $p_1(0)=\cdots =p_d(0)=0$, and let $A\in X$ with $\mu(A)>0$. Then
$$\liminf_{N\to\infty} \frac{1}{N}\sum_{n=0}^{N-1} \mu(T_1^{-p_1(n)}A\cap T_2^{-p_2(n)}A\cap \cdots \cap T_d^{-p_d(n)}A)>0.$$
\end{thm}

\section{Polynomial Furstenberg joinings}\label{section-furstenberg-joining}

In this section we introduce polynomial Furstenberg joinings and give some basic properties.

\subsection{Standard measures on product spaces }\
\medskip

In this subsection, we explain why we need to assume that $X$ is a compact metric space in a m.p.s. $(X,\X,\mu,T)$.
For more details see \cite{F77} and \cite{F}.

\subsubsection{}
Recall that two probability spaces $(X,\X,\mu)$, $(X', \X',\mu')$ are isomorphic if there exist null sets $N\subseteq X, \widetilde{N}\subseteq \widetilde{X}$ and an one to one measurable, measure preserving map $\phi: X\setminus N\rightarrow \widetilde{X}\setminus \widetilde{N}$. We say that $(X,\X,\mu)$ is a {\em regular measure space}\footnote{Here we use definitions from \cite{F77, F}. This notion consists with the  Lebesgue space used in the previous sections.} if it is isomorphic to a  measure space $(\widetilde{X},\widetilde{\X},\widetilde{\mu})$, where $\widetilde{X}$ is a compact metric space, $\widetilde{\X}$ the $\sigma$-algebra of Borel sets, and $\widetilde{\mu}$ a regular Borel measure (in this case, we say $(\widetilde{X},\widetilde{\X},\widetilde{\mu})$ is a {\em topological model} of $(X,\X,\mu)$). The advantage in dealing with compact metric spaces is that in this situation measures are determined by positive linear functionals on the algebra of continuous functions $C(X)$. If $X$ is compact metric, this algebra is separable and the functional is determined by its values on a dense countable set.

A m.p.s. $(X,\X,\mu, T)$ is called {\em regular} if $(X,\X,\mu)$ is a regular measure space. A regular m.p.s, always has a compact metric model $(\widetilde{X},\widetilde{\X},\widetilde{\mu},\widetilde{T})$, where $\widetilde{T}$ is a homeomorphism of the compact metric space $\widetilde{X}$ \cite[Theorem 5.1.5]{F}.

\subsubsection{}
Let $(X_i,\X_i,\mu_i)$ be regular probability spaces, $i\in I$ and $I$ be a finite or countable index set. Let $ \Omega=\prod_{i\in I} X_i$ with the product $\sigma$-algebra $ \B=\prod_{i\in I} \X_i$, i.e. the least $\sigma$-algebra with respect to which the projections
$\Omega\rightarrow X_i$ are measurable for all $i\in I$.

A measure $\lambda$ on $(\Omega,\B)$ is called {\em standard} if its image in each $X_i$ is $\mu_i$.\footnote{Let  $(X_i,\X_i,\mu_i, T_i)$ be m.p.s., where $i\in I$ and $I$ is an index set. Then a Borel probability measure $\lambda$ on $\prod_{i\in I} X_i$ is a joining of $\{(X_i,\X_i,\mu_i, T_i)\}_{i\in I}$ if it is a $\prod _{i\in I} T_i$-invariant standard measure on $\prod_{i\in I} X_i$.} For $i_1,\ldots, i_n\in I$ and $f_{i_j}: X_{i_j}\rightarrow \C, 1\le j\le n$, let $\displaystyle f_{i_1}\otimes f_{i_2}\otimes \cdots \otimes f_{i_n}: \Omega=\prod_{i\in I} X_i\rightarrow \C$ be  defined as follows:
\begin{equation}\label{de-otimes-2}
  f_{i_1}\otimes f_{i_2}\otimes \cdots \otimes f_{i_n}\big( (x_i)_{i\in I}\big)= f_{i_1}(x_{i_1})f_{i_2}(x_{i_2})\cdots f_{i_n}(x_{i_n}).
\end{equation}
It is easy to see that if $f_{i_j}\in L^\infty(X_{i_j},\mu_{i_j}), 1\le j\le n$, then $\displaystyle f_{i_1}\otimes f_{i_2}\otimes \cdots \otimes f_{i_n}\in L^\infty (\Omega, \lambda)$ since $\lambda$ is standard.


Note that if each $(X_i,\X_i,\mu_i)$ is regular, and $\lambda$ on $\Omega$ is standard, then $(\Omega,\B,\lambda)$ is also regular. Let $\{(\widetilde{X}_i, \widetilde{\X}_i, \widetilde{\mu}_i)\}_{i\in I}$ be compact metric models of $\{(X_i,\X_i,\mu_i)\}_{i\in I}$. Let $\{O_i\}_{i\in I}$ and $\{\widetilde{O}_i\}_{i\in I}$ be null sets in each $X_i$ and $\widetilde{X}_i$ respectively such that $\phi_i: X_i\setminus O_i\rightarrow \widetilde{X}_i\setminus \widetilde{O}_i$ is isomorphic. Since $\mu_i(O_i)=0$, we have
$$\lambda \big(\prod_{i\in I}(X_i\setminus O_i)\big)=1.$$
Thus $\lambda$ may be carried over to a unique regular Borel measure in $\prod_{i\in I} \widetilde{X}_i $.



\subsubsection{}
Now assume that each $X_i$ is a compact metric space and $\Omega=\prod_{i\in I} X_i$. Let
$$\F =\Big\{\sum_{t=1}^k f^{(t)}_{i^{(t)}_1}\otimes f^{(t)}_{i^{(t)}_2}\otimes \cdots \otimes f^{(t)}_{i^{(t)}_{n_t}}: k\in \N,  i^{(t)}_j\in I, f^{(t)}_{i^{(t)}_j}\in C( X_{i^{(t)}_j}), 1\le j\le n_t , 1\le t\le k\Big\}.$$

$\F$ is a subalgebra of $C(\Omega)$ and separates points. By the Stone-Weierstrass theorem, $\F$ is dense in $C(\Omega)$. Thus each positive linear functional on $C(\Omega)$ is determined uniquely by its value on $\F$. In particular, if we want to define a measure $\lambda$ on $\Omega$, we only need to check $\lambda (f_{i_1}\otimes f_{i_2}\otimes \cdots \otimes f_{i_n})$ for all $n\in \N$, and for all $f_{i_1}\in C(X_{i_1}),\ldots, f_{i_n}\in C(X_{i_n})$.

\medskip

{\bf In the sequel we always assume that $(X,\X,\mu, T)$ is regular, i.e. $X$ is a compact metric space.}

\subsection{A self-joining $\mu_\A^{(\infty)}$}

\subsubsection{Some notations}
Let $(X,\X,\mu, T)$ be a m.p.s. and $d\in \N$. Let $\A=\{p_1, p_2,\ldots, p_d\}$ be a family of non-constant integral polynomials.
We will define a joining $\mu^{(\infty)}_\A$ on $(X^d)^\Z$.

\medskip

The point of $(X^d)^{\Z}$ is denoted by
$${\bf x}=({\bf x}_n)_{n\in \Z}=(\ldots, {\bf x}_{-1}, \underset{\bullet}{\bf x_0},{\bf x}_1,\ldots )
=\Big((x^{(1)}_n, x^{(2)}_n,\ldots, x^{(d)}_n) \Big)_{n\in \Z},$$
where $`` \ \underset{\bullet } \ "$ means the $0$-th coordinate.
Let $\vec{p}=(p_1,p_2,\ldots, p_d)$ and let $T^{\vec{p}(n)}: X^d\rightarrow X^d$ be defined by
\begin{equation}\label{}
  T^{\vec{p}(n)}\Big((x_1, x_2, \ldots, x_d)\Big)=(T^{p_1(n)}x_1, T^{p_2(n)}x_2,\ldots, T^{p_d(n)}x_d).
\end{equation}

Define $T^\infty: (X^d)^{\Z}\rightarrow (X^d)^{\Z}$ such that
$T^\infty= \cdots \times T^{(d)}\times T^{(d)} \times \cdots \ (\infty \ \text{times}),$ i.e.
$$T^\infty ({\bf x}_n)_{n\in \Z}=(T^{(d)}{\bf x}_n)_{n\in \Z}.$$
Let $\sigma: (X^d)^{\Z}\rightarrow (X^d)^{\Z}$
be the shift map, i.e.,  for all ${\bf x}=({\bf x}_n)_{n\in \Z}\in (X^d)^{\Z}$,
$$(\sigma {\bf x})_n={\bf x}_{n+1} , \ \forall n\in \Z.$$
Let $x^{\otimes d}=(x,x,\ldots, x)\in X^d$ and
$$\D_{\infty}(X)=\{x^{(\infty)}\triangleq (\ldots, x^{\otimes d}, x^{\otimes d},\ldots ) \in (X^d)^{\Z}: x\in X\}.$$
For each $x\in X$, put
\begin{equation}\label{}
 \begin{split}
  \w_x^\A &\triangleq (T^{\vec{p}(n)}x^{\otimes d})_{n\in \Z}\\
   &=(\cdots, T^{\vec{p}(-1)}(x^{\otimes d}), \underset{\bullet}{T^{\vec{p}(0)}x^{\otimes d}},T^{\vec{p}(1)}(x^{\otimes d}),T^{\vec{p}(2)}(x^{\otimes d}), \cdots)\in (X^d)^{\Z}.
  \end{split}
\end{equation}

\subsubsection{An $\infty$-joining with respect to $\A$} Now we will define an $\infty$-joining with respect to a given m.p.s.
$(X, \X, \mu, T)$ and $\A=\{p_1,\ldots,p_d\}$.
\begin{de}\label{de-infty-measure}
Let $(X, \X, \mu, T)$ be a m.p.s., $d\in\N$ and $\A=\{p_1, p_2,\ldots, p_d\}$ be a set of non-constant integral polynomials. Let $\mu_\A ^{(\infty)}$ be the measure on $(X^d)^{\Z}$ defined by
\begin{equation}\label{f1}
  \int_{(X^d)^{\Z}} F({\bf x}) d\mu_\A^{(\infty)}({\bf x})
    = \lim_{N\rightarrow +\infty} \frac{1}{N}\sum_{n=0}^{N-1} \int_X F( \sigma^n(\w_x^\A)) d\mu(x)
\end{equation}
for all $F\in C((X^d)^{\Z})$. We call $\mu^{(\infty)}_\A$ the {\em $\infty$-joining with respect to $\A$}.

\end{de}
 Note that the limits exist by the remark below and the equality
\begin{equation*}
\begin{split}
     \sigma^n(\w_x^\A)=(T^{\vec{p}(n+j)}x^{\otimes d})_{j\in \Z}=
   (\ldots, T^{\vec{p}(n-1)}x^{\otimes d}, \underset{\bullet} {T^{\vec{p}(n)}x^{\otimes d}},T^{\vec{p}(n+1)}x^{\otimes d}, \ldots).
  \end{split}
\end{equation*}
Clearly, the measure $\mu_\A^{(\infty)}$ is invariant under $T^\infty$ and $\sigma$.

\begin{rem}
Now we explain the existence of \eqref{f1}. Let $f^\otimes$ be elements of $C(X^d)$ with the following forms:
$$f^\otimes\triangleq f^{(1)}\otimes f^{(2)} \otimes \cdots \otimes f^{(d)},$$
where $f^{(1)}, f^{(2)},\cdots, f^{(d)}\in C(X)$.
By the Stone-Weierstrass theorem, the set
$$\F= \left\{\bigotimes_{j=-l}^l f_j^\otimes : l\in \N, f^\otimes _1, f^\otimes _2, \ldots, f^\otimes _l\in C(X^d)\right\}$$
is dense in $C((X^d)^{\Z})$. Note that by \eqref{de-otimes-2}, $\bigotimes_{j=-l}^l f_j^\otimes$ can be regarded as an element of $C((X^d)^\Z)$.

Thus \eqref{f1} holds if and only if for all $l \in \N$ and all $f_j^\otimes = f_j^{(1)}\otimes f_j^{(2)} \otimes \cdots \otimes f_j^{(d)} \in C(X^d), -l \le j\le l$,
\begin{equation}\label{f2}
\begin{split}
  &\int_{(X^d)^{\Z}} \Big( \bigotimes_{j=-l}^l f_j^\otimes \Big)({\bf x})d\mu_\A^{(\infty)}({\bf x})\\
  &=\lim_{N\rightarrow +\infty} \frac{1}{N}\sum_{n=0}^{N-1}\int_X \prod_{j=-l}^l f_j^\otimes (T^{\vec{p}(n+j)}x^{\otimes d}) d\mu(x) \\
  &= \lim_{N\rightarrow +\infty} \frac{1}{N}\sum_{n=0}^{N-1}\int_X \prod_{j=-l}^l f_j^{(1)}(T^{p_1(n+j)}x) f_j^{(2)} (T^{p_2(n+j)}x) \cdots f_j^{(d)}(T^{p_d(n+j)}x) d\mu(x) ,
\end{split}
\end{equation}
where the limits exist by Theorem \ref{thm-Leibman-main}. Thus $\mu^{(\infty)}_\A$ is well defined.

\medskip

Another way to see $\mu_\A^{(\infty)}$ is as follows:
\begin{equation}\label{f3}
\frac {1}{N} \sum_{n=0}^{N-1} (\cdots \times T^{\vec{p}(n-1)}\times  \underset{\bullet}{T^{\vec{p}(n)}}\times T^{\vec{p}(n+1)}\times \cdots)_* \mu_\D^{\infty}\longrightarrow
\mu_\A^{(\infty)},
\end{equation}
as $N\to \infty$ in the weak$^*$ topology of ${\mathcal M}((X^d)^{\Z})$, where $\mu_\D^{\infty}$ is the diagonal measure on $(X^d)^{\Z}$ as defined as follows
\begin{equation*}
    \int_{(X^d)^{\Z}} \Big( \bigotimes_{j=-l}^l f_j^\otimes \Big)({\bf x})\ d \mu_\D^{\infty}({\bf x})=
    \int_X \prod_{j=-l}^l f_j^\otimes (x^{\otimes d}) \ d\mu(x),
\end{equation*}
for all $l \in \N$ and  all $f_j^\otimes = f_j^{(1)}\otimes f_j^{(2)} \otimes \cdots \otimes f_j^{(d)} \in C(X^d), -l \le j\le l$.
\end{rem}

\begin{rem}\label{rem-3.5}
Note that $\mu_\A^{(\infty)}$ is a standard measure. To see this fact, let $\Pi: (X^d)^{\Z}\rightarrow X$ be the projection to some coordinate. Let $\widetilde{\mu}=(\Pi)_*\mu_\A^{(\infty)}$. Then for all $f\in C(X)$,
$$\widetilde{\mu}(f)=\lim_{N\rightarrow +\infty} \frac{1}{N}\sum_{n=0}^{N-1}\int_X f(T^{q(n)}x) d\mu(x) =\int_Xfd\mu =\mu(f),$$
where $q(n)=p(n+m)$ for some $p\in \A, m\in \Z$.
Thus $\widetilde{\mu}=\mu$, and $\mu^{(\infty)}_\A$ is a standard measure.
\end{rem}

\subsubsection{}
Now we show \eqref{f2} still holds if we replace the continuous functions by bounded measurable ones, and the definition of $\mu^{(\infty)}_\A$ is independent of the choice of topological models of $(X,\X,\mu,T)$. That is, we will show

\begin{prop}\label{prop-mesurable}
Let $(X, \X, \mu, T)$ be a m.p.s., and let $\A=\{p_1, \ldots, p_d\}$ be a family of non-constant integral polynomials. Then for all $l \in \N$ and all $f_j^\otimes = f_j^{(1)}\otimes f_j^{(2)} \otimes \cdots \otimes f_j^{(d)}$ (where $ f^{(i)}_j\in L^\infty(X,\mu), 1\le i\le d,  -l \le j\le l$)
\begin{equation}\label{f7}
\begin{split}
   \int_{(X^d)^{\Z}} \Big( \bigotimes_{j=-l}^l f_j^\otimes \Big)({\bf x})d\mu_\A^{(\infty)}({\bf x})=\lim_{N\rightarrow +\infty} \frac{1}{N}\sum_{n=0}^{N-1}\int_X \prod_{j=-l}^l f_j^\otimes (T^{\vec{p}(n+j)}x^{\otimes d}) d\mu(x),
\end{split}
\end{equation}
where $f_j^\otimes (T^{\vec{p}(n+j)}x^{\otimes d})=f_j^{(1)}(T^{p_1(n+j)}x) f_j^{(2)} (T^{p_2(n+j)}x) \cdots f_j^{(d)}(T^{p_d(n+j)}x).$
\end{prop}

See Appendix \ref{appendix3} for a proof.

\begin{prop}\label{prop-iso}
Let $(X_i,\X_i,\mu_i,T_i), i=1,2$, be two isomorphic m.p.s. such that
$$\phi: (X_1,\X_1,\mu_1,T) \rightarrow (X_2,\X_2,\mu_2,T)$$ is an isomorphism. Let $\A=\{p_1, \ldots, p_d\}$ be a set of non-constant integral polynomials. Then
$$\phi^\infty: ((X_1^d)^\Z,(\X_1^d)^\Z, (\mu_1)_\A^{(\infty)}, \langle T_1^\infty, \sigma\rangle)\rightarrow ((X_2^d)^\Z,(\X_2^d)^\Z, (\mu_2)_\A^{(\infty)}, \langle T_2^\infty, \sigma\rangle) $$
is also an isomorphism.
\end{prop}

\begin{proof}
Note that $\phi^\infty $ is defined by $\phi^\infty ({\bf x}_n)_{n\in \Z}=(\phi^{(d)}{\bf x}_n)_{n\in \Z}, \forall ({\bf x}_n)_{n\in \Z}\in (X_1^d)^\Z.$ It is easy to verify that $\phi^\infty\circ T_1^\infty=T_2^\infty \circ \phi^\infty$ and $\sigma\circ \phi^\infty=\phi^\infty \circ \sigma$.

It suffices to show that $\phi^\infty_* (\mu_1)_\A^{(\infty)}=(\mu_2)_\A^{(\infty)}$. By Proposition \ref{prop-mesurable},
for all $l \in \N$ and all $f_j^\otimes = f_j^{(1)}\otimes f_j^{(2)} \otimes \cdots \otimes f_j^{(d)}$ (where $ f^{(i)}_j\in L^\infty(X_2,\mu_2), 1\le i\le d,  -l \le j\le l$)
\begin{equation*}
\begin{split}
   & \int_{(X_2^d)^{\Z}} \Big( \bigotimes_{j=-l}^l f_j^\otimes \Big)({\bf x})d(\mu_2)_\A^{(\infty)}({\bf x})=\lim_{N\rightarrow +\infty} \frac{1}{N}\sum_{n=0}^{N-1}\int_{X_2} \prod_{j=-l}^l f_j^\otimes (T_2^{\vec{p}(n+j)}x^{\otimes d}) d\mu_2(x)\\
   &\quad = \lim_{N\rightarrow +\infty} \frac{1}{N}\sum_{n=0}^{N-1}\int_{X_2} \prod_{j=-l}^l f_j^{(1)}(T_2^{p_1(n+j)}x)  \cdots f_j^{(d)}(T_2^{p_d(n+j)}x) d\mu_2(x)\\
   &\quad = \lim_{N\rightarrow +\infty} \frac{1}{N}\sum_{n=0}^{N-1}\int_{X_2} \prod_{j=-l}^l f_j^{(1)}(T_2^{p_1(n+j)}x)  \cdots f_j^{(d)}(T_2^{p_d(n+j)}x) d(\phi_*\mu_1)(x)\\
   &\quad = \lim_{N\rightarrow +\infty} \frac{1}{N}\sum_{n=0}^{N-1}\int_{X_1} \prod_{j=-l}^l f_j^{(1)}\circ\phi (T_1^{p_1(n+j)}z) \cdots f_j^{(d)}\circ\phi(T_1^{p_d(n+j)}z) d\mu_1(z)\\
   &\quad = \int_{(X_1^d)^{\Z}} \Big( \bigotimes_{j=-l}^l f_j^\otimes \Big)\circ \phi^\infty({\bf z})d(\mu_1)_\A^{(\infty)}({\bf z})= \int_{(X_2^d)^{\Z}} \Big( \bigotimes_{j=-l}^l f_j^\otimes \Big)({\bf x})d\big(\phi^\infty_*(\mu_1)_\A^{(\infty)}\big)({\bf x}).
\end{split}
\end{equation*}
The proof is complete.
\end{proof}


\subsection{Some special cases}\
\medskip

To get a better understanding of $\mu^{(\infty)}_\A$, we explain it by some examples.

\subsubsection{The case when $\A=\{p\}$ with $p(n)=n$}\label{subsection-linear-one}\
\medskip

In this case we show that
\begin{equation}\label{}
  (X^\Z, \X^\Z, \mu^{(\infty)}_\A,\sigma)\cong (X,\X,\mu, T),
\end{equation}
where $\cong$ means that two m.p.s. are isomorphic.

\medskip

Now $p(n)=n$, and we have for $x\in X$
$$\omega_x^\A=(\ldots, T^{-1}x, \underset{\bullet}x, Tx, T^2x,\ldots )=(T^nx)_{n\in \Z}.$$
Note that
$$\sigma \w_x^\A=(\ldots, T^{-1}x, x, \underset{\bullet}{Tx}, T^2x, \ldots )=T^\infty \w_x^\A.$$
For all $l\in \N$, and all $f_{-l}, f_{-l+1}, \ldots, f_l\in L^\infty(X,\mu)$,
\begin{equation*}
\begin{split}
  &\int_{X^{\Z}} f_{-l}(x_{-l})f_{-l+1}(x_{-l+1})\cdots f_l(x_l)d\mu_\A^{(\infty)}(\ldots, x_{-1},x_0,x_1,\ldots)\\
  &\quad =\lim_{N\rightarrow +\infty} \frac{1}{N}\sum_{n=0}^{N-1}\int_X f_{-l}(T^{n-l}x)f_{-l+1}(T^{n-l+1}x)\cdots f_l (T^{n+l}x) d\mu(x) \\
  &\quad = \int_X f_{-l}(T^{-l}x)f_{-l+1}(T^{-l+1}x)\cdots f_l(T^{l}x) d\mu(x) .
\end{split}
\end{equation*}
Equivalently, for all $l\in \N$ and $A_{-l},A_{-l+1},\ldots, A_l\in \X$,
$$\mu_\A^{(\infty)}(A_{-l}\times A_{-l+1}\times \cdots \times A_l)=\mu(T^{l}A_{-l}\cap T^{l-1}A_{-l+1}\cap \cdots \cap T^{-l}A_l).$$
That means
\begin{equation}\label{}
  \mu_\A^{(\infty)}=(\cdots \times T^{-1}\times \id\times T\times T^2\times \cdots )_*\mu^\infty_\Delta.
\end{equation}
Thus $(X^\Z, \X^\Z, \mu_\A^{(\infty)})$ is isomorphic to $(X^\Z,\X^\Z, \mu_\Delta^\infty)$, and hence it is isomorphic to $(X,\X,\mu)$.
That is, if we define $$\phi:(X,\X,\mu)\rightarrow (X^\Z,\X^\Z,\mu^{(\infty)}_\A), \ x\mapsto \omega_x ^\A, \ \forall x\in X$$
then it is an isomorphism between them.
So
$ (X^\Z, \X^\Z, \mu^{(\infty)}_\A,\sigma)\cong (X,\X,\mu, T).$

\medskip

\subsubsection{The case when $\A=\{p_1, p_2, \ldots, p_d\}$ with $p_i(n)=a_in$, $1\le i\le d$}\
\medskip

Let $\A=\{p_1, p_2, \ldots, p_d\}$ with $p_i(n)=a_in, 1\le i\le d$, where $a_1,\ldots,a_d$ are distinct non-zero integers. In this case we show that
\begin{equation}\label{}
  ((X^d)^\Z, (\X^d)^\Z, \mu^{(\infty)}_\A, \sigma)\cong (X^d, \X^d, \mu^{(d)}_{\vec{a}}, \tau_{\vec{a}}),
\end{equation}
where $\tau_{\vec{a}}=T^{a_1}\times \cdots\times T^{a_d}$ and $\mu^{(d)}_{\vec{a}}$ is the Furstenberg joining.

\medskip

As $p_i(n)=a_in, 1\le i\le d$, we have that
$$\omega_x^\A=\Big((T^{a_1n}x, T^{a_2n}x, \ldots, T^{a_dn}x)\Big)_{n\in \Z}\in (X^d)^\Z.$$
Note that
\begin{equation*}
\begin{split}
\sigma \omega_x^\A&= \Big((T^{a_1(n+1)}x, T^{a_2(n+1)}x, \ldots, T^{a_d(n+1)}x)\Big)_{n\in \Z}\\
&=\Big(\tau_{\vec{a}} (T^{a_1n}x, T^{a_2n}x, \ldots, T^{a_dn}x)\Big)_{n\in \Z}=\tau_{\vec{a}}^\infty \omega_x^\A,
\end{split}
\end{equation*}
where $\tau_{\vec{a}}^\infty=\cdots \times \tau_{\vec{a}}\times \tau_{\vec{a}} \times \cdots.$

Now we calculate the $\infty$-joining $\mu_\A^{(\infty)}$.
For all $l \in \N$ and all $f_j^\otimes = f_j^{(1)}\otimes \cdots \otimes f_j^{(d)} \in C(X^d), -l \le j\le l$,
\begin{equation*}
\begin{split}
  &\int_{(X^d)^{\Z}} \Big( \bigotimes_{j=-l}^l f_j^\otimes \Big)({\bf x})d\mu_\A^{(\infty)}({\bf x})\\
  &\quad =\lim_{N\rightarrow +\infty} \frac{1}{N}\sum_{n=0}^{N-1}\int_X \prod_{j=-l}^l f_j^\otimes (T^{a_1(n+j)}x,\ldots, T^{a_d(n+j)}x) d\mu(x) \\
   &\quad  =\lim_{N\rightarrow +\infty} \frac{1}{N}\sum_{n=0}^{N-1}\int_X (T^{a_1}\times T^{a_2}\times \cdots \times T^{a_d})^n \Big(\prod_{j=-l}^l f_j^\otimes (T^{a_1j}x,\ldots, T^{a_dj}x) \Big) d\mu(x)\\
   &\quad = \int_{X^d} \prod_{j=-l}^l f_j^\otimes (T^{a_1j}x_1,\ldots, T^{a_dj}x_d) d\mu^{(d)}_{\vec{a}}(x_1,x_2,\ldots, x_d) \\
  &\quad = \int_{X^d} \prod_{j=-l}^l f_j^{(1)}(T^{a_1j}x_1) f_j^{(2)} (T^{a_2j}x_2) \cdots f_j^{(d)}(T^{a_dj}x_d) d\mu^{(d)}_{\vec{a}}(x_1, x_2, \ldots, x_d) .
\end{split}
\end{equation*}
Equivalently, for all $l\in \N$ and $A_j^\otimes=A_j^{(1)}\times \cdots \times  A_j^{(d)}\in \X^d, -l \le j\le l$,
\begin{equation}\label{p1}
  \mu^{(\infty)}_\A(A_{-l}^\otimes \times A_{-l+1}^\otimes \times \cdots \times A_l^\otimes)=\mu^{(d)}_{\vec{a}}\big(\tau_{\vec{a}}^{l}A^\otimes _{-l}\cap \tau_{\vec{a}}^{l-1} A_{-l+1}^{\otimes }\cap \cdots \cap \tau_{\vec{a}}^{-l} A^\otimes_l \big).
\end{equation}

Similar to the definition of $\mu_\Delta^\infty$, we define $(\mu^{(d)}_{\vec{a}})^\infty_\Delta$ on $(X^d)^\Z$ as follows:
$$\int_{(X^d)^\Z}{\bf f}_{-l}({\bf x}_{-l})\cdots {\bf f}_l({\bf x}_l)d(\mu^{(d)}_{\vec{a}})^\infty_\Delta(\ldots, {\bf x}_{-1}, {\bf x}_0, {\bf x}_1, \ldots) = \int_{X^d} {\bf f}_{-l}({\bf x})\cdots {\bf f}_l({\bf x}) d\mu^{(d)}_{\vec{a}}({\bf x}), $$
for all $l\in \N$ and all ${\bf f}_{-l}, \ldots, {\bf f}_l\in C(X^d)$.
It is clear that $((X^d)^\Z, (\X^d)^\Z, (\mu^{(d)}_{\vec{a}})^\infty_\Delta)$ is isomorphic to $(X^d, \X^d, \mu^{(d)}_{\vec{a}})$.

Thus \eqref{p1} means that
$$\mu^{(\infty)}_\A=(\cdots \times \tau_{\vec{a}}^{-1} \times \id_{X^d}\times \tau_{\vec{a}} \times \tau_{\vec{a}}^2\times \cdots)_*(\mu^{(d)}_{\vec{a}})_\Delta^{\infty}.$$
It follows that $\big((X^d)^\Z, (\X^d)^\Z, \mu_\A^{(\infty)} \big)$ is isomorphic to $((X^d)^\Z, (\X^d)^\Z, (\mu^{(d)}_{\vec{a}})^\infty_\Delta)$,  and hence also isomorphic to $(X^d, \X^d, \mu^{(d)}_{\vec{a}})$.
That is, if we define
\begin{equation*}
  \begin{split}
     \phi:(X^d, \X^d, \mu^{(d)}_{\vec{a}}, \tau_{\vec{a}}) & \rightarrow \big((X^d)^\Z, (\X^d)^\Z, \mu_\A^{(\infty)},\sigma \big), \\
     {\bf y}& \mapsto (\ldots, \tau^{-1}_{\vec{a}} {\bf y}, {\bf y}, \tau_{\vec{a}} {\bf y}, \tau_{\vec{a}}^2 {\bf y},\ldots ), \ \forall {\bf y}\in X^d
   \end{split}
\end{equation*}
then it is an isomorphism.

\subsubsection{Weakly mixing m.p.s.}\
\medskip

\paragraph{\em (I)\ Case $\A=\{n^2\}$}
\
\medskip

Let $(X,\X,\mu, T)$ be a weakly mixing m.p.s. and $\A=\{p\}$ with $p(n)=n^2$. In this case we show that
\begin{equation}\label{}
  (X^\Z, \X^\Z, \mu^{(\infty)}_\A, \sigma)\cong (X^\Z, \X^\Z, \mu^{\Z}, \sigma),
\end{equation}
that is, it is a Bernoulli system.

Recall that for $x\in X$,
$\omega_x^\A=(\ldots, T^{p(-1)}x, \underset{\bullet}x,T^{p(1)}x, T^{p(2)}x, \ldots)=(T^{p(n)}x)_{n\in \Z}.$
For all $l\in \N$, and all $f_{-l}, \ldots, f_l\in C(X)$,
\begin{equation*}
\begin{split}
  &\int_{X^{\Z}} f_{-l}(x_{-l})f_{-l+1}(x_{-l+1})\cdots f_l(x_l)d\mu_\A^{(\infty)}(\ldots, x_{-1}, x_0, x_1,\ldots)\\
  &\quad =\lim_{N\rightarrow +\infty} \frac{1}{N}\sum_{n=0}^{N-1}\int_X \prod_{j=-l}^l f_j (T^{p(n+j)}x) d\mu(x) \\
  &\quad = \big(\int_X f_{-l} d\mu\big)\big( \int_X f_{-l+1} d\mu \big)\cdots \big(\int_X f_l d\mu\big),
\end{split}
\end{equation*}
where the last equation holds by \cite[Theorem 1.1]{Bergelson87}.
Thus $\mu^{(\infty)}_\A=\cdots \times \mu\times \mu \times \cdots =\mu^\Z.$

\medskip

\paragraph{\em (II)\ Case $\A=\{n, n^2\}$}
\
\medskip

Now we study the case when $\A=\{p_1, p_2\}$ with $p_1(n)=n, p_2(n)=n^2$. In this case we show that
\begin{equation}\label{}
  ((X^2)^\Z, (\X^2)^\Z, \mu^{(\infty)}_\A, \sigma)\cong (X\times X^\Z, \X\times \X^\Z, \mu\times \mu^{\Z}, T\times \sigma).
\end{equation}

For $x\in X$,
$\omega_x^\A=\Big((T^{n}x, T^{n^2}x)\Big)_{n\in \Z}.$ Thus, for all $l\in \N$, and all $f_{-l}^{(1)}, f_{-l}^{(2)},  \ldots, f_l^{(1)}, f_l^{(2)}\in C(X)$, ${\bf x}=\big((x_n^{(1)}, x_n^{(2)})\big)_{n\in \Z}$,
\begin{equation*}
\begin{split}
  &\int_{(X^2)^{\Z}}\prod_{j=-l}^l f_j^{(1)}(x_j^{(1)})f_j^{(2)}(x_j^{(2)}) d\mu_\A^{(\infty)}({\bf x})\\
  &\quad =\lim_{N\rightarrow +\infty} \frac{1}{N}\sum_{n=0}^{N-1}\int_X \prod_{j=-l}^l  f_j^{(1)}(T^{n+j}x)f_j^{(2)}(T^{(n+j)^2}x) d\mu(x))\\
  &\quad =\lim_{N\rightarrow +\infty} \frac{1}{N}\sum_{n=0}^{N-1}\int_X \prod_{j=-l}^l \big(f_j^{(1)}\circ T^{j}\big)(T^{n}x)
  \prod_{j=-l}^lf_j^{(2)}(T^{(n+j)^2}x) d\mu(x) \\
  &\quad = \Big(\int_X \prod_{j=-l}^l\big( f_j^{(1)}\circ T^{j} \big) d\mu \Big)\prod_{j=-l}^l\int_X f_{j}^{(2)} d\mu,
\end{split}
\end{equation*}
where the last equation holds by \cite[Theorem 1.1]{Bergelson87}.
Hence
\begin{equation*}
  \mu^{(\infty)}_\A=\big((\cdots \times T^{-1}\times \underset{\bullet}{\id} \times T\times \cdots )_*\mu^\infty_\Delta \big)\times \mu^\Z.
\end{equation*}
It follows that $\mu^{(\infty)}_\A$ is isomorphic to $\mu\times \mu^\Z$.
If we define
\begin{equation*}
  \begin{split}
    \phi:(X\times X^\Z, & \X\times \X^d, \mu\times \mu^{\Z}) \rightarrow ((X^2)^\Z, (\X^2)^\Z, \mu^{(\infty)}_\A), \\
     & \big(x, (x_n)_{n\in \Z}\big) \mapsto \big( T^nx, x_n\big )_{n\in \Z},\ \forall \big(x, (x_n)_{n\in \Z}\big)\in X\times X^\Z,
   \end{split}
\end{equation*}
then it is an isomorphism between them.

\subsection{An equivalent way to see $\mu^{(\infty)}_\A$} \label{subsection-linear-terms}\
\medskip

In previous subsections we see that when $\A=\{p_1, \ldots, p_d\}$ with $p_i(n)=a_in$, (where $a_1, \ldots, a_d$ are distinct non-zero integers), $((X^d)^\Z, (\X^d)^\Z, \mu^{(\infty)}_\A, \sigma)\cong (X^d, \X^d, \mu^{(d)}_{\vec{a}}, \tau_{\vec{a}})$. In this subsection,
we give an equivalent way to see $\mu^{(\infty)}_\A$, and it will be more convenient to deal with certain problems when $\A$ contains linear elements.

\subsubsection{}
We assume that $\A=\{p_1, \ldots, p_{s}, p_{s+1}, \ldots, p_{d}\}$, where $s\ge 0$, $a_1,\ldots, a_s$ are distinct non-zero integers, $p_i(n)=a_in$ ($1\le i\le s$) and $\deg p_{i}\ge 2, p_i(0)=0, s+1\le i\le d$.
Recall that
$$\omega_x^\A = \Big( (T^{a_1n}x, \ldots, T^{a_{s}n}x, T^{p_{s+1}(n)}x,\ldots, T^{p_d(n)}x)\Big)_{n\in \Z}\in (X^d)^\Z.$$

For all $l \in \N$ and all $f_j^\otimes = f_j^{(1)}\otimes f_j^{(2)} \otimes \cdots \otimes f_j^{(d)} \in C(X^d), -l \le j\le l$, denote
$$f_j^{\otimes_I} = f_j^{(1)}\otimes f_j^{(2)} \otimes \cdots \otimes f_j^{(s)},\  f_j^{\otimes_{II}} = f_{j}^{(s+1)}\otimes f_j^{(s+2)} \otimes \cdots \otimes f_j^{(d)}. $$
Thus $f_j^\otimes = f_j^{\otimes_I} \otimes f_j^{\otimes_{II}}$.
Now
\begin{equation}\label{g1}
\begin{split}
  &\quad \int_{(X^d)^{\Z}} \Big( \bigotimes_{j=-l}^l f_j^\otimes \Big)({\bf x})d\mu_\A^{(\infty)}({\bf x})\\
  &\quad =\lim_{N\rightarrow +\infty} \frac{1}{N}\sum_{n=0}^{N-1}\int_X \prod_{j=-l}^l f_j^\otimes (T^{a_1(n+j)}x, T^{a_2(n+j)}x, \ldots, T^{p_d(n+j)}x) d\mu(x)\\
  &\quad =\lim_{N\rightarrow +\infty} \frac{1}{N}\sum_{n=0}^{N-1}\int_X \prod_{j=-l}^l f_j^{\otimes_I} (T^{a_1(n+j)}x,\ldots, T^{a_s(n+j)}x)\cdot \\
  &\quad \quad \quad \quad \quad \quad \quad \quad \quad \quad f_l^{\otimes _{II}}( T^{p_{s+1}(n+j)}x, \ldots, T^{p_d(n+j)}x) d\mu(x) \\
  &\quad =\lim_{N\rightarrow +\infty} \frac{1}{N}\sum_{n=0}^{N-1}\int_X (T^{a_1}\times  \cdots \times T^{a_s})^n \Big( \prod_{j=-l}^l f_j^{\otimes_I} (T^{a_1j}x,\ldots, T^{a_s j}x) \Big)\cdot \\
  & \quad \quad \quad \quad \quad \quad \quad \quad \quad \quad  \prod_{j=-l}^l f_j^{\otimes _{II}}( T^{p_{s+1}(n+j)}x, \ldots, T^{p_d(n+j)}x) d\mu(x) .
\end{split}
\end{equation}





\subsubsection{Definition of $\widetilde{\mu}^{(\infty)}_\A$}
Now we define the  new system.

Let $\Omega=X^s\times (X^{d-s})^\Z$. Define $\widetilde{\sigma}: \Omega\rightarrow \Omega$ as follow: for
$(x_1,x_2, \ldots, x_s)\times {\bf x}\in \Omega=X^s\times (X^{d-s})^\Z,$
\begin{equation}\label{x1}
 \widetilde{\sigma}\Big((x_1,x_2, \ldots, x_s)\times {\bf x}\Big)=\Big((T^{a_1}x_1, T^{a_2}x_2, \ldots, T^{a_s}x_s)\times \sigma_{II} {\bf x},
\end{equation}
where $\sigma_{II}$ is the shift on $(X^{d-s})^\Z$. That is, $\widetilde{\sigma}=\tau_{\vec{a}}\times \sigma_{II}$, where $\tau_{\vec{a}}=T^{a_1}\times \cdots \times T^{a_s}$.
Let
$$\xi_x^\A = \Big((x, \ldots , x), (T^{p_{s+1}(j)}x,\ldots, T^{p_d(j)}x)_{j\in \Z}\Big)\in X^s\times (X^{d-s})^\Z.$$
Then
$$\widetilde {\sigma}^n\xi^\A_x=\Big((T^{a_1n }x, \ldots , T^{a_sn}x), (T^{p_{s+1}(n+j)}x,\ldots, T^{p_d(n+j)}x)_{j\in \Z}\Big)\in X^s\times (X^{d-s})^\Z.$$
Note that when $s=0$, $\xi_x^\A=\w_x^\A$.

Next we define a measure $\widetilde{\mu}^{(\infty)}_\A$ on $\Omega$.
For all $f_1,f_2,\ldots, f_s\in C(X)$, all $l \in \N$ and all $f_j^{\otimes_{II}} = f_j^{(s+1)}\otimes f_j^{(s+2)} \otimes \cdots \otimes f_j^{(d)} \in C(X^{d-s}), -l \le j\le l$,
\begin{equation}\label{g2}
\begin{split}
  & \int_{X^s\times (X^{d-s})^{\Z}} \big(f_1\otimes  \cdots \otimes f_s\big) \otimes  \Big( \bigotimes_{j=-l}^l f_j^{\otimes_{II}} \Big)((x_1,\ldots, x_s),{\bf x})d\widetilde{\mu}_\A^{(\infty)}\\
  &\quad =\lim_{N\rightarrow +\infty} \frac{1}{N}\sum_{n=0}^{N-1}\int_X \prod_{i=1}^sf_i(T^{a_in}x) \prod_{j=-l}^l f_j^{\otimes_{II}} (T^{{p_{s+1}}(n+j)}x, \ldots, T^{p_d(n+j)}x) d\mu(x) \\
  &\quad =\lim_{N\rightarrow +\infty} \frac{1}{N}\sum_{n=0}^{N-1}\int_X (T^{a_1}\times \cdots \times T^{a_s})^n (f_1\otimes  \cdots \otimes f_s)(x,\ldots, x) \\
  & \quad \quad \quad \quad \quad \prod_{j=-l}^l f_j^{\otimes_{II}} (T^{{p_{s+1}}(n+j)}x, \ldots, T^{p_d(n+j)}x) d\mu(x) ,
\end{split}
\end{equation}
where $((x_1,x_2,\ldots, x_s),{\bf x}) =((x_1,x_2,\ldots, x_s),({\bf x}_n)_{n\in \Z} )\in X^s\times (X^{d-s})^\Z$.

\subsubsection{$\mu^{(\infty)}_\A$ and $\widetilde{\mu}^{(\infty)}_\A$ are isomorphic}
Compared \eqref{g1} with \eqref{g2}, we see that
$$((X^d)^\Z, (\X^d)^\Z, \mu^{(\infty)}_\A, \sigma)\cong (X^s\times (X^{d-s})^\Z, \X^s\times (\X^{d-s})^\Z, \widetilde{\mu}^{(\infty)}_\A, \widetilde{\sigma}).$$
That is, if we define
\begin{equation}\label{a9}
  \begin{split}
     \phi: & (X^s\times (X^{d-s})^\Z,  \X^s\times (\X^{d-s})^\Z, \widetilde{\mu}^{(\infty)}_\A) \rightarrow ((X^d)^\Z, (\X^d)^\Z, \mu^{(\infty)}_\A), \\
       & ({\bf y}, {\bf x}) \mapsto \big( \tau_{\vec{a}}^j {\bf y}, {\bf x}_j\big)_{j\in \Z},\ \forall ({\bf y}, {\bf x})=\Big({\bf y}, \big({\bf x}_j\big)_{j\in \Z}\Big)\in X^s\times (X^{d-s})^\Z,
   \end{split}
\end{equation}
then it is an isomorphism.

Note that if $s=0$, then $\widetilde{\mu}^{(\infty)}_\A=\mu^{(\infty)}_\A$; and if $s=d$, then $\widetilde{\mu}^{(\infty)}_\A=\mu^{(d)}_{\vec{a}}$.

\subsubsection{Compare $\mu^{(\infty)}_\A$ with $\widetilde{\mu}^{(\infty)}_\A$}
We have shown that
$((X^d)^\Z, (\X^d)^\Z, \mu^{(\infty)}_\A, \sigma)$ is isomorphic to $(X^s\times (X^{d-s})^\Z, \X^s\times (\X^{d-s})^\Z, \widetilde{\mu}^{(\infty)}_\A, \widetilde{\sigma}).$
It is clear that the definition of $\mu^{(\infty)}_\A$ is cleaner than $\widetilde{\mu}^{(\infty)}_\A$, as we do not care whether there are linear terms  in $\A$ or not. But when we study certain dynamical properties, $\widetilde{\mu}^{(\infty)}_\A$ is easier to handle than $\mu^{(\infty)}_\A$. For example, in the next section, we use $\widetilde{\mu}^{(\infty)}_\A$ in Lemma \ref{lem-vdc}, and it much more complicated if we use $\mu^{(\infty)}_\A$ instead.

\subsubsection{}
Similar to Remark \ref{rem-3.5}, we have that $\mu_\A^{(\infty)}$ is a standard measure.
And similar to the proof of Proposition \ref{prop-mesurable}, in \eqref{g2} we can replace continuous functions by measurable ones.
That is, we have the following proposition.
\begin{prop}
Let $(X, \X, \mu, T)$ be a m.p.s. and $d\in \N$.  Let
$\A=\{p_1, \ldots, p_s, p_{s+1}, \ldots, p_{d}\}$ be a family of non-constant integral polynomials, where $s\ge 0$, $p_i(n)=a_in$ ($1\le i\le s$), $a_1,a_2,\ldots, a_s$ are distinct non-zero integers, and $\deg p_{i}\ge 2, p_i(0)=0, s+1\le i\le d$.
For all $f_1,f_2,\ldots, f_s\in L^\infty(X,\mu)$, all $l \in \N$ and all $f_j^{\otimes_{II}} = f_j^{(s+1)}\otimes f_j^{(s+2)} \otimes \cdots \otimes f_j^{(d)}$ (where $f^{(t)}_j\in L^\infty(X,\mu)$, $s+1\le t\le d, -l \le j\le l$),
\begin{equation}\label{g3}
\begin{split}
  & \quad \int_{X^s\times (X^{d-s})^{\Z}} \big(\bigotimes_{i=1}^sf_i\big) \otimes  \Big( \bigotimes_{j=-l}^l f_j^{\otimes_{II}} \Big)((x_1,x_2,\ldots, x_s),{\bf x})d\widetilde{\mu}_\A^{(\infty)}\\
  &\quad =\lim_{N\rightarrow +\infty} \frac{1}{N}\sum_{n=0}^{N-1}\int_X \prod_{i=1}^sf_i(T^{a_in}x) \prod_{j=-l}^l f_j^{\otimes_{II}} (T^{{p_{s+1}}(n+j)}x, \ldots, T^{p_d(n+j)}x) d\mu(x)  ,
\end{split}
\end{equation}
where $((x_1,x_2,\ldots, x_s),{\bf x}) =((x_1,x_2,\ldots, x_s),({\bf x}_n)_{n\in \Z} )\in X^s\times (X^{d-s})^\Z$.
\end{prop}

Also similar to Proposition \ref{prop-iso}, if $(X_i,\X_i,\mu_i,T_i), i=1,2$, are two isomorphic m.p.s., then
$(X_1\times (X_1^{d-s})^\Z, \X_1\times (\X_1^{d-s})^\Z, (\widetilde{\mu_1})_\A^{(\infty)}, \langle T_1^\infty, \widetilde{\sigma}\rangle)$ and $$(X_2\times (X_2^{d-s})^\Z, \X_2\times (\X_2^{d-s})^\Z, (\widetilde{\mu_2})_\A^{(\infty)}, \langle T_2^\infty, \widetilde{\sigma}\rangle) $$
are also isomorphic.

\subsection{Topological correspondences}\

\subsubsection{$N_\infty(X,\A)$ and $M_\infty(X,\A)$}
Let $(X,T)$ be a t.d.s.
We set
\begin{equation*}\label{}
  N_\infty(X,\A)=\overline{\bigcup\{\O(\w_x^\A,\sigma): x\in X\}}=\overline{\bigcup\{\sigma^n\w_x^\A: x\in X, n\in \Z\}}\subseteq (X^d)^{\Z},
\end{equation*}
and
\begin{equation*}\label{}
  M_\infty(X,\A)=\overline{\bigcup\{\O(\xi_x^\A, \widetilde{\sigma}): x\in X\}}=\overline{\bigcup\{\widetilde{\sigma}^n\xi_x^\A: x\in X, n\in\Z\}}\subseteq X^s\times (X^{d-s})^{\Z}.
\end{equation*}
It is clear that $N_\infty(X,\A )$ is invariant under the action of $T^\infty$ and $\sigma$, and
$M_\infty(X,\A )$ is invariant under the action of $T^\infty$ and $\widetilde{\sigma}$. Thus  $(N_\infty(X,\A ), \langle T^\infty, {\sigma}\rangle)$ and $(M_\infty(X,\A), \langle T^\infty, \widetilde{\sigma}\rangle)$ are $\Z^2$-t.d.s., which were introduced and investigated in \cite{HSY22-1}.

\subsubsection{}
The following theorem indicates the connection between them and the systems we introduced in the current paper.

\begin{thm}\label{support}
Let $({X},T)$ be a t.d.s. and let $\A=\{p_1,\cdots, p_d\}$ be a family of non-constant integral polynomials with $p_i(0)=0, 1\le i\le d$. Then for all $\mu\in M_T(X)$, we have
$${\rm supp}\ \mu_\A^{(\infty)}\subseteq  N_\infty(X,\A)\ \text{and}\ {\rm supp}\  \widetilde{\mu}_\A^{(\infty)}\subseteq  M_\infty(X,\A).$$
If in addition ${\rm supp}\ \mu=X$, then ${\rm supp}\ \mu_\A^{(\infty)}=  N_\infty(X,\A)$ and ${\rm supp}\  \widetilde{\mu}_\A^{(\infty)}=  M_\infty(X,\A).$
\end{thm}

\begin{proof}
We only show the first inclusion, 
 and the second one can be proved similarly.

The measure $\mu_\A^{(\infty)}$ is a weak limit of averages of Dirac masses at points
of the form $\sigma^n \w_x^\A$ for $n\in \Z$ and $x\in X$. Since all of these points belong to $N_\infty(X,\A)$, the measure $\mu^{(\infty)}_\A$ is concentrated on this set, i.e. ${\rm supp}\ \mu_\A^{(\infty)}\subseteq  N_\infty(X,\A)$.

It remains to show that $N_\infty(X,\A)={\rm supp}\ \mu_\A^{(\infty)}$ if ${\rm supp}\ \mu=X$. Since ${\rm supp}\ \mu_\A^{(\infty)}$ is invariant under the actions of $T^\infty$ and $\sigma$, it suffices to show that for any $x\in X$, we have $\w_x^\A\in {\rm supp}\ \mu_\A^{(\infty)}$. Thus we need to show that if $U$ is an open neighbourhood of $\w_x^\A$, then $ \mu_\A^{(\infty)}(U)>0$.
Without loss of generality, we assume that $U=\big(\prod_{j=-\infty}^{-l-1} X^d \big)\times \big(\prod_{j=-l}^l U^{\otimes}_j\big)\times\big(\prod_{j=l+1}^{\infty} X^d \big),$
where $l \in \N$, $U_j^\otimes = U_j^{(1)}\times U_j^{(2)} \times \cdots \times U_j^{(d)}$, and $U_j^{(t)}$ is open in $X$, $1\le t\le d,  -l \le j\le l$. Since $\w_x^\A=(T^{\vec{p}(n)}x^{\otimes d})_{n\in \Z}\in U$, we have $T^{p_t{(j)}}x\in U^{(t)}_j$, $1\le t\le d,  -l \le j\le l$. Choose a non-empty open subset $V$ of $X$ such that
$$V\subseteq T^{-p_t(j)}U_j^{(t)}, \ \forall 1\le t\le d,  -l \le j\le l.$$
Thus we have
\begin{equation*}
  \begin{split}
    \mu_\A^{(\infty)}(U)
    &= \mu_\A^{(\infty)}\Big(\big(\prod_{j=-\infty}^{-l-1} X^d \big)\times \big(\prod_{j=-l}^l U^{\otimes}_j\big)\times\big(\prod_{j=l+1}^{\infty} X^d \big) \Big)\\
    &=\lim_{N\to\infty}\frac{1}{N} \sum_{n=0}^{N-1} \mu\Big(\bigcap_{j=-1}^l \big(T^{-p_1(n+j)}U_j^{(1)}\cap \cdots \cap T^{-p_d(n+j)}U_j^{(d)} \big)\Big)\\
&= \lim_{N\to\infty}\frac{1}{N} \sum_{n=0}^{N-1} \mu\Big(\bigcap_{j=-1}^l \bigcap_{k=1}^d T^{-(p_k(n+j)-p_k(j))-p_k(j)}U_j^{(k)}\big)\\
&\ge \lim_{N\to\infty} \frac{1}{N}\sum_{n=0}^{N-1} \mu\Big(\bigcap_{j=-1}^l \big(T^{-(p_1(n+j)-p_1(j))}V\cap \cdots \cap T^{-(p_d(n+j)-p_d(j))}V \big)\Big).
   \end{split}
\end{equation*}
Since ${\rm supp} \ \mu=X$, we have $\mu(V)>0$.
By Theorem \ref{thm-mul-rec}, $$\lim_{N\to\infty} \frac{1}{N} \sum_{n=0}^{N-1} \mu\Big(\bigcap_{j=-1}^l \big(T^{-(p_1(n+j)-p_1(j))}V\cap \cdots \cap T^{-(p_d(n+j)-p_d(j))}V \big)\Big)>0,$$
which implies that $\mu_\A^{(\infty)}(U)>0$. The proof is complete.
\end{proof}

\begin{rem} We have
\begin{enumerate}
\item By the proof of Theorem \ref{support}, we actually proved the following equality:
$${\rm supp}\ \mu_\A^{(\infty)}=\overline{\{\sigma^n\w_x^\A: x\in {\rm supp} \mu, \ n\in \Z\}}.$$
When $(X,T)$ is minimal, we have ${\rm supp}\ \mu=X$ for each $\mu\in {\mathcal M}_T(X)$ and ${\rm supp}\ \mu_\A^{(\infty)}=  N_\infty(X,\A)$.
Similar statement holds for  $\widetilde{\mu}_\A^{(\infty)}$.

\item By Theorem \ref{support}, if $(X,T)$ is minimal, then $(N_\infty({X},\A), \langle T^\infty, \sigma\rangle)$ is an $E$-system, i.e. it is transitive and have an invariant measure with a full support. In fact we can say more, i.e. it was proved \cite{HSY22-1} that
$(N_\infty({X},\A), \langle T^\infty, \sigma\rangle)$ is an $M$-system, i.e., it is transitive and the set of $\langle T^\infty, \sigma\rangle$-minimal points is dense.
\end{enumerate}
\end{rem}

\section{The $\sigma$-algebra of invariant sets of $((X^d)^\Z, (\X^d)^\Z, \mu^{(\infty)}_\A, \sigma)$ }\label{section-reduing-nil}

In this section the $\sigma$-algebra of invariant measurable sets of
$((X^d)^\Z, (\X^d)^\Z, \mu^{(\infty)}_\A, \sigma)$ will be studied. The main result of this section is that its $\sigma$-algebra of invariant measurable sets
is isomorphic to the $\sigma$-algebra of invariant measurable sets of its $\infty$-step pro-nilfactors.

\subsection{Multiple ergodic averages along polynomials of several variable}\
\medskip

We  need the following multiple ergodic averages along polynomials of several variable.

\begin{thm}\cite[Theorem 3]{Leibman05-Isr}\label{thm-Leibman}
Let $(X,\X,\mu, T)$ be an ergodic m.p.s. and $m, d\in \N$. Let $p_1,p_2,\ldots, p_d: \Z^m\rightarrow \Z$ be non-constant essentially distinct polynomials (i.e. $p_i-p_j$ is not constant for $i\neq j$). Then there is $k\in \N$ such that for any $f_1,f_2,\ldots, f_d\in L^\infty(X)$ with $\HK f_j\HK_k=0$ for some $j\in \{1,2,\ldots, d\}$, one has
$$\lim_{N\to\infty} \frac{1}{|\Phi_N|}\sum_{{\bf n}\in \Phi_N} T^{p_1({\bf n})}f_1\cdots T^{p_d({\bf n})}f_d=0$$
in $L^2(X,\mu)$ for any F{\o}lner sequence $\{\Phi_N\}_{N=1}^\infty$ in $\Z^m$.
\end{thm}

\begin{rem}
The integer $k$ in Theorem \ref{thm-Leibman} depends only on the weight of polynomials $\{p_1,\ldots, p_d\}$ (see \cite{Leibman05-Isr} for the definition of the weight and the proof of this fact).
\end{rem}

\subsection{Condition $(\spadesuit)$}\
\medskip

For a given integral polynomial $p$ with $p(0)=0$, for each $j\in\Z$ let $p^{[j]}$ be defined by
$$p^{[j]}(n)=p(n+j)-p(j), \ \forall n\in\Z.$$


\begin{de}
Let $\A=\{p_1, p_2, \ldots, p_{d}\}$ be a family of integral polynomials. We say $\A$ satisfies {\em condition $(\spadesuit)$} if $p_1(0)=\ldots =p_d(0)=0$ and
\begin{enumerate}
  \item $p_i(n)=a_i n, 1\le i\le s$, where $s\ge 0$, and $a_1,a_2,\ldots, a_s$ are distinct non-zero integers;
  \item $\deg p_{j}\ge 2, s+1\le j\le d$;
  \item for each $i\neq j\in \{s+1,s+2,\ldots, d\}$, $p_j^{[k]}\neq p_i^{[t]}$ for any $k,t\in \Z$.

\end{enumerate}
\end{de}

\begin{rem}
The key point in the definition of condition $(\spadesuit)$ is the term (3). 
As we explained in the introduction, when we study $\A=\{p\}$ with $p(n)=n^2$, we actually consider infinitely many polynomials
$p^{[k]}$ ($k\in\Z$) defined by $p^{[k]}(n)=p(n+k)-p(k)=n^2+2kn, n\in \Z$. But when we study $\A=\{n^2, n^2+2n, n^2+8n\}$, we still consider the same  polynomials $p^{[k]}, k\in \Z$. To exclude this kind of situation, we add the term (3) when we study $\A$.
\end{rem}

We have the following simple observation concerning the condition $(\spadesuit)$.

\begin{lem}\label{ww=dem}
Let $k\in\N$ and $Q=\{q_1,q_2,\ldots,q_k\}$ be integral polynomials with $q_i(0)=0, 1\le i\le k$. Then there is $1\le d\le k$
and $\A=\{p_1,\ldots,p_d\}\subset \{q_1,\ldots,q_k\}$ such that $\A$ satisfies condition $(\spadesuit)$, and $\{q_1, \ldots, q_k\}\subseteq \{p_i^{[m]}: 1\le i\le d, m\in \Z \}$.

\end{lem}

\begin{proof}
If there are $1\le i\not=j\le k$ and $k_1,k_2\in \Z$ such that $q_i^{[k_1]}=q_j^{[k_2]}$ we just remove one of them (say $q_j$) to get a proper subset $Q'=Q\setminus\{q_j\}$ of $Q$. Continuing this process for finitely many steps, the argument will stop, and the resulting $\A$ is what we want.
\end{proof}


\begin{exam}\label{exam1}
Let $\A'=\{n^2, n^2+6n, n^2+10n\}$ and $\A=\{n^2\}$. Then by the definition, $\A'$ does not satisfy the condition $(\spadesuit)$, $\A$ satisfies the condition $(\spadesuit)$ and $\A'\subseteq \{p^{[m]}: p(n)=n^2\in \A, m\in \Z\}$. Let $(X,\X,\mu, T)$ be a m.p.s. Now we show that $(X^\Z,\X^\Z, \mu_\A^{(\infty)},\sigma)$ is isomorphic to $((X^3)^\Z, (\X^3)^\Z, \mu_{\A'}^{(\infty)},\sigma)$.

Recall that
$$\w_x^\A=(T^{n^2}x)_{n\in \Z}\in X^\Z, \ \text{and}\ \w_x^{\A'}=\Big((T^{n^2}x, T^{n^2+6n}x, T^{n^2+10n}x)\Big)_{n\in \Z}\in (X^3)^\Z .$$
Note that
$$\sigma^k\w_x^\A=(T^{(n+k)^2}x)_{n\in \Z}=(T^{n^2+2kn+k^2}x)_{n\in Z}=(T^{\infty})^{k^2}(T^{n^2+2kn}x)_{n\in \Z}, \forall k\in \Z.$$
We have $\w_x^{\A'}=\Big(\w_x^\A, (T^\infty)^{-9}\sigma^3 \w_x^\A, (T^\infty)^{-25}\sigma^5 \w_x^\A\Big)$ for all $x\in X$.

Define $\Psi=\id \times (T^\infty)^{-9}\sigma^3\times (T^\infty)^{-25}\sigma^5$:
$$\Psi: X^\Z \rightarrow (X^3)^\Z, \ {\bf x}\mapsto ({\bf x},(T^\infty)^{-9}\sigma^3{\bf x}, (T^\infty)^{-25}\sigma^5 {\bf x}), \ \forall {\bf x}\in X^\Z.$$
It is easy to verify that
$$\Psi_*(\mu_\A^{(\infty)})=\mu_{\A'}^{(\infty)},$$
and $\Psi$ is an isomorphism from $(X^\Z,\X^\Z, \mu_\A^{(\infty)},\sigma)$  to $((X^3)^\Z, (\X^3)^\Z, \mu_{\A'}^{(\infty)},\sigma)$.
\end{exam}

Similar to Example \ref{exam1}, we have the following general result:

\begin{thm}
Let $(X,\X,\mu, T)$ be a m.p.s. and $\A'=\{q_1,\ldots, q_k\}$ be a family of integral polynomials with $q_i(0)=0, 1\le i\le k$. Then there is a family of integral polynomials $\A=\{p_1,\ldots, p_d\}$ satisfying the condition $(\spadesuit)$ such that
$$((X^d)^\Z, (\X^d)^\Z, \mu^{(\infty)}_\A, \sigma)\cong ((X^k)^\Z, (\X^k)^\Z, \mu^{(\infty)}_{\A'}, \sigma),$$
and
$$(X^s\times (X^{d-s})^\Z, \X^s\times (\X^{d-s})^\Z, \widetilde{\mu}^{(\infty)}_\A, \widetilde{\sigma}) \cong (X^t\times (X^{k-t})^\Z, \X^t\times (\X^{k-t})^\Z, \widetilde{\mu}^{(\infty)}_{\A'}, \widetilde{\sigma}),$$
where $t$ is the number of linear polynomials in $\A'$.
\end{thm}


\subsection{Some lemmas}

\begin{lem}[van der Corput lemma]\cite{Bergelson87}\label{vanderCoput}
Let $\{\zeta_n\}$ be a bounded sequence in a Hilbert space $\mathcal{H}$
with norm $\parallel \cdot \parallel$ and inner product $\langle \cdot,
\cdot \rangle$. Then
\begin{equation*}
    \limsup_{N\to \infty} \Big\| \frac{1}{N}\sum _{n=1}^N
    \zeta_n\Big\|^2
    \le \limsup_{H\to \infty}\frac{1}{H} \sum_{h=1}^H\limsup_{N\to\infty}
    \left|\frac{1}{N}\sum_{n=1}^N \langle \zeta_n, \zeta_{n+h} \rangle  \right|.
\end{equation*}
\end{lem}


Recall that $\widetilde{\sigma}: X^s\times (X^{d-s})^\Z\rightarrow X^s\times (X^{d-s})^\Z$
is defined as follows: for
$\big((x_1, \ldots, x_s), {\bf x}\big)\in X^s\times (X^{d-s})^\Z,$
$ \widetilde{\sigma}\Big((x_1, \ldots, x_s),{\bf x}\Big)=\Big(\tau_{\vec{a}}(x_1,  \ldots, x_s), \sigma_{II} {\bf x}\Big),$
where $\tau_{\vec{a}}=T^{a_1}\times \cdots \times T^{a_s}$ and $\sigma_{II}$ is the shift on $(X^{d-s})^\Z$.


We now show
\begin{lem}\label{lem-vdc}
Let $(X, \X, \mu, T)$ be an ergodic m.p.s. and $d\in \N$.  Let $\A=\{p_1,\ldots, p_d\}$ be a family of non-constant integral polynomials satisfying the condition $(\spadesuit)$.
Then for
all $f_1,f_2,\ldots, f_s\in L^\infty(X,\mu)$, all $l \in \N$ and all $f^{(t)}_j\in L^\infty(X,\mu)$ with $\HK f_i \HK_\infty=0$ for some $1\le i\le s$ or $\HK f_j^{(t)}\HK_\infty=0$ for some $s+1\le t\le d, -l\le j\le l$, one has
\begin{equation}\label{AP-VDC}
\limsup_{N\to\infty}\left \|
\frac{1}{N}\sum_{n=0}^{N-1} \Big(\big(\bigotimes_{j=1}^sf_j \big)\otimes \bigotimes_{j=-l}^{l} f_j^{\otimes_{II}} \Big) \Big(\tau_{\vec{a}}^n(x_1,\ldots, x_s), \sigma_{II}^n {\bf x}\Big)\right\|_{L^2( \widetilde{\mu}_\A^{(\infty)})}=0.
\end{equation}\end{lem}
Note that here $f_j^{\otimes_{II}} = f_j^{(s+1)}\otimes f_j^{(s+2)} \otimes \cdots \otimes f_j^{(d)}$, and
$\big( \bigotimes_{j=1}^sf_j \big) \otimes \bigotimes_{j=-l}^{l} f_j^{\otimes_{II}} $ is regarded as an element of $L^\infty(X^s\times (X^{d-s})^{\Z}, \widetilde{\mu}_\A^{(\infty)})$.
\begin{proof}
Let $\A=\{p_1, \ldots, p_s, p_{s+1}, \ldots, p_{d}\}$ be a family of non-constant integral polynomials, where $s\ge 0$, $p_i(n)=a_in$ ($1\le i\le s$) satisfying $(\spadesuit)$.
Let $f_1,\ldots, f_s\in L^\infty(X,\mu)$, $l \in \N$ and $f^{(t)}_j\in L^\infty(X,\mu)$ with $\HK f_i \HK_\infty=0$ for some $1\le i\le s$ or $\HK f_j^{(t)}\HK_\infty=0$ for some $s+1\le t\le d, -l\le j\le l$. Without loss of generality, we assume that $\|f_i\|_\infty\le 1$ for all $i$
and $\|f_j^{(t)}\|_\infty\le 1$ for all $j, t$.  Write
\begin{equation*}
  \begin{split}
     \zeta_n& =\Big( \big(f_1\otimes \cdots \otimes f_s\big) \otimes \bigotimes_{j=-l}^{l} f_j^{\otimes_{II}} \Big) \Big(\tau_{\vec{a}}^n(x_1,\ldots, x_s), \sigma_{II}^n {\bf x}\Big)\\
     & = f_1(T^{a_1n}x_1)f_2(T^{a_2n}x_2)\cdots f_s(T^{a_sn}x_s)f_{-l}^{\otimes_{II}} ({\bf x}_{n-l}) f_{-l+1}^{\otimes_{II}} ({\bf x}_{n-l+1}) \cdots f_{l}^{\otimes_{II}}({\bf x}_{n+l})  \\
&= f_1(T^{a_1n}x_1)f_2(T^{a_2n}x_2)\cdots f_s(T^{a_sn}x_s) \prod_{j=-l}^l f_j^{(s+1)}(x^{(s+1)}_{n+j}) f_j^{(s+2)} (x_{n+j}^{(s+2)}) \cdots f_j^{(d)}(x_{n+j}^{(d)}),
   \end{split}
\end{equation*}
where ${\bf x}=({\bf x}_n)_{n\in \Z}=\Big((x^{(s+1)}_n, x^{(s+2)}_n,\ldots, x^{(d)}_n) \Big)_{n\in \Z} \in (X^{d-s})^{\Z}$.

By van der Corput lemma (Lemma \ref{vanderCoput}),
\begin{equation*}
    \limsup_{N\to \infty} \big\| \frac{1}{N}\sum _{n=0}^{N-1}
    \zeta_n\big\|^2_{L^2(\widetilde{\mu}_\A^{(\infty)})}
     \le \limsup_{H\to \infty}\frac{1}{H} \sum_{h=0}^{H-1}\limsup_{N\to\infty}
    \left|\frac{1}{N}\sum_{n=0}^{N-1} \int_{X^s\times (X^{d-s})^{\Z}} \overline{\zeta}_{n+h}\cdot\zeta_n d\widetilde{\mu}_\A^{(\infty)} \right|.
\end{equation*}
For simplicity, we only consider real-valued functions. Firstly we show two special cases, and then give the proof of the general case.

\medskip

\noindent {\bf Case I}: \ {\em  $s=d$}.

\medskip

In this case all polynomials involved are linear and $\widetilde{\mu}_\A^{(\infty)}={\mu}^{(d)}_{\vec{a}}$, where $\vec{a}=(a_1, \ldots, a_d)$. The results follows from Lemma \ref{lem-AP-vdc}.

\medskip

\noindent {\bf Case II}:\ {\em $s=0$}.

\medskip

In this case there are no linear items in $\A$ and $\widetilde{\mu}_\A^{(\infty)}={\mu}_\A^{(\infty)}$.
For $h> 2l$, we have
\begin{equation*}
\begin{split}
    & \left|\frac{1}{N}\sum_{n=0}^{N-1} \int_{(X^d)^{\Z}} {\zeta}_{n+h}\cdot\zeta_n d\mu_\A^{(\infty)}
    \right| \\
   &\quad = \left | \frac{1}{N}\sum_{n=0}^{N-1} \int_{(X^d)^{\Z}}
   \prod_{j=-l}^l {f_j^\otimes} ({\bf x}_{n+h+j}) \prod_{j=-l}^l f_j^\otimes ({\bf x}_{n+j}) d \mu_\A^{(\infty)}({\bf x})\right|\\
& \quad= \left | \frac{1}{N}\sum_{n=0}^{N-1} \int_{(X^d)^{\Z}}
    \prod_{j=-l}^l f_j^\otimes ({\bf x}_{n+j}) {f_j^\otimes} ({\bf x}_{n+h+j})  d \mu_\A^{(\infty)}({\bf x})\right|.
\end{split}
\end{equation*}
By the definition of $\mu_\A^{(\infty)}$,
\begin{equation}\label{z1}
\begin{split}
    & \int_{(X^d)^{\Z}}
    \prod_{j=-l}^l f_j^\otimes ({\bf x}_{n+j}) {f_j^\otimes} ({\bf x}_{n+h+j})  d \mu_\A ^{(\infty)}({\bf x}) \\
&\quad = \lim_{M\to\infty}\frac{1}{M}\sum_{m=0}^{M-1}
\int_{X} \prod_{j=-l}^l f_j^\otimes (T^{\vec{p}(m+n+j)}x^{\otimes d}) {f_j^\otimes} (T^{\vec{p}(m+n+h+j)}x^{\otimes d}) d\mu (x)
    \\
&\quad  =\lim_{M\to\infty}\frac{1}{M}\sum_{m=0}^{M-1}
\int_{X} \prod_{j=-l}^l  f_j^{(1)}(T^{p_1(m+n+j)}x) f_j^{(2)}(T^{p_2(m+n+j)}x)\cdots f_j^{(d)}(T^{p_d(m+n+j)}x)\cdot\\
& \quad \quad \quad \quad  {f_j^{(1)}}(T^{p_1(m+n+h+j)}x) {f_j^{(2)}}(T^{p_2(m+n+h+j)}x)\cdots { f_j^{(d)}}(T^{p_d(m+n+h+j)}x)  d \mu (x).
\end{split}
\end{equation}
By (3) of the condition $(\spadesuit)$ on $\A$, all $m$-polynomials
$$\{p_i(m+n+j), p_i(m+n+h+j), 1\le i\le d, -l\le j\le l\}$$ are essentially distinct, and by Theorem \ref{thm-Leibman}, the equation \eqref{z1} is $0$.
Thus for $h>2l$,
\begin{equation*}
 \left|\frac{1}{N}\sum_{n=0}^{N-1} \int_{(X^d)^{\Z}} {\zeta}_{n+h}\cdot\zeta_n d\mu_\A ^{(\infty)}
    \right| =0.
\end{equation*}
So by  Lemma \ref{vanderCoput},
\begin{equation*}
    \limsup_{N\to \infty} \Big\| \frac{1}{N}\sum _{n=0}^{N-1}
    \zeta_n\Big\|^2_{L^2(\mu_\A ^{(\infty)})} \le \limsup_{H\to \infty}\frac{1}{H} \sum_{h=0}^{H-1}\limsup_{N\to\infty}
    \left|\frac{1}{N}\sum_{n=0}^{N-1} \int_{X^s\times (X^{d-s})^{\Z}} {\zeta}_{n+h}\cdot\zeta_n d{\mu}_\A^{(\infty)} \right|=0.
\end{equation*}

\medskip

\noindent {\bf Case III}:\ {\em $1\le s\le  d-1$}.

\medskip
It is easy to see

$$  \Big\| \frac{1}{N}\sum _{n=0}^{N-1}
    \zeta_n\Big \|^2_{L^2(\widetilde{\mu}_\A^{(\infty)})}=
    \Big\langle\frac{1}{N}\sum _{n=0}^{N-1}
    \zeta_n,\frac{1}{N}\sum _{n=0}^{N-1}
    \zeta_n \Big\rangle=\frac{1}{N^2}\sum _{n_1,n_2=0}^{N-1}
    \langle\zeta_{n_1}, \zeta_{n_2}\rangle.
$$
For each $(n_1,n_2)\in [0,N-1]^2$ with $|n_1-n_2|>2l$, we have that
\begin{equation}\label{z3}
\begin{split}
& \langle\zeta_{n_1}, \zeta_{n_2}\rangle= \int_{X^s\times (X^{d-s})^{\Z}} {\zeta}_{n_1}\cdot\zeta_{n_2} d\widetilde{\mu}_\A^{(\infty)}\\
& \quad= \int_{X^s\times (X^{d-s})^{\Z}} f_1(T^{a_1n_1}x_1)\cdots f_s(T^{a_sn_1}x_s)f_{-l}^{\otimes_{II}} ({\bf x}_{n_{1}-l})\ldots f_{l}^{\otimes_{II}}({\bf x}_{n_1+l})\cdot\\
& \quad \quad \quad \quad \quad \quad \quad f_1(T^{a_1n_2}x_1)\cdots f_s(T^{a_sn_2}x_s) f_{-l}^{\otimes_{II}} ({\bf x}_{n_2-l}) \cdots f_{l}^{\otimes_{II}}({\bf x}_{n_2+l})  d\widetilde{\mu}_\A^{(\infty)}\\
&\quad = \int_{X^s\times (X^{d-s})^{\Z}} \Big( \prod_{i=1}^s f_i(T^{a_in_1}x_i)f_i(T^{a_in_2}x_i) \Big)\Big(\prod_{j=-l}^l f_j^{\otimes_{II}} ({\bf x}_{n_1+j}) f_j^{\otimes_{II}} ({\bf x}_{n_2+j})\Big) d\widetilde{\mu}_\A^{(\infty)}\\
&\quad =\lim_{M\to\infty}\frac{1}{M}\sum_{m=0}^{M-1}\int_{X} B(n_1,n_2,m) d\mu (x),
\end{split}
\end{equation}
where
\begin{equation}\label{z6}
\begin{split}
    &B(n_1,n_2,m)\\
    &=\Big( \prod_{i=1}^s f_i(T^{a_i(m+n_1)}x)f_i(T^{a_i(m+n_2)}x)\Big)\cdot
    \prod_{j=-l}^l f_j^{\otimes_{II}} (T^{p_{s+1}(m+n_1+j)}x,\ldots, T^{p_d(m+n_1+j)}x)\cdot\\
& \quad \quad \quad \quad \quad \quad \quad \quad f_j^{\otimes_{II}} (T^{p_{s+1}(m+n_2+j)}x,\ldots, T^{p_d(m+n_2+j)}x).
\end{split}
\end{equation}

For each $N\in \N$, by \eqref{z3} we choose an $M(N)\in \N$ such that $\lim_{N\to\infty} M(N)=+\infty$ and for each $(n_1,n_2)\in [0,N-1]^2$ with $|n_1-n_2| > 2l$, we have that
\begin{equation}\label{z4}
\begin{split}
\Big|\langle\zeta_{n_1}, \zeta_{n_2}\rangle-\frac{1}{M(N)}\sum_{m=0}^{M(N)-1}\int_{X} B(n_1,n_2,m) d\mu (x)\Big|<\frac{1}{N^3} .
\end{split}
\end{equation}

Let $\Phi_N=[0,N-1]^2 \times [0,M(N)]$. Then $\{\Phi_N\}_{N=1}^\infty$ is a
F{\o}lner sequence of $\Z^3$. Let $C_l^N=|\{(n_1,n_2)\in [0,N-1]^2: |n_1-n_2|\le 2l\}|$. By \eqref{z4}
\begin{equation}\label{z5}
\begin{split}
    & \ \ \ \big\| \frac{1}{N}\sum _{n=0}^{N-1}
    \zeta_n\big\|^2_{L^2(\widetilde{\mu}_\A^{(\infty)})}=\frac{1}{N^2}\sum _{n_1,n_2=0}^{N-1}
    \langle\zeta_{n_1}, \zeta_{n_2}\rangle \le \frac{1}{N^2}\sum _{n_1,n_2\in [0,N-1]^2\setminus C_l^N}
    \langle\zeta_{n_1}, \zeta_{n_2}\rangle +\frac{C_l^N}{N^2} \\
    &\quad \le \frac{1}{N^2}\sum _{n_1,n_2\in [0,N-1]^2\setminus C_l^N}
   \Big(\frac{1}{M(N)}\sum_{m=0}^{M(N)-1}\int_{X} B(n_1,n_2,m) d\mu (x)+\frac{1}{N^3} \Big) +\frac{C_l^N}{N^2}\\
   &\quad \le \frac{1}{N^2}\sum _{n_1,n_2\in [0,N-1]^2}\frac{1}{M(N)}\sum_{m=0}^{M(N)-1}\int_{X} B(n_1,n_2,m) d\mu (x) + \frac{1}{N^3}+\frac{2C_l^N}{N^2}\\
     &\quad = \frac{1}{|\Phi_N|} \sum_{(n_1,n_2,m)\in \Phi_N} \int_{X} B(n_1,n_2,m) d\mu (x) + \frac{1}{N^3}+\frac{2C_l^N}{N^2}.
\end{split}
\end{equation}
Since $\lim_{N\to\infty}\frac{2C_l^N}{N^2} =0$, we have
\begin{equation}\label{z7}
\begin{split}
    &\limsup_{N\to\infty}\big\| \frac{1}{N}\sum _{n=0}^{N-1}
    \zeta_n\big\|^2_{L^2(\widetilde{\mu}_\A^{(\infty)})}= \limsup_{N\to\infty}\frac{1}{N^2}\sum _{n_1,n_2=0}^{N-1}
    \langle\zeta_{n_1}, \zeta_{n_2}\rangle
     \\ & \le \limsup_{N\to\infty} \frac{1}{|\Phi_N|} \sum_{(n_1,n_2,m)\in \Phi_N} \int_{X} B(n_1,n_2,m) d\mu (x) \\
     &\le \limsup_{N\to\infty} \left\| \frac{1}{|\Phi_N|} \sum_{(n_1,n_2,m)\in \Phi_N} B(n_1,n_2,m)\right\|_{L^1(\mu)}\\
     & \le \limsup_{N\to\infty} \left\| \frac{1}{|\Phi_N|} \sum_{(n_1,n_2,m)\in \Phi_N} B(n_1,n_2,m)\right\|_{L^2(\mu)}\\
     & = \limsup_{N\to\infty} \Big\| \frac{1}{|\Phi_N|} \sum_{(n_1,n_2,m)\in \Phi_N} \prod_{i=1}^s f_i(T^{a_i(m+n_1)}x)f_i(T^{a_i(m+n_2)}x)\cdot \\
     &\quad \quad \quad \quad \quad \quad\prod_{j=-l}^l f_j^{\otimes_{II}} (T^{p_{s+1}(m+n_1+j)}x,\ldots, T^{p_d(m+n_1+j)}x)\cdot \\
     &\quad \quad \quad \quad \quad \quad \quad \quad \quad \quad f_j^{\otimes_{II}} (T^{p_{s+1}(m+n_2+j)}x,\ldots, T^{p_d(m+n_2+j)}x)\Big\|_{L^2(\mu)}.
\end{split}
\end{equation}

Recall that we assume $\HK f_i \HK_\infty=0$ for some $1\le i\le s$ or $\HK f_j^{(t)}\HK_\infty=0$ for some $s+1\le t\le d, 1\le j\le l$. By the condition $(\spadesuit)$ on $\A$, all $(n_1,n_2,m)$-polynomials $$\{p_t(m+n_1+j), p_t(m+n_2+j), s+1\le t\le d, -l \le j\le l\}$$ and $\{a_i(m+n_1), a_i(m+n_2): 1\le i\le s\}$ as polynomials of several variable $\Z^3\rightarrow \Z$ are essentially distinct.
Thus by Theorem \ref{thm-Leibman}, the last line in \eqref{z7} is equal to $0$. So by \eqref{z7},
\begin{equation*}
    \limsup_{N\to \infty} \Big\| \frac{1}{N}\sum _{n=0}^{N-1}
    \zeta_n\Big\|^2_{L^2(\widetilde{\mu}_\A ^{(\infty)})}=0.
\end{equation*}
The whole proof is complete.
\end{proof}

The final lemma we need is the following
\begin{lem}\label{lem-vdc2}
Let $(X,\X,\mu,T)$ be an ergodic m.p.s. and $d\in \N$. Let $\A= \{p_1, \cdots,  p_d \}$ be a family of non-constant integral polynomials satisfying $(\spadesuit)$. Let $\pi_\infty: (X,\X,\mu,T) \rightarrow (Z_\infty,\ZZ_\infty,\mu_\infty,T)$ be the factor map from $X$ to its $\infty$-step pro-nilfactor. Then for all $f_1,\ldots, f_s\in L^\infty(X,\mu)$, all $l \in \N$ and
all $f_j^{\otimes_{II}} = f_j^{(s+1)}\otimes f_j^{(s+2)} \otimes \cdots \otimes f_j^{(d)}$ (where $f^{(t)}_j\in L^\infty(X,\mu)$, $s+1\le t\le d, -l \le j\le l$), one has
\begin{equation}\label{b3}
\begin{split}
   & \E_{\widetilde{\mu}_\A^{(\infty)}}\Big(\big(\bigotimes_{k=1}^s f_k\big)\otimes \bigotimes_{j=-l}^{l} f_j^{\otimes_{II}} \Big|\I(X^s\times (X^{d-s})^{\Z}, \widetilde{\mu}_\A^{(\infty)},  \widetilde{\sigma})\Big) \\
   &\quad =\E_{\widetilde{\mu}_\A^{(\infty)}} \Big( \bigotimes_{k=1}^s\E_\mu(f_k | \ZZ_\infty) \otimes  \bigotimes_{j=-l}^l
    \E_\mu(f_j^{\otimes_{II}} |\ZZ_{\infty})\Big |\I(X^s\times (X^{d-s})^{\Z}, \widetilde{\mu}_\A^{(\infty)}, \widetilde{\sigma}) \Big),
\end{split}
\end{equation}
where
$$\E_\mu(f^{\otimes_{II}} _j|\ZZ_\infty) \triangleq \E_\mu(f^{(s+1)}_j|\ZZ_\infty)\otimes \E_\mu(f^{(2)}_j|\ZZ_\infty)\otimes\cdots \otimes \E_\mu(f^{(d)}_j|\ZZ_\infty),$$ $-l\le j\le l$, and $ \widetilde{\sigma}=\tau_{\vec{a}}\times \sigma_{II}$.
\end{lem}

\begin{proof}
Let $f_1,\ldots, f_s\in L^\infty(X,\mu)$, $l \in \N$ and $f^{(s+1)}_j,\cdots, f_j ^{(d)} \in L^\infty(X,\mu)$  with $-l \le j\le l$. Since $\widetilde{\mu}_\A^{(\infty)}$ is a standard measure,
$\big( \bigotimes_{k=1}^s f_k\big) \otimes \bigotimes_{j=-l}^{l} f_j^{\otimes_{II}} $ is an element of $L^\infty(X^s\times (X^{d-s})^{\Z}, \widetilde{\mu}_\A^{(\infty)})$.

By the telescoping, it suffices to show that sum
\begin{equation}\label{}
\E_{\widetilde{\mu}_\A^{(\infty)}}\Big(\big(\bigotimes_{k=1}^s f_k \big) \otimes \bigotimes_{j=-l}^l \bigotimes_{t=s+1}^d f_j^{(t)}\Big|\I(X^s\times (X^{d-s})^{\Z}, \widetilde{\mu}_\A^{(\infty)},  \widetilde{\sigma})\Big)
=0
\end{equation}
whenever $\E_\mu(f_i|\ZZ_\infty)=0$ for some $1\le i\le s$ or $\E_\mu(f_j^{(t)}|\ZZ_{\infty})=0$ for some $-l\le j\le l, s+1\le t\le d$, i.e.
$\HK f_i \HK_\infty=0$ for some $1\le i\le s$ or $\HK f_j^{(t)}\HK_\infty=0$ for some $-l\le j\le l, s+1\le t\le d$.

By the Ergodic theorem we have
\begin{equation*}
\begin{split}
&\Big\|\E_{\widetilde{\mu}_\A^{(\infty)}}\Big(\big(\bigotimes_{k=1}^s f_k \big)\otimes \bigotimes_{j=-l}^l \bigotimes_{t=s+1}^d f_j^{(t)}\Big|\I(X^s\times (X^{d-s})^{\Z}, \widetilde{\mu}_\A^{(\infty)},  \widetilde{\sigma})\Big) \Big\|_{L^2(\widetilde{\mu}_\A^{(\infty)})}
\\ &\quad = \Big\|\E_{\widetilde{\mu}_\A^{(\infty)}}\Big(\big(\bigotimes_{k=1}^s f_k  \big)\otimes \bigotimes_{j=-l}^l f_j^{\otimes_{II} }\Big|\I(X^s\times (X^{d-s})^{\Z}, \widetilde{\mu}_\A^{(\infty)},  \widetilde{\sigma})\Big) \Big\|_{L^2( \widetilde{\mu}_\A^{(\infty)})} \\
&\quad =\lim_{N\to\infty}\left \|
\frac{1}{N}\sum_{n=0}^{N-1} \Big( \big(\bigotimes_{k=1}^s f_k  \big)\otimes \bigotimes_{j=-l}^{l} f_j^{\otimes_{II}} \Big) \Big(\tau_{\vec{a}}^n(x_1,\ldots, x_s), \sigma_{II}^n {\bf x}\Big)\right\|_{L^2(\widetilde{\mu}_\A^{(\infty)})}\\
&\quad =0,
\end{split}
\end{equation*}
where the last equality holds following from Lemma \ref{lem-vdc} since $\HK f_i \HK_\infty=0$ for some $1\le i\le s$ or $\HK f_j^{(t)}\HK_\infty=0$ for some $-l\le j\le l, s+1\le t\le d$. This ends the proof.
\end{proof}

Now we are able to show

\begin{prop}\label{prop-ergodic-exten}
Let $(X,\X,\mu,T)$ be an ergodic m.p.s. and $d\in \N$. Let $\A= \{p_1, \cdots,  p_d \}$ be a family of non-constant integral polynomials satisfying $(\spadesuit)$. Let $\pi_\infty: (X,\X,\mu,T) \rightarrow (Z_\infty,\ZZ_\infty,\mu_\infty,T)$ be the factor map from $X$ to its $\infty$-step pro-nilfactor.
Then the $\sigma$-algebra $\I(X^s\times (X^{d-s})^{\Z}, \widetilde{\mu}_\A^{(\infty)},  \widetilde{\sigma})$ is measurable with
respect to $\ZZ_\infty^s\times (\ZZ_{\infty}^{d-s})^{\Z}$.
\end{prop}

\begin{proof}
Every bounded function on $X^s \times (X^{d-s})^{\Z}$ which is measurable with respect
to $\I(X^s\times (X^{d-s})^{\Z}, \widetilde{\mu}_\A^{(\infty)},  \widetilde{\sigma})$ can be approximated in
$L^2(X^s\times (X^{d-s})^\Z, \widetilde{\mu}_\A ^{(\infty)})$ by finite sums of functions of the form
$$\E_{\widetilde{\mu}_\A^{(\infty)}}\Big(\big(f_1\otimes \ldots \otimes f_s\big) \otimes \bigotimes_{j=-l}^l f_j^{\otimes_{II} }\Big|\I(X^s\times (X^{d-s})^{\Z}, \widetilde{\mu}_\A^{(\infty)},  \widetilde{\sigma} )\Big),$$ where $f_1,\ldots, f_s\in L^\infty(X, \mu)$,
$l\in \N$ and $f^{(s+1)}_j,\ldots, f_j ^{(d)} \in L^\infty(X,\mu), -l \le j\le l$.
By Lemma \ref{lem-vdc2}, one can assume that these functions are
measurable with respect to $\ZZ_{\infty}$. In this case
$\big(f_1\otimes \ldots \otimes f_s\big) \otimes \bigotimes_{j=-l}^l f_j^{\otimes_{II} }$ is measurable with respect to
$\ZZ_\infty^s\times (\ZZ_{\infty}^{d-s})^{\Z}$. Since this $\sigma$-algebra $\ZZ^s_\infty \times (\ZZ^{d-s}_{\infty})^{\Z}$ is
invariant under $\widetilde{\sigma}=\tau_{\vec{a}}\times \sigma_{II}$, we have
\begin{equation*}
  \begin{split}
   & \E_{\widetilde{\mu}_\A^{(\infty)}}\Big(\big(f_1\otimes \cdots \otimes f_s\big) \otimes \bigotimes_{j=-l}^l f_j^{\otimes_{II} }\Big|\I(X^s\times (X^{d-s})^{\Z}, \widetilde{\mu}_\A^{(\infty)},  \widetilde{\sigma} )\Big)\\
   &\quad =\lim \limits_{N\to \infty}\frac{1}{N}\sum \limits_{n=0}^{N-1}
\Big( \big(f_1\otimes \cdots \otimes f_s\big) \otimes \bigotimes_{j=-l}^l f_j^{\otimes_{II} } \Big)\circ \widetilde{\sigma}^n
   \end{split}
\end{equation*}
is also measurable with respect to
$\ZZ^s_\infty\times (\ZZ_{\infty}^{d-s})^{\Z}$. Therefore $\I(X^s\times (X^{d-s})^{\Z}, \widetilde{\mu}_\A^{(\infty)},  \widetilde{\sigma} )$ is measurable with
respect to $\ZZ_\infty^s\times (\ZZ_{\infty}^{d-s})^{\Z}$.
\end{proof}

\subsection{The $\sigma$-algebra of invariant sets of $((X^d)^\Z, (\X^d)^\Z, \mu^{(\infty)}_\A, \sigma)$}\
\medskip

Before going on, we recall that when $(Y,\Y,\nu, S)$ is a factor of $(X,\X,\mu,T)$, we denote by the same letter the $\sigma$-algebra
$\Y$ and its inverse image by $\pi$, $\pi^{-1}(\Y)$. In other words, if $(Y,\Y, \nu,
S)$ is a factor of $(X,\X, \mu, T)$, we think of $\Y$ as a
sub-$\sigma$-algebra of $\X$.

\begin{de}
Let $\pi: (X,\X,\mu,T)\rightarrow (Y,\Y,\nu,S)$ be a homomorphism. $\pi$ is {\em  ergodic}
or $(X,\X,\mu,T)$ is an {\em ergodic extension} of $(Y,\Y,\nu,S)$ if $T$-invariant sets of $\X$ is contained in $\Y$, i.e. $\I(X,\X, \mu, T)\subseteq \Y$.
\end{de}

By this definition, we can rephrase Proposition \ref{prop-ergodic-exten} as follows:

\begin{thm}\label{thm-ergodic-extension}
Let $(X,\X,\mu,T)$ be an ergodic m.p.s. and $d\in \N$. Let $\A= \{p_1, \ldots,  p_d \}$ be a family of non-constant integral polynomials satisfying $(\spadesuit)$. Let $\pi_\infty: (X,\X,\mu,T) \rightarrow (Z_\infty,\ZZ_\infty,\mu_\infty,T)$ be the factor map from $X$ to its $\infty$-step pro-nilfactor. Then the factor map
\begin{equation*}
  \begin{split}
\pi_\infty^{(s)}\times \pi_\infty^\infty:\ &(X^s\times (X^{d-s})^\Z, \X^s\times (\X^{d-s})^\Z, \widetilde{\mu}^{(\infty)}_\A, \widetilde{\sigma}) \\ &\quad \quad \quad \quad \longrightarrow
(Z_{\infty}^s\times ({Z_\infty^{d-s}})^\Z, \ZZ_{\infty}^s\times (\ZZ_{\infty}^{d-s})^\Z, \widetilde{(\mu_\infty)}^{(\infty)}_\A, \widetilde{\sigma})
\end{split}
\end{equation*}
is ergodic.

In particular, one has that the $\sigma$-algebra $\I(X^s\times (X^{d-s})^{\Z}, \widetilde{\mu}_\A^{(\infty)}, \widetilde{\sigma})$ is equal to (modulo isomorphisms) the $\sigma$-algebra $\I(Z_{\infty}^s\times ({Z_\infty^{d-s}})^\Z, (\widetilde{\mu_\infty})_\A^{(\infty)},\widetilde{\sigma})$.
\end{thm}

\begin{proof}
By Proposition \ref{prop-ergodic-exten}, the $\sigma$-algebra $\I(X^s\times (X^{d-s})^{\Z}, \widetilde{\mu}_\A^{(\infty)},  \widetilde{\sigma})$ is measurable with
respect to $\ZZ_\infty^s\times (\ZZ_{\infty}^{d-s})^{\Z}$. Thus, it means that $\pi_\infty^{(s)}\times \pi_\infty^\infty$ is ergodic by the definition.

Since $(Z_{\infty}^s\times ({Z_\infty^{d-s}})^\Z, \ZZ_{\infty}^s\times (\ZZ_{\infty}^{d-s})^\Z, \widetilde{(\mu_\infty)}^{(\infty)}_\A, \widetilde{\sigma})$ is the factor of
$$(X^s\times (X^{d-s})^\Z, \X^s\times (\X^{d-s})^\Z, \widetilde{\mu}^{(\infty)}_\A, \widetilde{\sigma}),$$ $\I(Z_{\infty}^s\times ({Z_\infty^{d-s}})^\Z, (\widetilde{\mu_\infty})_\A^{(\infty)},\widetilde{\sigma})$ is contained in
$\I(X^s\times (X^{d-s})^{\Z}, \widetilde{\mu}_\A^{(\infty)}, \widetilde{\sigma})$.
On the other hand, as $$\I(X^s\times (X^{d-s})^{\Z}, \widetilde{\mu}_\A^{(\infty)},  \widetilde{\sigma})\subseteq \ZZ_\infty^s\times (\ZZ_{\infty}^{d-s})^{\Z},$$ we have that $\I(X^s\times (X^{d-s})^{\Z}, \widetilde{\mu}_\A^{(\infty)}, \widetilde{\sigma})$ is contained in $\I(Z_{\infty}^s\times ({Z_\infty^{d-s}})^\Z, (\widetilde{\mu_\infty})_\A^{(\infty)},\widetilde{\sigma})$.
\end{proof}

\begin{thm}\label{thm-measure-over-nil}
Let $(X,\X,\mu,T)$ be an ergodic m.p.s. and $d\in \N$. Let $\A= \{p_1,  p_2, \cdots,  p_d \}$ be a family of non-constant integral polynomials satisfying $(\spadesuit)$.
Let $\pi_\infty: (X,\X,\mu,T) \rightarrow (Z_\infty,\ZZ_\infty,\mu_\infty,T)$ be the factor map from $X$ to its $\infty$-step pro-nilfactor, and let $\mu=\int _{Z_{\infty}} \nu_z \ d\mu_{\infty}(z)$ be the disintegration of $\mu$ over $\mu_{\infty}$.
Then $\widetilde{\mu}_\A^{(\infty)}$ is the conditional product measure with
respect to $(\widetilde{\mu_\infty})_\A^{(\infty)}$.
That is,
$$\widetilde{\mu}_\A ^{(\infty)}=\int_{Z^s_\infty\times(Z_\infty^{d-s})^{\Z}}\big(\nu_{z_1}\times \cdots \times \nu_{z_s}\big)\times \prod_{n\in \Z} \Big( \nu_{z_n^{(s+1)}}\times \cdots \times \nu_{z_n^{(d)}} \Big) d (\widetilde{\mu_\infty})_\A ^{(\infty)} ((z_1,\ldots, z_s),{\bf z}),$$
where $\big((z_1,\ldots, z_s),{\bf z}\big)
=\big((z_1,\ldots, z_s),(z_n^{(s+1)},\ldots, z_n^{(d)})_{n\in \Z}\big)\in Z^s_\infty\times (Z_\infty^{d-s})^{\Z}$.
\end{thm}

\begin{proof}
We need to show that  for
all $f_1,\ldots, f_s\in L^\infty(X,\mu)$, all $l \in \N$ and
all $f_j^{\otimes_{II}} = f_j^{(s+1)}\otimes f_j^{(s+2)} \otimes \cdots \otimes f_j^{(d)}$ (where $f^{(t)}_j\in L^\infty(X,\mu)$, $s+1\le t\le d, -l \le j\le l$), one has
\begin{equation*}
  \begin{split}
     & \int_{X^s\times (X^{d-s})^\Z} \big(f_1\otimes \cdots \otimes f_s\big) \otimes \bigotimes_{j=-l}^l f_j^{\otimes_{II} } d \widetilde{\mu}_\A ^{(\infty)}\\
     &= \int_{Z^s_\infty\times(Z_\infty^{d-s})^{\Z}} \E_\mu(f_1 | \ZZ_\infty) \otimes \cdots \otimes \E_\mu (f_s | \ZZ_{\infty}) \otimes  \bigotimes_{j=-l}^l
    \E_\mu(f_j^{\otimes_{II}} |\ZZ_{\infty}) d (\widetilde{\mu_\infty})_\A ^{(\infty)},
   \end{split}
\end{equation*}
where $$\E_\mu(f^{\otimes_{II}} _j|\ZZ_\infty)=\E_\mu(f^{(s+1)}_j|\ZZ_\infty)\otimes \E_\mu(f^{(2)}_j|\ZZ_\infty)\otimes\cdots \otimes \E_\mu(f^{(d)}_j|\ZZ_\infty),$$
$ -l\le j\le l$. But this follows from Lemma \ref{lem-vdc2} by taking integrals both sides of \eqref{b3}.
\end{proof}

\medskip
Since by \eqref{a9} we have the isomorphism
$$\phi: ((X^d)^\Z, (\X^d)^\Z, \mu^{(\infty)}_\A, \sigma)\cong (X^s\times (X^{d-s})^\Z, \X^s\times (\X^{d-s})^\Z, \widetilde{\mu}^{(\infty)}_\A, \widetilde{\sigma}),$$
we have the following commuting diagram
\begin{equation*}
  \xymatrix{
 ((X^d)^\Z, (\X^d)^\Z, \mu^{(\infty)}_\A, \sigma) \ar[d]_{\pi_\infty^\infty} \ar[r]_{}
                & (X^s\times (X^{d-s})^\Z, \X^s\times (\X^{d-s})^\Z, \widetilde{\mu}^{(\infty)}_\A, \widetilde{\sigma}) \ar[d]^{\pi_\infty^{(s)}\times \pi_\infty^\infty} \\
  ((Z_{\infty}^d)^{\Z}, (\ZZ_\infty^d)^{\Z}, (\mu_\infty)_\A^{(\infty)},\sigma)\ar[r]^{}
                & (Z_{\infty}^s\times ({Z_\infty^{d-s}})^\Z, \ZZ_{\infty}^s\times (\ZZ_{\infty}^{d-s})^\Z, \widetilde{(\mu_\infty)}^{(\infty)}_\A, \widetilde{\sigma})  . }
\end{equation*}
By the diagram, we can replace $\widetilde{\mu}^{(\infty)}_\A$ in results above by $\mu^{(\infty)}_\A$. Thus, by Theorem \ref{thm-ergodic-extension} we have

\begin{thm}
Let $(X,\X,\mu,T)$ be an ergodic m.p.s. and $d\in \N$. Let $\A= \{p_1,  p_2, \cdots,  p_d \}$ be a family of non-constant integral polynomials satisfying $(\spadesuit)$. Let $\pi_\infty: (X,\X,\mu,T) \rightarrow (Z_\infty,\ZZ_\infty,\mu_\infty,T)$ be the factor map from $X$ to its $\infty$-step pro-nilfactor.  Then the factor map
$$\pi_{\infty}^\infty: ((X^d)^{\Z}, (\X^d)^{\Z}, \mu_\A^{(\infty)}, \sigma) \rightarrow
((Z_{\infty}^d)^{\Z}, (\ZZ_\infty^d)^{\Z}, (\mu_\infty)_\A^{(\infty)}, \sigma)$$
is ergodic.

In particular, one has that the $\sigma$-algebra $\I((X^d)^{\Z}, (\X^d)^{\Z}, \mu_\A^{(\infty)}, \sigma)$ is equal to (modulo isomorphism) the $\sigma$-algebra $\I((Z_{\infty}^d)^{\Z}, (\ZZ_\infty^d)^{\Z}, (\mu_\infty)_\A^{(\infty)},\sigma)$.
\end{thm}

\section{Polynomial orbits in nilsystems}\label{section-nilsystem}

Now we work on pro-nilsystems. All results remain true for $\infty$-step pro-nilsystems.
Since $((X^d)^\Z, (\X^d)^\Z, \mu^{(\infty)}_\A, \sigma) \cong (X^s\times (X^{d-s})^\Z, \X^s\times (\X^{d-s})^\Z, \widetilde{\mu}^{(\infty)}_\A, \widetilde{\sigma}),$ we only give statements concering $\mu^{(\infty)}_\A$, and the corresponding results for $ \widetilde{\mu}^{(\infty)}_\A$ can be proved similarly. 

\subsection{The measure $\mu_\A^{(\infty)}$ on nilsystems}\

\subsubsection{}
Let $(X = G/\Gamma, \nu , T )$ be a nilsystem with $T$ the translation by an element $\tau \in G$.
Then $(X,T)$ is a t.d.s. By Theorem \ref{thm-ParryLeibman}, $(X, T )$ is uniquely ergodic if and only if $(X, \nu , T )$ is ergodic if and only if $(X, T )$ is minimal.
This result holds for pro-nilsystems.

\subsubsection{}
Now let $(X, T)$ be a minimal topological pro-nilsystem (a minimal pro-nilsystem for short).
Let $x\in X$ and $d\in \N$. Let $\A= \{p_1, \cdots,  p_d \}$ be a family of non-constant essentially distinct integral polynomials with $p_1(0)=\ldots =p_d(0)=0$. Recall that
\begin{equation}\label{}
  \omega_x^\A=\Big((\tau^{p_1(n)}x, \tau^{p_2(n)}x, \ldots, \tau^{p_d(n)}x )\Big)_{n\in \Z}\in (X^d)^{\Z}.
\end{equation}
Note that  $(\sigma\omega^\A_x)(n)= (\tau^{p_1(n+1)}x, \tau^{p_2(n+1)}x, \ldots, \tau^{p_d(n+1)}x )$ for $n\in \Z$.
Recall that
\begin{equation}\label{}
  N_\infty(X,\A)=\overline{\bigcup\{\sigma^n\omega^\A_x:n\in \Z, x\in X\}}\subseteq (X^d)^{\Z}.
\end{equation}

In \cite{HSY22-1} we have shown the following result.

\begin{thm}\cite[Theorem 5.5]{HSY22-1}\label{thm-nil-minimal}
Let $(X, T)$ be a minimal  pro-nilsystem.
Let $\A= \{p_1, \ldots,  p_d \}$ be a family of non-constant essentially distinct integral polynomials with $p_1(0)=\cdots =p_d(0)=0$. Then we have
\begin{enumerate}
  \item The system $(N_\infty(X,\A), \langle T^{\infty},\sigma \rangle)$ is a minimal pro-nilsystem.

  \item For each $x\in X$, the system $(\overline{\O}(\omega_x^\A, \sigma),\sigma)$ is a
minimal  pro-nilsystem.
\end{enumerate}
\end{thm}

Now we prove the following

\begin{thm}\label{thm-nil-ergodic}
Let $(X, T)$ be a minimal  pro-nilsystem. Let $\A = \{p_1, \cdots,  p_d \} $ be a family of non-constant essentially distinct integral polynomials. With the notations above, we have
\begin{enumerate}
  \item $(N_\infty(X,\A), \langle T^{\infty},\sigma \rangle)$ is uniquely
  ergodic with the $\infty$-joining $\nu^{(\infty)}_{\A}$.

  \item For each $x\in X$,  $(\overline{\O}(\omega^\A_x, \sigma),\sigma)$ is
uniquely ergodic with some measure $\nu^{(\infty)}_{\A, x}$.

  \item $\displaystyle \nu^{(\infty)}_{\A} = \int_X \nu^{(\infty)}_{\A, x}\
  d\nu(x)$.

  \item For all $l \in \N$ and all $f_j^\otimes = f_j^{(1)}\otimes f_j^{(2)} \otimes \cdots \otimes f_j^{(d)}$ (where $ f^{(i)}_j\in L^\infty(X,\mu), 1\le i\le d,  -l \le j\le l$)
\begin{equation}\label{}
\begin{split}
 & \lim_{N\to\infty}\frac{1}{N} \sum_{n=0}^{N-1}
   \prod_{j=-l}^l f_j^\otimes (T^{\vec{p}(n+j)}x^{\otimes d})
  = \int_{(X^d)^{\Z}} \Big( \bigotimes_{j=-l}^l f_j^\otimes \Big)({\bf x}) \ d \nu_{\A, x}^{(\infty)}({\bf x}),
\end{split}
\end{equation}
for $\nu$ a.e. $x\in X$. 
\end{enumerate}
\end{thm}

\begin{proof}
By Theorem \ref{thm-nil-minimal},  $(N_\infty(X,\A), \langle T^{\infty},\sigma \rangle)$ is a minimal pro-nilsystem, and for each $x\in X$,
so is the system $(\overline{\O}(\omega_x^\A, \sigma),\sigma)$. 
It follows by Theorem \ref{thm-ParryLeibman} that they are uniquely ergodic. We denote their unique measures by $\nu_{\A}^{(\infty)}$ and $\nu_{\A,x}^{(\infty)}$ respectively.

\medskip

Now we show that (3) holds, that is,  $\displaystyle \nu^{(\infty)}_{\A} = \int_X \nu^{(\infty)}_{\A,x}\ d\nu(x)$.
Let
$$\varphi: X\rightarrow {\mathcal M}((X^d)^\Z), \ x\mapsto \nu^{(\infty)}_{\A,x}.$$
Since the map
$$X\rightarrow {\mathcal M}((X^d)^\Z), \ x\mapsto \frac {1}{N} \sum_{n=0}^{N-1} (\cdots \times T^{\vec{p}(n-1)}\times \underset{\bullet}{T^{\vec{p}(n)}}\times T^{\vec{p}(n+1)}\times \cdots)_* \d_{x^{(\infty)}}$$
is continuous (here $x^{(\infty)}=(\ldots, x^{\otimes d},x^{\otimes d}, \ldots)$) and
$$\displaystyle \frac {1}{N} \sum_{n=0}^{N-1} (\cdots \times T^{\vec{p}(n-1)}\times \underset{\bullet}{T^{\vec{p}(n)}}\times T^{\vec{p}(n+1)}\times \cdots)_* \d_{x^{(\infty)}}$$ converges to $\nu^{(\infty)}_{\A,x}$ weakly $\big($ as $(\overline{\O}(\omega_x^\A,\sigma),\nu_{\A,x}^{(\infty)},\sigma)$ is uniquely ergodic $\big)$, we have that $\varphi$ is a Borel map. Thus we can define a measure $\displaystyle \kappa = \int_X \nu^{(\infty)}_{\A,x}\ d\nu(x)$. It is easy to check that $\kappa$ is $\langle T^{\infty},\sigma \rangle$-invariant. By (1), the system $(N_\infty(X,\A), \langle T^{\infty},\sigma \rangle)$ is uniquely ergodic with the $\infty$-joining $\nu^{(\infty)}_{\A}$, and hence by the uniqueness
we have
$$\nu^{(\infty)}_{\A} =\kappa= \int_X \nu^{(\infty)}_{\A,x}\ d\nu(x).$$

Since for each $x\in X$,  $(\overline{\O}(\omega^\A_x, \sigma),\sigma)$ is
uniquely ergodic with the measure $\nu^{(\infty)}_{\A,x}$, we have that
for all $F\in C((X^d)^{\Z})$,
$$\lim_{n\to\infty} \frac{1}{N} \sum_{n=0}^{N-1} F(\sigma^n(\omega_x^\A))
      = \int_{(X^d)^{\Z}} F({\bf x})\
   d \nu^{(\infty)}_{\A, x}({\bf x} ),\ \ \text{where}\ {\bf x}=({\bf x}_n)_{n\in \Z}\in (X^d)^{\Z}.$$
That is, for all $l \in \N$ and all $f_j^\otimes = f_j^{(1)}\otimes f_j^{(2)} \otimes \cdots \otimes f_j^{(d)}$ (where $ f^{(i)}_j\in C(X), 1\le i\le d,  -l \le j\le l$)
\begin{equation*}
\begin{split}
 &\lim_{N\to\infty}\frac{1}{N} \sum_{n=0}^{N-1}
   \prod_{j=-l}^l f_j^\otimes (T^{\vec{p}(n+j)}x^{\otimes d}) \\ &= \lim_{N\to\infty} \frac{1}{N}\sum_{n=0}^{N-1}\prod_{j=-l}^l f_j^{(1)}(T^{p_1(n+j)}x) f_j^{(2)} (T^{p_2(n+j)}x) \cdots f_j^{(d)}(T^{p_d(n+j)}x)  \\
  &= \int_{(X^d)^{\Z}} \Big( \bigotimes_{j=-l}^l f_j^\otimes \Big)({\bf x}) \ d \nu_{\A, x}^{(\infty)}({\bf x}).
\end{split}
\end{equation*}
Similar to the proof of Theorem \ref{thm-A2} in the appendix, we may pass to bounded measurable functions.
The whole proof is now complete.
\end{proof}

\begin{rem}
Note that though $\nu^{(\infty)}_{\A}$ is standard,  for $x\in X$ the measure $\nu^{(\infty)}_{\A,x}$ may not be standard.
\end{rem}

\subsubsection{}
Similarly, we have a  result related to $\widetilde{\nu}^{(\infty)}_{\A}$

\begin{thm}\label{thm-nil-ergodic2}
Let $(X, T)$ be a minimal  pro-nilsystem. Let $\A$ be a family of non-constant essentially distinct integral polynomials. With the notations above, we have
\begin{enumerate}
  \item $(M_\infty(X,\A), \langle T^{\infty},\widetilde{\sigma} \rangle)$ is uniquely
  ergodic with the $\infty$-joining $\widetilde{\nu}^{(\infty)}_{\A}$.

  \item For each $x\in X$,  $(\overline{\O}(\xi^\A_x, \widetilde{\sigma}),\widetilde{\sigma})$ is
uniquely ergodic with some measure $\widetilde{\nu}^{(\infty)}_{\A,x}$.

  \item $\displaystyle \widetilde{\nu}^{(\infty)}_{\A} = \int_X\widetilde{ \nu}^{(\infty)}_{\A,x}\
  d\nu(x)$.

  \item For  all $f_1,\ldots, f_s\in L^\infty(X,\mu)$, all $l \in \N$ and all $f_j^{\otimes_{II}} = f_j^{(s+1)}\otimes f_j^{(s+2)} \otimes \cdots \otimes f_j^{(d)}$ (where $f^{(t)}_j\in L^\infty(X,\mu)$, $s+1\le t\le d, -l \le j\le l$),
\begin{equation}
\begin{split}
  &  \lim_{N\rightarrow +\infty} \frac{1}{N}\sum_{n=0}^{N-1} \prod_{i=1}^s f_i(T^{a_in}x) \prod_{j=-l}^l f_j^{\otimes_{II}} (T^{{p_{s+1}}(n+j)}x, \ldots, T^{p_d(n+j)}x)  \\
  &\quad =\int_{X^s\times (X^{d-s})^{\Z}} \big(\bigotimes_{i=1}^sf_i \big) \otimes  \Big( \bigotimes_{j=-l}^l f_j^{\otimes_{II}} \Big)((x_1,\ldots, x_s),{\bf x})d\widetilde{\nu}_{\A,x}^{(\infty)},
\end{split}
\end{equation}
for $\nu$ a.e. $x\in X$,
where $((x_1,\ldots, x_s),{\bf x}) =((x_1,\ldots, x_s),({\bf x}_n)_{n\in \Z} )\in X^s\times (X^{d-s})^\Z$.

\end{enumerate}
\end{thm}


\subsection{The ergodic decomposition}\
\medskip

In Theorem \ref{thm-nil-ergodic}, we have proved that $\displaystyle \nu^{(\infty)}_{\A} = \int_X \nu^{(\infty)}_{\A,x}\ d\nu(x)$ and $\nu^{(\infty)}_{\A,x}$ is ergodic under $\sigma$ for each $x\in X$. In the rest of the section we study the ergodic decomposition of  $\nu^{(\infty)}_{\A}$ under $\sigma$.
First we recall some basic facts concerning the ergodic decomposition.

\begin{thm}[The ergodic decomposition]\cite[Theorem 8.7]{Glasner}\label{ergodic-decomposition}
Let $(X,\X,\mu,T)$ be a m.p.s. and $\I=\I(X,\X,\mu,T)$ be the $\sigma$-algebra of $T$-invariant sets. Let $\pi: (X,\X,\mu, T)\rightarrow (W, {\mathcal W},\varsigma, T)$ be the factor defined by $\I$ so that $\I=\pi^{-1}({\mathcal W})$. Let $ \mu =\int_W \mu_w d\varsigma(w)$ be the disintegration of $\mu$ over $\varsigma$ and for every $w\in W$ denote $X_w=\pi^{-1}(w)$ and $\X_w=\X\cap X_w$, then:
\begin{enumerate}
  \item $T$ acts as the identity $\id_W$ on $W$.
  \item For $\varsigma$ a.e. $w\in W$, the system $(X_w,\X_w,\varsigma_w, T|_{X_w})$ is ergodic.
  \item This decomposition is unique in the following sense:

If $(Z,\ZZ,\vartheta)$ is a probability space and $\psi: z\mapsto \widetilde{\mu}_z$ a measurable map from $Z$ into ${\mathcal M}_T^{erg}(X)$ such that $ \mu =\int_Z\widetilde{\mu}_z d\vartheta(z)$, then there exists a measurable map $\phi: (Z,\ZZ,\vartheta)\rightarrow (W,{\mathcal W}, \varsigma)$ such that $\vartheta$-a.e. $z\in Z$, $\mu_{\phi(z)}=\widetilde{\mu}_z$.
\end{enumerate}
\end{thm}

\begin{rem} \label{rem-6.7} (1)
For a m.p.s. $(X,\X,\mu,T)$, a factor $(Y,\Y,\nu, T)$ with a factor map $\pi: X\rightarrow Y$ can be viewed as the $T$-invariant subalgebra $\pi^{-1}(\Y)\subset \X$. One can also retrieve the factor $(Y,\Y,\nu,T)$ as a measure-preserving transformation on the space ${\mathcal M}(X)$ of Borel probability measures on $X$ as follows. Disintegrate the measure $\mu$ over $\nu$, $\mu= \int_Y \mu_y d \nu(y)$ and observe that the $T$-invariance of $\mu$ implies that $T_* \mu_y=\mu_{Ty}$ for $\nu$-a.e. $y\in Y$. Let $\phi: Y\rightarrow {\mathcal M}(X), \ y\mapsto \mu_y$ and let $\kappa=\phi_*(\nu)$. Then $\phi: (Y,\Y,\nu,T)\rightarrow ({\mathcal M}(X), \kappa, T_*)$ is an isomorphism. In this case
$$\mu=\int_Y \mu_y d \nu(y)=\int_{{\mathcal M}(X)}\theta d\kappa (\theta),$$
which means for all $f\in C(X)$,
$$\int_{{\mathcal M}(X)}\left(\int_Xf(x)d\theta(x) \right) d \kappa(\theta)=\int_X f(x) d\mu(x).$$
Thus the trivial one point factor is $(\{\mu\},\d_\mu)$, and
$(X,\mu,T)\cong \Big( \{\d_x: x\in X\}, \mu, T_* \Big ).$
Refer to \cite[Chapter 8, Section 1]{Glasner} for details.

(2) Now go back to Theorem \ref{ergodic-decomposition}. For a t.d.s. $(X,T)$, ${\mathcal M}_T(X)$ is a simplex and by Choquet's theorem, for each $\mu\in M_T(X)$, there is a unique representation $$\mu=\int_{{\mathcal M}^{erg}_T(X)}\theta d\kappa(\theta),$$
with $\kappa$ a Borel probability measure on the $G_\delta$ subset ${\mathcal M}^{erg}_T(X)$ of ${\mathcal M}_T(X)$. And $(W,{\mathcal W},\varsigma,\id)$ is isomorphic to
$({\mathcal M}(X),\kappa, T_*)=\Big( \{\mu_w:w\in W\},\kappa, T_* \Big ).$
See the proof of \cite[Theorem 8.7]{Glasner} for details.
\end{rem}


\subsection{The ergodic decomposition of $\nu_\A^{(\infty)}$ under $\sigma$}\

\subsubsection{}
Now consider a minimal  pro-nilsystem $(X, T)$. Let $\A$ be a family of non-constant essentially distinct integral polynomials.
By Theorem \ref{thm-nil-ergodic}, $(N_\infty(X,\A), \nu_\A^{(\infty)}, \sigma)$ is a m.p.s.

Now we apply Theorem \ref{ergodic-decomposition} to the system $(N_\infty(X,\A), \nu_\A^{(\infty)}, \sigma)$. Let $(W,\mathcal{W}, \varsigma)$ be the factor of $(N_\infty(X,\A), \nu_\A^{(\infty)}, \sigma)$ defined by $\I(N_\infty(X,\A), \nu_\A^{(\infty)}, \sigma)$, and let $\psi: (N_\infty(X,\A), \nu_\A^{(\infty)}, \sigma)\rightarrow (W,\mathcal{W}, \varsigma,\id)$ be the corresponding factor map. Let
$$\displaystyle \nu_\A^{(\infty)} =\int_W [\nu_\A^{(\infty)}]_w d\varsigma(w)$$ be the disintegration of $\nu_\A^{(\infty)}$ over $\varsigma$,  and it is the ergodic decomposition of $\nu_\A^{(\infty)}$ under $\sigma$.

By Remark \ref{rem-6.7}, there is some measure $\kappa$ on ${\mathcal M} (N_\infty(X,\A))$ such that
$$(W,\varsigma)\cong \Big( {\mathcal M}(N_\infty(X,\A)), \kappa \Big)=\Big( \{[\nu_\A^{(\infty)}]_w: w\in W\},\kappa \Big),$$
and
\begin{equation}\label{e3}
  \nu_\A^{(\infty)} =\int_W [\nu_\A^{(\infty)}]_w d\varsigma(w)=\int_{{\mathcal M}_\sigma^{erg}(N_\infty(X,\A))}\theta d\kappa(\theta).
\end{equation}

In the proof of Theorem \ref{thm-nil-ergodic}, we have defined a Borel measurable map:
$$\varphi: X\rightarrow {\mathcal M}^{erg}_\sigma (N_\infty(X,\A)), \ x\mapsto \nu^{(\infty)}_{\A,x}.$$
Let $\widetilde{\kappa}=\varphi_* (\nu)$. Then combining with Theorem \ref{thm-nil-ergodic}, we have
\begin{equation}\label{e4}
  \nu^{(\infty)}_{\A} = \int_X \nu^{(\infty)}_{\A,x}\
  d\nu(x)=\int_{{\mathcal M}_\sigma^{erg}(N_\infty(X,\A))}\theta d\widetilde{\kappa}(\theta).
\end{equation}
By \eqref{e3}, \eqref{e4} and the uniqueness of the representation in Choquet's theorem, we get $\widetilde{\kappa}=\kappa$ which implies that  $\nu$-a.e. $x\in X$, $\varphi(x)\in \{[\nu_\A^{(\infty)}]_w: w\in W\}$. Thus we have the following Borel measurable map
$$\zeta: X\rightarrow W, \ x\mapsto \varphi(x)=[\nu_\A^{(\infty)}]_w\mapsto w.$$

To sum up, we have the following result:

\begin{thm}
Let $(X, \nu , T )$ be a minimal pro-nilsystem, and $\A$ be a family of non-constant essentially distinct integral polynomials. Let $(W,\mathcal{W}, \varsigma)$ be the factor of $(N_\infty(X,\A), \nu_\A^{(\infty)}, \sigma)$ defined by $\I(N_\infty(X,\A), \nu_\A^{(\infty)}, \sigma)$. Then there is a measure preserving map $\zeta: (X,\X,\nu)\rightarrow (W,\mathcal{W},\varsigma)$ such that
$$ \nu_\A^{(\infty)} =\int_W [\nu_{\A}^{(\infty)}\big ]_w d\varsigma(w)$$
is the ergodic decomposition of $\nu_\A^{(\infty)}$ under $\sigma$,
where $\big[ \nu_{\A}^{(\infty)} \big ]_{\zeta(x)}=\nu_{\A, x}^{(\infty)}, \ \mu \ a.e. \ x\in X.$
\end{thm}

\subsubsection{}
We have a similar result concerning $\widetilde{\nu}_\A^{(\infty)}$.

\begin{thm}\label{thm-6.8}
Let $(X, \nu , T )$ be a minimal  pro-nilsystem, and $\A$ be a family of non-constant essentially distinct integral polynomials. Let $(W,\mathcal{W}, \varsigma)$ be the factor of $(M_\infty(X,\A),\widetilde{ \nu}_\A^{(\infty)}, \widetilde{\sigma})$ defined by $\I(M_\infty(X,\A),\widetilde{ \nu}_\A^{(\infty)}, \widetilde{\sigma})$. Then there is a measure preserving map $\zeta: (X,\X,\nu)\rightarrow (W,\mathcal{W},\varsigma)$ such that
$$ \widetilde{\nu}_\A^{(\infty)} =\int_W [\widetilde{\nu}_{\A}^{(\infty)}\big ]_w d\varsigma(w)$$
is the ergodic decomposition of $\widetilde{\nu}_\A^{(\infty)}$ under $\widetilde{\sigma}$,
where $\big[ \widetilde{\nu}_{\A}^{(\infty)} \big ]_{\zeta(x)}=\widetilde{\nu}_{\A, x}^{(\infty)}, \ \mu \ a.e. \ x\in X.$
\end{thm}

\section{The ergodic decomposition of $\mu^{(\infty)}_\A$}\label{section-ergodic-decom}

In this section, using the results in the previous sections, we show that for each family $\A=\{p_1,\ldots,p_d\}$ of non-constant integral polynomials satisfying $(\spadesuit)$ and an ergodic m.p.s. $(X,\X,\mu,T)$, the m.p.s.
$((X^d)^\Z,(\X^d)^\Z, \mu_\A^{(\infty)}, \langle T^\infty, \sigma\rangle)$
is ergodic and give the ergodic decomposition theorem of $\mu_\A^{(\infty)}$ under $\sigma$. Moreover, 
we show that when each element in $\A$ has degree greater than 1, then  $((X^d)^\Z, (\X^d)^\Z, {\mu}^{(\infty)}_{\A,x}, \sigma)$ is the product of an $\infty$-step pro-nilsystem $((Z_\infty^d)^\Z, (\ZZ_\infty ^d)^\Z, {(\mu_\infty)}^{(\infty)}_{\A,x}, \sigma)$ and a Bernoulli system
for $\mu$-a.e. $x\in X$.


\subsection{The ergodic decomposition theorem of $\mu_\A^{(\infty)}$ under $\sigma$}

\subsubsection{}
The main result of this section is the following one.

\begin{thm}\label{thm-ergodic-measures}
Let $(X,\X,\mu,T)$ be an ergodic m.p.s. and $d\in \N$. Let $\A= \{p_1, \ldots, p_{d}\}$ be a family of non-constant integral polynomials satisfying $(\spadesuit)$ with $p_i(n)=a_in$, $a_i\in\Z$ and $0\le i\le s$; and $deg(p_j)\ge 2$, $s+1\le j\le d$.  Then there exists a family $\big\{\widetilde{\mu}^{(\infty)}_{\A,x}\big\}_{x\in X}$ of Borel probability measures on $(X^d)^{\Z}$ such that
\begin{enumerate}
  \item The $\Z^2$-m.p.s. $(X^s\times (X^{d-s})^\Z,\X^s\times (\X^{d-s})^\Z, \widetilde{\mu}_\A^{(\infty)}, \langle T^\infty, \widetilde{\sigma}\rangle)$ is ergodic.

  \item The disintegrations of $\widetilde{\mu}_\A^{(\infty)}$ over $\mu$ is given by
$\widetilde{\mu}_\A ^{(\infty)}=\int_X \widetilde{\mu}_{\A, x}^{(\infty)} d\mu(x), $
and for $\mu$ a.e. $x\in X$, $\widetilde{\mu}^{(\infty)}_{\A, x}$ is ergodic under $\widetilde{\sigma}$;

  \item For all $f_1,f_2,\ldots, f_s\in L^\infty(X,\mu)$, all $l \in \N$ and all $f_j^{\otimes_{II}} = f_j^{(s+1)}\otimes f_j^{(s+2)} \otimes \cdots \otimes f_j^{(d)}$ (where $f^{(t)}_j\in L^\infty(X,\mu)$, $s+1\le t\le d, -l \le j\le l$),
\begin{equation}\label{ss}
\begin{split}
  &\lim_{N\rightarrow +\infty} \frac{1}{N}\sum_{n=0}^{N-1} f_1(T^{a_1n}x)\cdots f_s(T^{a_sn}x) \prod_{j=-l}^l f_j^{\otimes_{II}} (T^{{p_{s+1}}(n+j)}x, \ldots, T^{p_d(n+j)}x) \\
  &\quad  =\int_{X^s\times (X^{d-s})^{\Z}} \big(\bigotimes_{i=1}^s f_i\big) \otimes  \Big( \bigotimes_{j=-l}^l f_j^{\otimes_{II}} \Big)((x_1,\ldots, x_s),{\bf x})d\widetilde{\mu}_{\A,x}^{(\infty)}
\end{split}
\end{equation}
in $L^2(X,\mu)$, where $((x_1,\ldots, x_s),{\bf x}) =((x_1,\ldots, x_s),({\bf x}_n)_{n\in \Z} )\in X^s\times (X^{d-s})^\Z$.
\end{enumerate}
\end{thm}

\begin{proof}
We divide the proof into several steps.

\medskip
\noindent{\em Step 0. The results of Theorem \ref{thm-ergodic-measures} are independent of the choice of the topological models of $(X,\X,\mu,T)$.}
\medskip

First we show that Theorem \ref{thm-ergodic-measures} is independent of the topological models of $(X,\X,\mu,T)$ which we choose. So, we need to show that all terms in Theorem \ref{thm-ergodic-measures} are isomorphic invariant.

Let $(X_i,\X_i,\mu_i,T_i), i=1,2$, be two isomorphic ergodic m.p.s. and $\phi: (X_1,\X_1,\mu_1,T)\rightarrow (X_2,\X_2,\mu_2,T)$
be the isomorphism. We assume that all results hold for $(X_1,\X_1,\mu_1,T)$. That is, there exists a family $\big\{(\widetilde{\mu_1})^{(\infty)}_{\A,x}\big\}_{x\in X_1}$ of Borel probability measures on $(X_1^d)^{\Z}$ such that
\begin{enumerate}
  \item[(1.1)] $(X_1^s\times (X_1^{d-s})^\Z,\X_1^s\times (\X_1^{d-s})^\Z, (\widetilde{\mu_1})_\A^{(\infty)}, \langle T^\infty, \widetilde{\sigma}\rangle)$ is ergodic.

  \item[(1.2)] The disintegrations of $(\widetilde{\mu_1})_\A^{(\infty)}$ over $\mu_1$ is as follows
$$(\widetilde{\mu_1})_\A ^{(\infty)}=\int_{X_1} (\widetilde{\mu_1})_{\A, x}^{(\infty)} d\mu_1(x),$$
and for $\mu_1$ a.e. $x\in X_1$, $(\widetilde{\mu_1})^{(\infty)}_{\A, x}$ is ergodic under $\widetilde{\sigma}$;

  \item[(1.3)] For all $f_1,f_2,\ldots, f_s\in L^\infty(X_1,\mu_1)$,  all $l \in \N$ and all $f_j^{\otimes_{II}} = f_j^{(s+1)}\otimes f_j^{(s+2)} \otimes \cdots \otimes f_j^{(d)}$ (where $f^{(t)}_j\in L^\infty(X_1,\mu_1)$, $s+1\le t\le d, -l \le j\le l$),
\begin{equation*}\label{}
\begin{split}
 & \lim_{N\rightarrow +\infty} \frac{1}{N}\sum_{n=0}^{N-1} f_1(T^{a_1n}x)\cdots f_s(T^{a_sn}x) \prod_{j=-l}^l f_j^{\otimes_{II}} (T^{{p_{s+1}}(n+j)}x, \ldots, T^{p_d(n+j)}x) \\
  &\quad =\int_{X_1^s\times (X_1^{d-s})^{\Z}} \big(f_1\otimes f_2\otimes \cdots \otimes f_s\big) \otimes  \Big( \bigotimes_{j=-l}^l f_j^{\otimes_{II}} \Big)((x_1,\ldots, x_s),{\bf x})d(\widetilde{\mu_1})_{\A,x}^{(\infty)}
\end{split}
\end{equation*}
in $L^2(X_1,\mu_1)$, where $((x_1,\ldots, x_s),{\bf x}) =((x_1,\ldots, x_s),({\bf x}_n)_{n\in \Z} )\in X_1^s\times (X_1^{d-s})^\Z$.
\end{enumerate}

\medskip

Now we show that we have the same results for $(X_2,\X_2,\mu_2,T)$.
Similar to the proof of Proposition \ref{prop-iso}, we have that
$$\phi^{(s)}\times \phi^\infty: (X_1^s\times (X_1^{d-s})^\Z, (\widetilde{\mu_1})_\A^{(\infty)}, \langle T^\infty, \widetilde{\sigma}\rangle)\rightarrow (X_2\times (X_2^{d-s})^\Z,  (\widetilde{\mu_2})_\A^{(\infty)}, \langle T^\infty, \widetilde{\sigma}\rangle) $$
$$((x_1,x_2,\ldots, x_s),({\bf x}_n)_{n\in \Z} ) \mapsto ((\phi(x_1),\phi(x_2),\ldots, \phi(x_s)),(\phi^{(d-s)}({\bf x}_n))_{n\in \Z} )$$
is an isomorphism and $(\phi^{(s)}\times \phi^\infty)_* (\widetilde{\mu_1})_\A^{(\infty)}=(\widetilde{\mu_2})_\A^{(\infty)}$. Thus $((X_2^d)^\Z,(\X_2^d)^\Z, (\widetilde{\mu_2})_\A^{(\infty)}, \langle T^\infty, \widetilde{\sigma}\rangle)$ is ergodic.

For $\mu_2$ a.e. $x\in X_2$, we define
$$(\widetilde{\mu_2})^{(\infty)}_{\A,x}=(\phi^{(s)}\times\phi^\infty)_*\big((\widetilde{\mu_1})^{(\infty)}_{\A,\phi^{-1}(x)}\big).$$
Then $(\widetilde{\mu_2})^{(\infty)}_{\A,x}$ is ergodic and we have
\begin{equation*}
  \begin{split}
    (\widetilde{\mu_2})_\A ^{(\infty)}&= (\phi^{(s)}\times \phi^\infty)_*\Big( (\widetilde{\mu_1})_\A ^{(\infty)}\Big)
   \overset{(1.2)}= (\phi^{(s)}\times \phi^\infty)_*\Big(\int_{X_1} (\widetilde{\mu_1})_{\A, x}^{(\infty)} d\mu_1(x)\Big)\\
   & = \int_{X_2}(\phi^{(s)}\times \phi^\infty)_*\big( (\widetilde{\mu_1})_{\A, \phi^{-1}(x)}^{(\infty)}\big) d\mu_2(x)
   = \int_{X_2} (\widetilde{\mu_2})_{\A, x}^{(\infty)} d\mu_2(x).
   \end{split}
\end{equation*}
That is, we have that $\displaystyle (\widetilde{\mu_2})_\A ^{(\infty)}=\int_{X_2} (\widetilde{\mu_2})_{\A, x}^{(\infty)} d\mu_2(x)$.

Since $\phi: (X_1,\X_1,\mu_1,T)\rightarrow (X_2,\X_2,\mu_2,T)$ is an isomorphism, $U_\phi: L^2(X_2, \mu_2)\rightarrow L^2(X_1,\mu_1),$ $ f\mapsto f\circ \phi$ is an isometry. Thus,
for all $f_1,f_2,\ldots, f_s\in L^\infty(X_2,\mu_2)$,  $l \in \N$ and $f_j^{\otimes_{II}} = f_j^{(s+1)}\otimes f_j^{(s+2)} \otimes \cdots \otimes f_j^{(d)}$ (where $ f^{(t)}_j\in L^\infty(X_2,\mu_2), s+1\le t\le d,  -l \le j\le l$), we have
$$\frac{1}{N}\sum_{n=0}^{N-1} f_1(T^{a_1n}x)\cdots f_s(T^{a_sn}x) \prod_{j=-l}^l f_j^{\otimes_{II}} (T^{{p_{s+1}}(n+j)}x, \ldots, T^{p_d(n+j)}x)$$
converges in $L^2(X_2,\mu_2)$ if and only if
$$\frac{1}{N}\sum_{n=0}^{N-1} f_1\circ\phi(T^{a_1n}x)\cdots f_s\circ\phi(T^{a_sn}x) \prod_{j=-l}^l f_j^{\otimes_{II}}\circ\phi^{(d-s)} (T^{{p_{s+1}}(n+j)}x, \ldots, T^{p_d(n+j)}x)$$
converges in $L^2(X_1,\mu_1)$. And for  $f\in L^\infty(X_2,\mu)$,  we have
\begin{equation*}
\begin{split}
& \int_{X_2} f(x)\Big(\int_{X_2^s\times (X_2^{d-s})^{\Z}}\bigotimes_{i=1}^s f_i\otimes \Big( \bigotimes_{j=-l}^l f_j^{\otimes_{II}} \Big)((x_1,\ldots, x_s),{\bf x}) \
   d (\widetilde{\mu_2})^{(\infty)}_{\A, x}\Big) d\mu_2(x)\\
   &= \int_{X_2} f(x)\Big(\int_{X_2^s\times (X_2^{d-s})^{\Z}} \bigotimes_{i=1}^s f_i\otimes\\
     & \quad \quad \quad \quad  \quad \quad \Big( \bigotimes_{j=-l}^l f_j^{\otimes_{II}} \Big)((x_1,\ldots, x_s),{\bf x}) \
   d (\phi^{(s)}\times \phi^\infty)_*\big((\widetilde{\mu_1})^{(\infty)}_{\A,\phi^{-1}(x)}\big)\Big)  d\phi_*\mu_1(x)\\
   &= \int_{X_1} f\circ\phi(x)\Big(\int_{X_1^s\times (X_1^{d-s})^{\Z}} \bigotimes_{i=1}^s f_i\circ\phi \otimes \\
    & \quad \quad \quad \quad  \quad \quad \Big( \bigotimes_{j=-l}^l f_j^{\otimes_{II}}\circ
    \phi^{(d-s)} \Big)((x_1,\ldots, x_s),{\bf x}) \
   d \big((\widetilde{\mu_1})^{(\infty)}_{\A,x}\big)\Big)  d\mu_1(x)\\
   &\overset{(1.3)}= \lim_{N\to\infty}\frac{1}{N}
\sum_{n=0}^{N-1} \int_{X_1} f\circ\phi(x) \prod_{i=1}^s f_i\circ\phi(T^{a_in}x)\cdot \\
&\quad \quad \quad \quad  \quad \quad   \prod_{j=-l}^l f_j^{\otimes_{II}}\circ\phi^{(d-s)} (T^{{p_{s+1}}(n+j)}x, \ldots, T^{p_d(n+j)}x) d\mu_1(x)\\
\end{split}
\end{equation*}
which is equal to
\begin{equation*}
\begin{split}
= \int_{X_2} f(x)\left(\lim_{N\to\infty}\frac{1}{N}
\sum_{n=0}^{N-1}  \prod_{i=1}^s f_i(T^{a_in}x)\cdot \prod_{j=-l}^l f_j^{\otimes_{II}} (T^{{p_{s+1}}(n+j)}x, \ldots, T^{p_d(n+j)}x)\right) d\mu_2(x).
\end{split}
\end{equation*}

Since $f$ is arbitrary, we have
\begin{equation*}
  \begin{split}
    \frac{1}{N} \sum_{n=0}^{N-1}
  & \prod_{i=1}^s  f_i(T^{a_in}x)\cdot   \prod_{j=-l}^l f_j^{\otimes_{II}} (T^{{p_{s+1}}(n+j)}x, \ldots, T^{p_d(n+j)}x)\\
    &\overset{L^2(\mu_2)}\longrightarrow  \int_{X_2^s\times (X_2^{d-s})^{\Z}}\bigotimes_{i=1}^s f_i\otimes \Big( \bigotimes_{j=-l}^l f_j^{\otimes_{II}} \Big)((x_1,x_2,\ldots, x_s),{\bf x}) \
   d (\widetilde{\mu_2})^{(\infty)}_{\A, x}, \ N\to\infty.
   \end{split}
\end{equation*}
Hence we have showed that all results in Theorem \ref{thm-ergodic-measures} are isomorphic invariant.

\medskip
\noindent{\em Step 1. Notation and the proof of (1).}
\medskip

Let $(X, \X,\mu, T)$ be an ergodic m.p.s., and let $\pi_{\infty}: (X,\X,\mu, T)\rightarrow (Z_{\infty},\ZZ_\infty,\mu_\infty, T)$
be the factor map, where $Z_{\infty}$ is the $\infty$-step pro-nilfactor of $X$.
Assume that $Z_{\infty}$ is isomorphic to a topological $\infty$-step pro-nilsystem (see Theorem \ref{thm-HKM}, still denote it by $Z_{\infty}$). We build the uniquely ergodic minimal topological model $(\h{X},\hat{T})$ of $(X,\X,\mu,T)$ in the following way
by Weiss's theorem.
\[
\begin{CD}
X @>{}>> \h{X}\\
@V{\pi_{\infty}}VV      @VV{\h{\pi}_{\infty}}V\\
Z_{\infty} @>{ }>> Z_{\infty}\\
\end{CD}
\]
Without loss of generality we assume that $(X, T)=(\h{X},\h{T})$ and $\pi_\infty=\hat{\pi}_\infty$.

By Theorem \ref{thm-ergodic-extension} the factor map
\begin{equation*}
  \begin{split}
\pi_\infty^{(s)}\times \pi_\infty^\infty:\ &(X^s\times (X^{d-s})^\Z, \X^s\times (\X^{d-s})^\Z, \widetilde{\mu}^{(\infty)}_\A, \widetilde{\sigma}) \\ &\quad \quad \quad \quad \longrightarrow
(Z_{\infty}^s\times ({Z_\infty^{d-s}})^\Z, \ZZ_{\infty}^s\times (\ZZ_{\infty}^{d-s})^\Z, \widetilde{(\mu_\infty)}^{(\infty)}_\A, \widetilde{\sigma})
\end{split}
\end{equation*}
is ergodic.
And by Theorem \ref{thm-measure-over-nil},
if $\displaystyle \mu=\int _{Z_{\infty}} \nu_z \ d\mu_{\infty}(z)$ is the disintegration of $\mu$ over $\mu_{\infty}$, then $\widetilde{\mu}_\A^{(\infty)}$ is the conditional product measure with
respect to $(\widetilde{\mu_\infty})_\A^{(\infty)}$.
That is,
\begin{equation}\label{d0}
  \widetilde{\mu}_\A ^{(\infty)}=\int_{Z^s_\infty\times(Z_\infty^{d-s})^{\Z}}\big(\prod_{i=1}^s\nu_{z_i}\big)\times \prod_{n\in \Z} \Big( \nu_{z_n^{(s+1)}}\times \cdots \times \nu_{z_n^{(d)}} \Big) d (\widetilde{\mu_\infty})_\A ^{(\infty)} ((z_1,\ldots, z_s),{\bf z}),
\end{equation}
where $\big((z_1,\ldots, z_s),{\bf z}\big)
=\big((z_1,\ldots, z_s),(z_n^{(s+1)},\ldots, z_n^{(d)})_{n\in \Z}\big)\in Z^s_\infty\times (Z_\infty^{d-s})^{\Z}$.

Since ${\rm supp}\ \widetilde{\mu}^{(\infty)}_\A= M_\infty(X,\A)$ and ${\rm supp} \ (\widetilde{\mu_\infty})^{(\infty)}_\A=M_\infty(Z_\infty,\A)$
(by Theorem \ref{support} and the minimality of $(X,T)$ and $(Z_\infty, T)$), we may denote the factor map $\pi^{(s)}_\infty\times\pi^\infty_\infty$ by
$$\pi^{(s)}_\infty\times \pi_{\infty}^\infty: \Big(M_\infty(X,\A), \widetilde{\mu}_\A^{(\infty)}, \widetilde{\sigma}\Big)\rightarrow
\Big( M_\infty(Z_\infty, \A), (\widetilde{\mu_\infty})_\A^{(\infty)},\widetilde{\sigma}\Big).$$
By Theorem \ref{thm-ergodic-extension}, $\pi^{(s)}_\infty\times \pi_\infty^\infty$ is ergodic, and
the $\sigma$-algebra $\I(M_\infty(X,\A), \widetilde{\mu}_\A^{(\infty)}, \widetilde{\sigma})$ is isomorphic to the $\sigma$-algebra $\I(M_\infty(Z_\infty,\A), (\widetilde{\mu_\infty})_\A^{(\infty)},\widetilde{\sigma})$.
We denote the corresponding m.p.s. by
$(W, \mathcal {W}, \varsigma , \id)$. That is, we have
\begin{equation*}
  (M_\infty(X,\A), \widetilde{\mu}_\A^{(\infty)}, \widetilde{\sigma})\overset{\pi_\infty^{(s)}\times \pi^\infty_\infty}\longrightarrow
(M_\infty(Z_\infty,\A), (\widetilde{\mu_\infty})_\A^{(\infty)},\widetilde{\sigma})\stackrel{\phi}\longrightarrow
(W, \mathcal {W}, \varsigma, \id)
\end{equation*}
so that
\begin{equation}\label{d1}
\begin{split}
  \phi^{-1}(\mathcal{W}) =\I(M_\infty(Z_\infty,\A), (\widetilde{\mu_\infty})_\A^{(\infty)},\widetilde{\sigma})
   \cong \Phi^{-1}(\mathcal{W})
  =\I(M_\infty(X,\A), \widetilde{\mu}_\A^{(\infty)}, \widetilde{\sigma}),
\end{split}
\end{equation}
where $\Phi=\phi\circ(\pi_\infty^{(s)}\times \pi^\infty_\infty)$.

By Theorem \ref{thm-nil-ergodic2}-(1), $(M_\infty(Z_\infty,\A), \langle T^{\infty}, \widetilde{\sigma} \rangle)$ is uniquely ergodic, which means that \linebreak $\I(M_\infty(Z_\infty,\A), (\widetilde{\mu_\infty})_\A^{(\infty)}, \langle T^{\infty},\widetilde{\sigma} \rangle)$ is trivial.

\medskip

Now we show that $(M_\infty(X,\A), \widetilde{\mu}_\A^{(\infty)}, \langle T^\infty, \widetilde{\sigma}\rangle)$ is ergodic. Let $A$ be a $\langle T^\infty, \widetilde{\sigma}\rangle$-invariant subset of $M_\infty(X,\A)$, i.e. $\widetilde{\mu}_\A^{(\infty)}(A\D (T^\infty)^{-1} A)=0, \widetilde{\mu}_\A^{(\infty)}(A\D \widetilde{\sigma}^{-1}A)=0$. Since $\widetilde{\mu}_\A^{(\infty)}(A\D \widetilde{\sigma}^{-1}A)=0$, we have that $A\in \I(M_\infty(X,\A), \widetilde{\mu}_\A^{(\infty)}, \widetilde{\sigma})$. By \eqref{d1}, there is some $B\in \mathcal W$ and a measurable subset $A'\subseteq M_\infty(Z_\infty,\A)$ such that $\phi^{-1}(B)=A'$ and $A=\Phi^{-1}(B)=(\pi_\infty^{(s)}\times \pi^\infty_\infty)^{-1}(A')$.
Now $$(\widetilde{\mu_\infty})_\A^{(\infty)}(A'\D (T^\infty)^{-1}A')=\widetilde{\mu}_\A^{(\infty)}(A\D (T^\infty)^{-1} A)=0, $$
and
$$(\widetilde{\mu_\infty})_\A^{(\infty)}(A'\D \widetilde{\sigma}^{-1}A')=\widetilde{\mu}_\A^{(\infty)}(A\D \widetilde{\sigma}^{-1} A)=0.$$
As $\I(M_\infty(Z_\infty,\A), (\widetilde{\mu_\infty})_\A^{(\infty)}, \langle T^{\infty},\widetilde{\sigma} \rangle)$ is trivial, we have $A'=\emptyset$ or $A'=M_\infty(Z_\infty,\A) \pmod{ (\widetilde{\mu_\infty})_\A^{(\infty)}}$.
Thus $A=\emptyset$ or $A=M_\infty(X,\A) \pmod{ \widetilde{\mu}_\A^{(\infty)}}$.
It follows that $\I(M_\infty(X,\A), \widetilde{\mu}_\A^{(\infty)}, \langle T^{\infty},\widetilde{\sigma} \rangle)$ is also trivial, and hence $(M_\infty(X,\A), \widetilde{\mu}_\A^{(\infty)}, \langle T^\infty, \widetilde{\sigma}\rangle)$ is ergodic. Thus (1) holds.

\medskip
\noindent{\em Step 2. The proof of (2).}
\medskip

By Step 1, we have
\begin{equation*}
  (M_\infty(X,\A), \widetilde{\mu}_\A^{(\infty)}, \widetilde{\sigma})\overset{\pi_\infty^{(s)}\times \pi^\infty_\infty}\longrightarrow
(M_\infty(Z_\infty,\A), (\widetilde{\mu_\infty})_\A^{(\infty)},\widetilde{\sigma})\stackrel{\phi}\longrightarrow
(W, \mathcal {W}, \varsigma, \id)
\end{equation*}
so that
$ \phi^{-1}(\mathcal{W}) =\I(M_\infty(Z_\infty,\A), (\widetilde{\mu_\infty})_\A^{(\infty)},\widetilde{\sigma})
   = \Phi^{-1}(\mathcal{W})
  =\I(M_\infty(X,\A), \widetilde{\mu}_\A^{(\infty)}, \widetilde{\sigma})$,
where $\Phi=\phi\circ(\pi_\infty^{(s)}\times \pi^\infty_\infty)$.

By 
Theorem \ref{ergodic-decomposition}, we have
\begin{equation}\label{d2}
  \widetilde{\mu}_\A^{(\infty)} =\int_W \big[\widetilde{\mu}_\A^{(\infty)}\big]_w d\varsigma (w)=\int_{{\mathcal M}_\sigma^{erg}(M_\infty(X,\A))}\theta d\kappa(\theta)
\end{equation}
and
\begin{equation}\label{d3}
  (\widetilde{\mu_\infty})_\A^{(\infty)} =\int_W \big[(\widetilde{\mu_\infty})_\A^{(\infty)}\big]_w d\varsigma (w)=\int_{{\mathcal M}_\sigma^{erg}(M_\infty(Z_\infty,\A))}\theta d\widehat{\kappa}(\theta)
\end{equation}
be the ergodic decomposition of $\widetilde{\mu}^{(\infty)}_\A$ and $(\widetilde{\mu_\infty})^{(\infty)}_\A$ respectively, and one may assume that
\begin{equation}\label{d8}
  \Phi_* \big( \big[\widetilde{\mu}_\A^{(\infty)}\big]_w\big)=\phi_* \big(  \big[(\widetilde{\mu_\infty})_\A^{(\infty)}\big]_w\big)=\d_w \quad \varsigma \ a.e.\ w\in W.
\end{equation}
By Remark \ref{rem-6.7} and \eqref{d1},
\begin{equation}\label{d4}
  \left(\big\{\big[\widetilde{\mu}_\A^{(\infty)}\big]_w: w\in W\big \},\kappa \right)\cong (W,\varsigma ) \cong \left(\big\{\big[(\widetilde{\mu_\infty})_\A^{(\infty)}\big]_w: w\in W\big\},\widehat{\kappa} \right)
\end{equation}

\medskip

By Theorem \ref{thm-6.8}, there is a measure preserving map $$\zeta: (Z_\infty,\ZZ_\infty,\mu_\infty)\rightarrow (W,\mathcal{W},\varsigma ) $$ such that
$$ (\widetilde{\mu_\infty})_\A^{(\infty)} =\int_W \big[(\widetilde{\mu_\infty})_{\A}^{(\infty)}\big]_w d\varsigma(w)= \int_{Z_\infty} (\widetilde{\mu_\infty})_{\A,z}^{(\infty)} d\mu_\infty(z)$$
and
\begin{equation}\label{d5-1}
  \big[(\widetilde{\mu_\infty})_{\A}^{(\infty)}\big]_{\zeta(z)}=(\widetilde{\mu_\infty})_{\A, z}^{(\infty)}, \ \mu_\infty \ a.e. \ z\in Z_\infty.
\end{equation}
Let $$\psi=\zeta\circ \pi_\infty: (X,\X,\mu) \rightarrow (W,\mathcal{W}, \varsigma ).$$
Then $\psi$ is a measure preserving map. For $\mu$-a.e. $x\in X$, define
\begin{equation}\label{d5}
  \widetilde{\mu}^{(\infty)}_{\A,x}\triangleq \big[\widetilde{\mu}^{(\infty)}_\A\big]_{\psi(x)}.
\end{equation}
Then we have
\begin{equation}\label{d6}
\begin{split}
  \widetilde{\mu}_\A^{(\infty)} & =\int_W \big[\widetilde{\mu}_{\A}^{(\infty)}\big]_w d\varsigma(w)= \int_W \big[\widetilde{\mu}_{\A}^{(\infty)}\big]_w d\psi_*\mu(w)
  \\ &= \int_X \big[\widetilde{\mu}_{\A}^{(\infty)}\big]_{\psi(x)} d\mu(x)= \int_{X} \widetilde{\mu}_{\A,x}^{(\infty)} d\mu(x),
\end{split}
\end{equation}
and for $\mu$ a.e. $x\in X$, $\widetilde{\mu}^{(\infty)}_{\A, x}$ is ergodic under $\widetilde{\sigma}$.

\medskip
\noindent{\em Step 3. For $\mu$ a.e. $x\in X$, $\widetilde{\mu}^{(\infty)}_{\A, x}$ is the conditional product measure with
respect to $(\widetilde{\mu_\infty})_{\A,\pi(x)}^{(\infty)}$. That is, if $\displaystyle \mu=\int _{Z_{\infty}} \nu_z \ d\mu_{\infty}(z)$ is the disintegration of $\mu$ over $\mu_{\infty}$, then
\begin{equation*}\label{}
  \widetilde{\mu}_{\A,x} ^{(\infty)}=\int_{Z^s_\infty\times(Z_\infty^{d-s})^{\Z}}\big(\nu_{z_1}\times \cdots \times \nu_{z_s}\big)\times \prod_{n\in \Z} \Big( \nu_{z_n^{(s+1)}}\times \cdots \times \nu_{z_n^{(d)}} \Big) d (\widetilde{\mu_\infty})_{\A,\pi_\infty(x)} ^{(\infty)} ((z_1,\ldots, z_s),{\bf z}),
\end{equation*}
where $\big((z_1,\ldots, z_s),{\bf z}\big)=\big((z_1,\ldots, z_s),(z_n^{(s+1)},\ldots, z_n^{(d)})_{n\in \Z}\big)\in Z^s_\infty\times (Z_\infty^{d-s})^{\Z}$.}

\medskip

By \eqref{d0}, $\widetilde{\mu}_\A^{(\infty)}$ is the conditional product measure with
respect to $(\widetilde{\mu_\infty})_\A^{(\infty)}$.
That is,
\begin{equation}
  \widetilde{\mu}_\A ^{(\infty)}=\int_{Z^s_\infty\times(Z_\infty^{d-s})^{\Z}}\big(\prod_{i=1}^s\nu_{z_i}\big)\times \prod_{n\in \Z} \Big( \nu_{z_n^{(s+1)}}\times \cdots \times \nu_{z_n^{(d)}} \Big) d (\widetilde{\mu_\infty})_\A ^{(\infty)} ((z_1,\ldots, z_s),{\bf z}),
\end{equation}
where $\big((z_1,\ldots, z_s),{\bf z}\big)=\big((z_1,\ldots, z_s),(z_n^{(s+1)},\ldots, z_n^{(d)})_{n\in \Z}\big)\in Z^s_\infty\times (Z_\infty^{d-s})^{\Z}$.

By \eqref{d3}, we have
\begin{equation*}\label{}
  \begin{split}
    &\widetilde{\mu}_\A ^{(\infty)} =\int_{Z^s_\infty\times(Z_\infty^{d-s})^{\Z}}\big(\prod_{i=1}^s\nu_{z_i}\big)\times \prod_{n\in \Z} \Big( \nu_{z_n^{(s+1)}}\times \cdots \times \nu_{z_n^{(d)}} \Big) d (\widetilde{\mu_\infty})_\A ^{(\infty)} ((z_1,\ldots, z_s),{\bf z})\\
    &= \int_{M_\infty(Z_\infty,\A)}\big(\prod_{i=1}^s\nu_{z_i}\big)\times \prod_{n\in \Z} \Big( \nu_{z_n^{(s+1)}}\times \cdots \times \nu_{z_n^{(d)}} \Big) d (\widetilde{\mu_\infty})_\A ^{(\infty)} ((z_1,\ldots, z_s),{\bf z})\\
    &= \int_W \left( \int_{M_\infty(Z_\infty,\A)}\big(\prod_{i=1}^s\nu_{z_i}\big)\times \prod_{n\in \Z} \Big( \nu_{z_n^{(s+1)}}\times \cdots \times \nu_{z_n^{(d)}} \Big) d \big[(\widetilde{\mu_\infty})_\A ^{(\infty)}\big]_w ((z_1,\ldots, z_s),{\bf z})\right) d\varsigma(w).
  \end{split}
\end{equation*}
On the other hand, by \eqref{d2}, we have
\begin{equation}\label{}
  \begin{split}
     \widetilde{\mu}_\A ^{(\infty)}= \int_W \big[\widetilde{\mu}_\A^{(\infty)}\big]_w d\varsigma (w).
  \end{split}
\end{equation}
Recall that $  (M_\infty(X,\A), \widetilde{\mu}_\A^{(\infty)}, \widetilde{\sigma})\overset{\pi_\infty^{(s)}\times \pi^\infty_\infty}\longrightarrow
(M_\infty(Z_\infty,\A), (\widetilde{\mu_\infty})_\A^{(\infty)},\widetilde{\sigma})\stackrel{\phi}\longrightarrow
(W, \mathcal {W}, \varsigma, \id)$ and $\Phi=\phi\circ(\pi_\infty^{(s)}\times \pi^\infty_\infty)$.
We have that for $\varsigma$-a.e. $w\in W$,
\begin{equation*}\label{}
  \begin{split}
     & \Phi_*\Big( \int_{M_\infty(Z_\infty,\A)}\big(\prod_{i=1}^s\nu_{z_i}\big)\times \prod_{n\in \Z} \Big( \nu_{z_n^{(s+1)}}\times \cdots \times \nu_{z_n^{(d)}} \Big) d \big[(\widetilde{\mu_\infty})_\A ^{(\infty)}\big]_w ((z_1,\ldots, z_s),{\bf z}) \Big) \\
    &=  \phi_*\circ (\pi_\infty^{(s)}\times \pi^\infty_\infty)_*\Big( \int_{M_\infty(Z_\infty,\A)}\big(\prod_{i=1}^s\nu_{z_i}\big)\times \prod_{n\in \Z} \Big( \nu_{z_n^{(s+1)}}\times \cdots \times \nu_{z_n^{(d)}} \Big) \\
    &\quad \quad \quad \quad \quad \quad \quad \quad \quad \quad \quad \quad \quad \quad \quad \quad \quad \quad \quad \quad \quad d \big[(\widetilde{\mu_\infty})_\A ^{(\infty)}\big]_w ((z_1,\ldots, z_s),{\bf z}) \Big) \\
    &=  \phi_*\Big( \int_{M_\infty(Z_\infty,\A)}(\pi_\infty^{(s)}\times \pi^\infty_\infty)_*\Big(\big(\prod_{i=1}^s\nu_{z_i}\big)\times \prod_{n\in \Z} \Big( \nu_{z_n^{(s+1)}}\times \cdots \times \nu_{z_n^{(d)}} \Big) \Big)\\
    &\quad \quad \quad \quad \quad \quad \quad \quad \quad \quad \quad \quad \quad \quad \quad \quad \quad \quad \quad \quad \quad d \big[(\widetilde{\mu_\infty})_\A ^{(\infty)}\big]_w ((z_1,\ldots, z_s),{\bf z}) \Big) \\
    &= \phi_*\Big( \int_{M_\infty(Z_\infty,\A)}\d_{((z_1,\ldots, z_s),{\bf z})}d \big[(\widetilde{\mu_\infty})_\A ^{(\infty)}\big]_w ((z_1,\ldots, z_s),{\bf z}) \Big) \\
    &= \phi_*\Big(\big[(\widetilde{\mu_\infty})_\A ^{(\infty)}\big]_w \Big) \overset{\eqref{d8}} = \d_w .
  \end{split}
\end{equation*}
Thus by \eqref{d8},  for $\varsigma$-a.e. $w\in W$,
\begin{equation*}\label{}
  \begin{split}
  & \Phi_*\Big( \int_{M_\infty(Z_\infty,\A)}\big(\prod_{i=1}^s\nu_{z_i}\big)\times \prod_{n\in \Z} \Big( \nu_{z_n^{(s+1)}}\times \cdots \times \nu_{z_n^{(d)}} \Big) d \big[(\widetilde{\mu_\infty})_\A ^{(\infty)}\big]_w ((z_1,\ldots, z_s),{\bf z}) \Big)\\
  &= \Phi_* \big( \big[\widetilde{\mu}_\A^{(\infty)}\big]_w\big)=\d_w.
  \end{split}
\end{equation*}
By the uniqueness of the disintegration (\cite[Theorem 5.8]{F} or \cite[Theorem A.7]{Glasner}), we have that for $\varsigma$-a.e. $w\in W$,
\begin{equation*}
  \big[\widetilde{\mu}_\A^{(\infty)}\big]_w=\int_{M_\infty(Z_\infty,\A)}\big(\prod_{i=1}^s\nu_{z_i}\big)\times \prod_{n\in \Z} \Big( \nu_{z_n^{(s+1)}}\times \cdots \times \nu_{z_n^{(d)}} \Big) d \big[(\widetilde{\mu_\infty})_\A ^{(\infty)}\big]_w ((z_1,\ldots, z_s),{\bf z}).
\end{equation*}
Thus combined with \eqref{d5-1} and \eqref{d5}, we have
for $\mu$-a.e. $x\in X$,
\begin{equation*}\label{}
  \begin{split}
  &  \mu^{(\infty)}_{\A,x}= \big[\widetilde{\mu}_\A^{(\infty)}\big]_{\psi(x)}\\
  &= \int_{M_\infty(Z_\infty,\A)}\big(\prod_{i=1}^s\nu_{z_i}\big)\times \prod_{n\in \Z} \Big( \nu_{z_n^{(s+1)}}\times \cdots \times \nu_{z_n^{(d)}} \Big) d \big[(\widetilde{\mu_\infty})_\A ^{(\infty)}\big]_{\psi(x)} ((z_1,\ldots, z_s),{\bf z})\\
  &= \int_{M_\infty(Z_\infty,\A)}\big(\prod_{i=1}^s\nu_{z_i}\big)\times \prod_{n\in \Z} \Big( \nu_{z_n^{(s+1)}}\times \cdots \times \nu_{z_n^{(d)}} \Big) d \big[(\widetilde{\mu_\infty})_\A ^{(\infty)}\big]_{\phi(\pi_\infty(x))} ((z_1,\ldots, z_s),{\bf z})\\
  &= \int_{M_\infty(Z_\infty,\A)}\big(\prod_{i=1}^s\nu_{z_i}\big)\times \prod_{n\in \Z} \Big( \nu_{z_n^{(s+1)}}\times \cdots \times \nu_{z_n^{(d)}} \Big) d (\widetilde{\mu_\infty})_{\A,\pi_\infty(x)} ^{(\infty)} ((z_1,\ldots, z_s),{\bf z}).
  \end{split}
\end{equation*}
This ends the proof of Step 3.

\medskip
\noindent{\em Step 4. The proof of (3).}
\medskip

Now we show (3).
For all $f_1,f_2,\ldots, f_s\in L^\infty(X,\mu)$, all $l \in \N$ and all $f_j^{\otimes_{II}} = f_j^{(s+1)}\otimes f_j^{(s+2)} \otimes \cdots \otimes f_j^{(d)}$ (where $f^{(t)}_j\in L^\infty(X,\mu)$, $s+1\le t\le d, -l \le j\le l$), we show
\begin{equation}\label{eq1}
\begin{split}
  & \lim_{N\rightarrow +\infty} \frac{1}{N}\sum_{n=0}^{N-1} f_1(T^{a_1n}x)\cdots f_s(T^{a_sn}x) \prod_{j=-l}^l f_j^{\otimes_{II}} (T^{{p_{s+1}}(n+j)}x, \ldots, T^{p_d(n+j)}x) \\
  & =\int_{X^s\times (X^{d-s})^{\Z}} \big(f_1\otimes f_2\otimes \cdots \otimes f_s\big) \otimes  \Big( \bigotimes_{j=-l}^l f_j^{\otimes_{II}} \Big)d\widetilde{\mu}_{\A,x}^{(\infty)}
\end{split}
\end{equation}
in $L^2(X,\mu)$. Similar to the proof of Theorem \ref{thm-A1}-(1), for $\mu$-a.e. $x\in X$, $ \big(f_1\otimes f_2\otimes \ldots \otimes f_s\big) \otimes  \Big( \bigotimes_{j=-l}^l f_j^{\otimes_{II}} \Big)\in L^\infty(X^s\times (X^{d-s})^{\Z},\widetilde{\mu}_{\A,x}^{(\infty)})$ and the right side of the equation \eqref{eq1} makes sense. Again similar to  the proof of Theorem \ref{thm-A1}, we only need show \eqref{eq1} for continuous functions.

By Step 3, For $\mu$ a.e. $x\in X$, $\widetilde{\mu}^{(\infty)}_{\A, x}$ is the conditional product measure with
respect to $(\widetilde{\mu_\infty})_{\A,\pi_\infty(x)}^{(\infty)}$. Thus for all $l \in \N$ and all $f_i, f^{(t)}_j \in C(X)$, $1\le u\le s, s+1\le t\le d, -l \le j\le l$,
\begin{equation}\label{b1}
\begin{split}
  & \int_{X^s\times (X^{d-s})^{\Z}} \big(f_1\otimes f_2\otimes \cdots \otimes f_s\big) \otimes  \Big( \bigotimes_{j=-l}^l f_j^{\otimes_{II}} \Big)((x_1,x_2,\ldots, x_s),{\bf x})d\widetilde{\mu}_{\A,x}^{(\infty)}\\
  =& \int_{Z^s_\infty\times(Z_\infty^{d-s})^{\Z}} \int_{X^s\times (X^{d-s})^{\Z}} \big(f_1\otimes f_2\otimes \cdots \otimes f_s\big) \otimes  \Big( \bigotimes_{j=-l}^l f_j^{\otimes_{II}} \Big) \\ & \quad d \Big(\big(\nu_{z_1}\times \ldots \times \nu_{z_s}\big)\times \prod_{n\in \Z} \Big( \nu_{z_n^{(s+1)}}\times \cdots \times \nu_{z_n^{(d)}} \Big)\Big) d (\widetilde{\mu_\infty})_{\A,\pi_\infty(x)} ^{(\infty)} ((z_1,\ldots, z_s),{\bf z})\\
  =& \int_{Z^s_\infty\times(Z_\infty^{d-s})^{\Z}} \Big(\int_X f_1 d\nu_{z_1}\otimes \cdots \otimes \int_X f_s d\nu_{z_s} \Big)\otimes \\
  & \quad \quad \quad \quad \quad \quad \quad \quad \quad \quad \quad \quad \bigotimes_{j=-l}^l \prod_{t=s+1}^d \int_{X}f_j^{(t)} d \nu_{z_n^{(t)}} d (\widetilde{\mu_\infty})_{\A,\pi_\infty(x)} ^{(\infty)} ((z_1,\ldots, z_s),{\bf z})\\
  =& \int_{Z^s_\infty\times(Z_\infty^{d-s})^{\Z}} \E_\mu(f_1 | \ZZ_\infty) \otimes \cdots \otimes \E_\mu (f_s | \ZZ_{\infty}) \otimes  \bigotimes_{j=-l}^l
    \E_\mu(f_j^{\otimes_{II}} |\ZZ_{\infty}) d (\widetilde{\mu_\infty})_{\A,\pi_\infty(x)} ^{(\infty)},
\end{split}
\end{equation}
where $\E_\mu(f^{\otimes_{II}} _j|\ZZ_\infty)=\E_\mu(f^{(s+1)}_j|\ZZ_\infty)\otimes \E_\mu(f^{(2)}_j|\ZZ_\infty)\otimes\cdots \otimes \E_\mu(f^{(d)}_j|\ZZ_\infty), -l\le j\le l.$

We need to explain the last equality. Since $\displaystyle \mu=\int _{Z_{\infty}} \nu_z \ d\mu_{\infty}(z)$ is the disintegration of $\mu$ over $\mu_{\infty}$, we have that for every $f\in L^\infty(X,\mu)$, there is some subset $A_{f}\in \ZZ_\infty$ with $\mu_\infty(A_{f})=1$ such that $f$ is in $L^\infty(X,\nu_z)$, and $\displaystyle \E_\mu(f|\ZZ_\infty)(z)=\int_X fd\nu_z$ for each $z\in A_{f}$.
Since $(\widetilde{\mu_\infty})_\A^{(\infty)}$ is a standard measure on $Z^s_\infty\times(Z_\infty^{d-s})^{\Z}$, we have that
$$(\widetilde{\mu_\infty})_\A^{(\infty)}\left((\prod_{i=1}^s A_{f_i})\times \Big(\big(\prod_{j=-\infty}^{-l-1}Z_\infty^{d-s} \big)\times \big(\prod_{j=-l}^l \prod_{t=s+1}^d A_{f^{(t)}_j}\big)\times \big(\prod_{j=l+1}^{\infty}Z_\infty^{d-s}\big)\Big)\right)=1.$$
Since $\displaystyle (\widetilde{\mu_\infty})_\A^{(\infty)} = \int_{Z_\infty} (\widetilde{\mu_\infty})_{\A,z}^{(\infty)} d\mu_\infty(z)$, there is some $B\in \ZZ_\infty$ with $\mu_\infty(B)=1$ such that for all $z\in B$, we have
$$(\widetilde{\mu_\infty})_{\A,z}^{(\infty)}\left((\prod_{i=1}^s A_{f_i})\times \Big(\big(\prod_{j=-\infty}^{-l-1}Z_\infty^{d-s} \big)\times \big(\prod_{j=-l}^l \prod_{t=s+1}^d A_{f^{(t)}_j}\big)\times \big(\prod_{j=l+1}^{\infty}Z_\infty^{d-s}\big)\Big)\right)=1.$$
Thus for each $z\in B$ and for $(\widetilde{\mu_\infty})_{\A,z}^{(\infty)}$-a.e. all $ ((z_1,\ldots, z_s),{\bf z})\in Z^s_\infty\times(Z_\infty^{d-s})^{\Z}$,
\begin{equation*}
\begin{split}
  & \Big(\int_X f_1 d\nu_{z_1}\otimes \cdots \otimes \int_X f_s d\nu_{z_s} \Big)\otimes   \bigotimes_{j=-l}^l \prod_{t=s+1}^d \int_{X}f_j^{(t)} d \nu_{z_n^{(t)}} \\
& =  \E_\mu(f_1 | \ZZ_\infty) \otimes \cdots \otimes \E_\mu (f_s | \ZZ_{\infty}) \otimes  \bigotimes_{j=-l}^l
    \E_\mu(f_j^{\otimes_{II}} |\ZZ_{\infty}).
\end{split}
\end{equation*}
And the last equality of \eqref{b1} holds for $\mu$-a.e. $x\in X$.

\medskip

By Theorem \ref{thm-nil-ergodic2}, in $L^2(X,\mu)$ we have
\begin{equation*}
\begin{split}
  & \int_{X^s\times (X^{d-s})^{\Z}} \big(f_1\otimes f_2\otimes \cdots \otimes f_s\big) \otimes  \Big( \bigotimes_{j=-l}^l f_j^{\otimes_{II}} \Big)((x_1,x_2,\ldots, x_s),{\bf x})d\widetilde{\mu}_{\A,x}^{(\infty)}\\
  =& \int_{Z^s_\infty\times(Z_\infty^{d-s})^{\Z}} \E_\mu(f_1 | \ZZ_\infty) \otimes \cdots \otimes \E_\mu (f_s | \ZZ_{\infty}) \otimes  \bigotimes_{j=-l}^l
    \E_\mu(f_j^{\otimes_{II}} |\ZZ_{\infty}) d (\widetilde{\mu_\infty})_{\A,\pi_\infty(x)} ^{(\infty)}\\
 \overset{L^2(\mu)} =& \lim_{N\rightarrow +\infty} \frac{1}{N}\sum_{n=0}^{N-1} \prod_{i=1}^s \E_\mu(f_i|\ZZ_\infty)(T^{a_in}\pi_\infty(x))\cdot\\
  & \quad \quad \quad\quad \quad \prod_{j=-l}^l \E_\mu(f_j^{\otimes_{II}}|\ZZ_\infty )(T^{{p_{s+1}}(n+j)}\pi_\infty(x), \ldots, T^{p_d(n+j)}\pi_\infty(x))\\
  =& \lim_{N\rightarrow +\infty} \frac{1}{N}\sum_{n=0}^{N-1} f_1(T^{a_1n}x)\cdots f_s(T^{a_sn}x) \prod_{j=-l}^l f_j^{\otimes_{II}} (T^{{p_{s+1}}(n+j)}x, \ldots, T^{p_d(n+j)}x).
\end{split}
\end{equation*}
The last equality follows from Lemma \ref{lem-vdc}.
The whole proof is complete.
\end{proof}

Since the Step 3 of the proof is important in the sequel, we restate it here.

\begin{cor}\label{cor-structure}
Let $(X,\X,\mu,T)$ be an ergodic m.p.s. and $d\in \N$. Let $\pi_{\infty}$ be factor map from  $(X,\X,\mu, T)$ to $(Z_{\infty},\ZZ_\infty,\mu_\infty, T)$, where $Z_{\infty}$ is the $\infty$-step pro-nilfactor of $X$. Let $\A= \{p_1, \ldots, p_{d}\}$ be a family of non-constant integral polynomials satisfying $(\spadesuit)$. Let $\big\{\widetilde{\mu}^{(\infty)}_{\A,x}\big\}_{x\in X}$ be the family of Borel probability measures on $(X^d)^{\Z}$ in Theorem \ref{thm-ergodic-measures}. Then for $\mu$ a.e. $x\in X$, $\widetilde{\mu}^{(\infty)}_{\A, x}$ is the conditional product measure with
respect to $(\widetilde{\mu_\infty})_{\A,\pi_\infty(x)}^{(\infty)}$. That is, if $ \mu=\int _{Z_{\infty}} \nu_s \ d\mu_{\infty}(s)$ be the disintegration of $\mu$ over $\mu_{\infty}$, then
\begin{equation*}\label{}
  \widetilde{\mu}_{\A,x} ^{(\infty)}=\int_{Z^s_\infty\times(Z_\infty^{d-s})^{\Z}}\big(\prod_{i=1}^s\nu_{z_i}\big)\times \prod_{n\in \Z} \Big( \nu_{z_n^{(s+1)}}\times \cdots \times \nu_{z_n^{(d)}} \Big) d (\widetilde{\mu_\infty})_{\A,\pi_\infty(x)} ^{(\infty)} ((z_1,\ldots, z_s),{\bf z}),
\end{equation*}
where $\big((z_1,\ldots, z_s),{\bf z}\big)=\big((z_1,\ldots, z_s),(z_n^{(s+1)},\ldots, z_n^{(d)})_{n\in \Z}\big)\in Z^s_\infty\times (Z_\infty^{d-s})^{\Z}$. In particular,
$$(\pi_\infty^{(s)}\times \pi_\infty^\infty)_*  \widetilde{\mu}_{\A,x} ^{(\infty)}= (\widetilde{\mu_\infty})_{\A,\pi_\infty(x)} ^{(\infty)}.$$
\end{cor}

We can give a similar result for $\mu^{(\infty)}_\A$ as for $\widetilde{\mu}^{(\infty)}_\A$ stated in Theorem \ref{thm-ergodic-measures}.

\begin{thm}\label{thm-ergodic-measures2}
Let $(X,\X,\mu,T)$ be an ergodic m.p.s. and $d\in \N$. Let $\A= \{p_1,  p_2, \cdots,  p_d \}$ be a family of non-constant integral polynomials satisfying $(\spadesuit)$.  Then there exists a family $\big\{\mu^{(\infty)}_{\A,x}\big\}_{x\in X}$ of Borel probability measures on $(X^d)^{\Z}$ such that
\begin{enumerate}
  \item $((X^d)^\Z,(\X^d)^\Z, \mu_\A^{(\infty)}, \langle T^\infty, \sigma\rangle)$ is ergodic.

  \item The disintegrations of $\mu_\A^{(\infty)}$ over $\mu$ is given by 
$\mu_\A ^{(\infty)}=\int_X \mu_{\A, x}^{(\infty)} d\mu(x).$
And for $\mu$ a.e. $x\in X$, $\mu^{(\infty)}_{\A, x}$ is ergodic under $\sigma$;

  \item For all $l \in \N$ and all $f_j^\otimes = f_j^{(1)}\otimes f_j^{(2)} \otimes \cdots \otimes f_j^{(d)}$ (where $ f^{(i)}_j\in L^\infty(X,\mu), 1\le i\le d,  -l \le j\le l$)
\begin{equation}
\begin{split}
 \frac{1}{N} \sum_{n=0}^{N-1}
   \prod_{j=-l}^l f_j^\otimes (T^{\vec{p}(n+j)}x^{\otimes d})
  \overset{L^2(\mu)}\longrightarrow  \int_{(X^d)^{\Z}} \Big( \bigotimes_{j=-l}^l f_j^\otimes \Big)({\bf x})d\mu_{\A, x}^{(\infty)}({\bf x}), \ N\to\infty,
\end{split}
\end{equation}
where ${\bf x}=({\bf x}_n)_{n\in \Z}\in (X^d)^{\Z}$.
\end{enumerate}
\end{thm}

Note that when $s=0$, i.e. all polynomials in $\A$ are non-linear, then Theorem \ref{thm-ergodic-measures} and Theorem \ref{thm-ergodic-measures2} are the same.

\subsection{The structure of $\big(X^s\times (X^{d-s})^\Z, \X^s\times (\X^{d-s})^\Z, \widetilde{\mu}^{(\infty)}_{\A,x}, \widetilde{\sigma} \big)$ }\
\medskip

We can give an explicit structure of $\big(X^s\times (X^{d-s})^\Z, \X^s\times (\X^{d-s})^\Z, \widetilde{\mu}^{(\infty)}_{\A,x}, \widetilde{\sigma} \big)$ for almost all $x\in X$ by using Corollary \ref{cor-structure}. Since we need Lemma \ref{lem-vdc} in the proof, we do not have a nice corresponding result for ${\mu}^{(\infty)}_{\A,x}$.

\begin{thm}\label{measure-like}
Let $(X, \X,\mu, T)$ be ergodic and $\pi$ be the factor map from $(X,\X,\mu, T)$ to $(Z_{\infty},\ZZ_\infty,\mu_\infty, T)$, where $Z_{\infty}$ is the $\infty$-step pro-nilfactor of $X$.  Let $\A= \{p_1,  \cdots,  p_d \}$ be a family of non-constant integral polynomials satisfying $(\spadesuit)$. Then there exist a probability space $(U, {\mathcal U},\rho)$ and a measurable cocycle $\a: Z_\infty\rightarrow {\rm Aut} (U,\rho)$, and there is a subset $X_0\in \X$ with $\mu(X_0)=1$ such that
for each $x\in X_0$, $\big(X^s\times (X^{d-s})^\Z, \X^s\times (\X^{d-s})^\Z, \widetilde{\mu}^{(\infty)}_{\A,x}, \widetilde{\sigma} \big)$ is isomorphic to $\Big(\big(Z_\infty^s\times (Z_\infty^{d-s})^\Z\big)\times \big(U^s\times(U^{d-s})^\Z\big),\big(\ZZ_\infty^s\times (\ZZ_\infty^{d-s})^\Z\big)\times \big({\mathcal U}^s\times({\mathcal U}^{d-s})^\Z\big),(\widetilde{\mu_\infty})^{(\infty)}_{\A,\pi_\infty(x)}\times \big(\rho^s\times (\rho^{d-s})^\Z\big), \widehat{\sigma} \Big)$,
      where
      $$\widehat{\sigma}: \big(Z_\infty^s\times (Z_\infty^{d-s})^\Z\big)\times \big(U^s\times(U^{d-s})^\Z\big)\rightarrow \big(Z_\infty^s\times (Z_\infty^{d-s})^\Z\big)\times \big(U^s\times(U^{d-s})^\Z\big)$$ is defined as follows:
      \begin{equation}\label{}
      \begin{split}
        &\widehat{\sigma}\big(((z_1,\ldots, z_s),{\bf z}), ((u_1,\ldots, u_s),{\bf u}) \big)\\
        &\quad = \big(\widetilde{\sigma}((z_1,\ldots, z_s),{\bf z}), ((\a(z_1)u_1,\ldots, \a(z_s)u_s), \sigma{\bf u}) \big),
      \end{split}
      \end{equation}
        where $\big(((z_1,\ldots, z_s),{\bf z}), ((u_1,\ldots, u_s),{\bf u}) \big)\in \big(Z_\infty^s\times (Z_\infty^{d-s})^\Z\big)\times \big(U^s\times(U^{d-s})^\Z\big)$ and $\sigma:(U^{d-s})^\Z\rightarrow  (U^{d-s})^\Z$ is the shift.
\end{thm}

\begin{proof}
Let $\pi$ be the factor map from $(X,\X,\mu, T)$ to $(Z_{\infty},\ZZ_\infty,\mu_\infty, T)$. By Rohlin's skew-product theorem (Theorem \ref{Rohlin}), we may assume that there exist a probability pace $(U, {\mathcal U},\rho)$ and a measurable cocycle $\a: Z_\infty\rightarrow {\rm Aut} (U,\rho)$ such that
$$(X,\X,\mu,T)=(Z_\infty \times U, \ZZ_\infty \times {\mathcal U}, \mu_\infty \times \rho, T_\a),$$
where $T_\a(y,u)=(Ty, \a (y)u)$.
In this case $\displaystyle \E_\mu(f|\ZZ_\infty)(y,u)=\int_{U}f(y,u) d\rho(u)$ for all $f\in L^1(X,\mu)$.
We assume that $X, Z_\infty$ and $U$ are all compact metric spaces.

Define
$$\Psi: X^s\times (X^{d-s})^\Z \rightarrow \big(Z_\infty^s\times (Z_\infty^{d-s})^\Z\big)\times \big(U^s\times(U^{d-s})^\Z\big)$$
$$\Big(\big((z_1,u_1),\ldots, (z_s,u_s)\big), {\bf x} \Big)\mapsto \big(((z_1,z_2,\ldots, z_s),{\bf z}), ((u_1,u_2,\ldots, u_s),{\bf u}) \big),$$
where ${\bf x}=\big((z_n^{(s+1)},u_n^{(s+1)}),\ldots, (z_n^{(d)},u_n^{(d)})\big)_{n\in \Z}\in (X^{d-s})^\Z$, ${\bf z}=(z_n^{(s+1)},\ldots, z_n^{(d)})_{n\in \Z}\in (Z_\infty^{d-s})^\Z$ and ${\bf u}=(u_n^{(s+1)},\ldots, u_n^{(d)})_{n\in \Z}\in (U^{d-s})^\Z$.
Via the map $\Psi$, we may regard $\big(X^s\times (X^{d-s})^\Z, \widetilde{\sigma} \big)$ as $\Big(\big(Z_\infty^s\times (Z_\infty^{d-s})^\Z\big)\times \big(U^s\times(U^{d-s})^\Z\big), \widehat{\sigma} \Big)$.
Now we show that for $\mu$-a.e. $x\in X$, $\widetilde{\mu}^{(\infty)}_{\A,x}\cong (\widetilde{\mu_\infty})^{(\infty)}_{\A,\pi_\infty(x)}\times \big(\rho^s\times (\rho^{d-s})^\Z\big)$.

\medskip

By Corollary \ref{cor-structure}, there is a subset $X_0\in \X$ with $\mu(X_0)=1$ such that
for each $x\in X_0$, $\widetilde{\mu}^{(\infty)}_{\A, x}$ is the conditional product measure with
respect to $(\widetilde{\mu_\infty})_{\A,\pi_\infty(x)}^{(\infty)}$. That is, if $ \mu=\int _{Z_{\infty}} \nu_z \ d\mu_{\infty}(z)$ be the disintegration of $\mu$ over $\mu_{\infty}$, then
\begin{equation*}\label{}
  \widetilde{\mu}_{\A,x} ^{(\infty)}=\int_{Z^s_\infty\times(Z_\infty^{d-s})^{\Z}}\big(\prod_{i=1}^s\nu_{z_i}\big)\times \prod_{n\in \Z} \Big( \nu_{z_n^{(s+1)}}\times \cdots \times \nu_{z_n^{(d)}} \Big) d (\widetilde{\mu_\infty})_{\A,\pi_\infty(x)} ^{(\infty)} ((z_1,\ldots, z_s),{\bf z}),
\end{equation*}
where $\big((z_1,\ldots, z_s),{\bf z}\big)=\big((z_1,\ldots, z_s),(z_n^{(s+1)},\ldots, z_n^{(d)})_{n\in \Z}\big)\in Z^s_\infty\times (Z_\infty^{d-s})^{\Z}$.

Now $\displaystyle \mu=\mu_\infty \times \rho=\int_{Z_\infty}\d_z\times \rho d\mu_\infty(z)$, i.e. $\nu_z=\d_z\times \rho$ for $\mu_\infty$-a.e. $z\in Z_\infty$. It follows that
for $x\in X_0$
\begin{equation*}
\begin{split}
\widetilde{\mu}_{\A,x} ^{(\infty)}&=\int_{Z^s_\infty\times(Z_\infty^{d-s})^{\Z}}\big(\prod_{i=1}^s\nu_{z_i}\big)\times \prod_{n\in \Z} \Big( \nu_{z_n^{(s+1)}}\times \cdots \times \nu_{z_n^{(d)}} \Big) d (\widetilde{\mu_\infty})_{\A,\pi_\infty(x)} ^{(\infty)} ((z_1,\ldots, z_s),{\bf z})\\
&= \int_{Z^s_\infty\times(Z_\infty^{d-s})^{\Z}}\big( \d_{(z_1,\ldots,z_s)}\times \rho^s\big) \times \Big( \d_{{\bf z}}\times (\rho^{d-s})^\Z\Big) d (\widetilde{\mu_\infty})_{\A,\pi_\infty(x)} ^{(\infty)} ((z_1,\ldots, z_s),{\bf z})\\
&\cong\left( \int_{Z^s_\infty\times(Z_\infty^{d-s})^{\Z}} \d_{(z_1,\ldots,z_s)}\times \d_{{\bf z}} d(\widetilde{\mu_\infty})^{(\infty)}_{\A,\pi_\infty(x)}((z_1,\ldots, z_s),{\bf z})\right) \times \big(\rho^s\times (\rho^{d-s})^\Z\big) \\
&= (\widetilde{\mu_\infty})^{(\infty)}_{\A,\pi_\infty(x)}\times \big(\rho^s\times (\rho^{d-s})^\Z\big).
\end{split}
\end{equation*}
The proof is complete.
\end{proof}

\begin{rem}\label{rem-mea} We have
\begin{enumerate}
  \item When $s=d$, each element in $\A$ is linear and in Theorem \ref{measure-like} we have $\mu^{(d)}_{\vec{a},x}\cong (\mu_\infty)^{(d)}_{\vec{a},\pi_\infty(x)}\times \rho^d$, where $\vec{a}=(a_1,\ldots, a_d)$.
  \item When $s=0$, each element in $\A$ has degree greater than 1. Hence $\widetilde{\mu}^{(\infty)}_{\A}=\mu^{(\infty)}_\A$. By Theorem \ref{measure-like}, for $x\in X_0$
      $${\mu}^{(\infty)}_{\A,x}=\widetilde{\mu}^{(\infty)}_{\A,x} \cong (\widetilde{\mu_\infty})^{(\infty)}_{\A,\pi_\infty(x)}\times (\rho^{d})^\Z=({\mu_\infty})^{(\infty)}_{\A,\pi_\infty(x)}\times (\rho^{d})^\Z.$$
      That is, $((X^d)^\Z, (\X^d)^\Z, {\mu}^{(\infty)}_{\A,x}, \sigma)$ is the product of an $\infty$-step  pro-nilsystem  \linebreak $\big((Z_\infty^d)^\Z, (\ZZ_\infty ^d)^\Z, {(\mu_\infty)}^{(\infty)}_{\A,\pi_\infty(x)}, \sigma\big)$ and a Bernoulli system $((U^d)^\Z,({\mathcal U}^d)^\Z, (\rho^d)^\Z,\sigma)$.
\end{enumerate}
\end{rem}

We now derive some consequences of the above results.
\begin{thm}\label{measure-sequence}
Let $(X, \X,\mu, T)$ be an ergodic m.p.s., where $(X,T)$ is a t.d.s. Let $\A= \{p_1,  p_2, \cdots,  p_d \}$ be a family of non-constant integral polynomials satisfying $(\spadesuit)$. Then there is a subset $X_0\in \X$ with $\mu(X_0)=1$ such that
there is some increasing subsequence $\{N_i\}_{i\in \N}$ of $\N$ such that in ${\mathcal M}(X^s\times (X^{d-s})^\Z)$,  for all $x\in X_0$,
  $$\frac{1}{N_i}\sum_{n=1}^{N_i}\d_{\widetilde{\sigma}^n\xi^\A_x}\rightarrow \widetilde{\mu}^{(\infty)}_{\A,x}, \quad i\to\infty.$$
\end{thm}

\begin{proof}

By Theorem \ref{thm-ergodic-measures}, for all $F\in C(X^s\times (X^{d-s})^{\Z})$,
$$\frac{1}{N} \sum_{n=0}^{N-1} F(\widetilde{\sigma}^n(\xi_x^\A))
      \overset{L^2(\mu)}\longrightarrow  \int_{(X^d)^{\Z}} F\big(((x_1,x_2,\ldots, x_s),{\bf x})\big)\
   d \widetilde{\mu}^{(\infty)}_{\A, x}(((x_1,x_2,\ldots, x_s),{\bf x}) ),$$
where $((x_1,x_2,\ldots, x_s),{\bf x}) =((x_1, x_2,\ldots, x_s),({\bf x}_n)_{n\in \Z} )\in X^s\times (X^{d-s})^\Z$.

Let ${\mathcal F}$ be a countable dense subset in $C(X)$. Since mean convergence of a sequence of functions implies pointwise almost everywhere convergence along a subsequence, we can find a subsequence $\{N_i\}_{i\in \N}$ of $\N$ with $N_i\to \infty$ such that
there is some measurable set $X_0\subseteq X$ with $\mu(X_0)=1$ satisfying that for $x\in X_0$
and for all $f_1, f_2,\ldots, f_s\in {\mathcal F}$, and for all $l \in \N$ and all $f_j^{\otimes_{II}} = f_j^{(s+1)}\otimes f_j^{(s+2)} \otimes \cdots \otimes f_j^{(d)}$ (where $f^{(t)}_j\in {\mathcal F}$, $s+1\le t\le d, -l \le j\le l$),
\begin{equation*}
\begin{split}
  & \quad \lim_{i\to \infty} \frac{1}{N_i}\sum_{n=0}^{N_i-1} f_1(T^{a_1n}x)\cdots f_s(T^{a_sn}x) \prod_{j=-l}^l f_j^{\otimes_{II}} (T^{{p_{s+1}}(n+j)}x, \ldots, T^{p_d(n+j)}x)  \\
  &= \int_{X^s\times (X^{d-s})^{\Z}} \big( f_1\otimes f_2\otimes \cdots \otimes f_s\big) \otimes  \Big( \bigotimes_{j=-l}^l f_j^{\otimes_{II}} \Big)({\bf x})d\widetilde{\mu}_{\A,x}^{(\infty)}((x_1, x_2,\ldots, x_s),{\bf x}).
\end{split}
\end{equation*}
That is, for all $x\in X_0$,
$$\widetilde{\mu}^{(\infty)}_{\A,x}=\lim_{i\to\infty} \frac{1}{N_i}\sum_{n=0}^{N_i-1}\d_{\widetilde {\sigma}^n\xi_x^\A}, \ w^* \ \text{in ${\mathcal M}(X^s\times(X^{d-s})^\Z)$}.$$
The proof is complete.
\end{proof}

\begin{rem}\label{rem-mea2}
We have
\begin{enumerate}

  \item If in Theorem \ref{thm-ergodic-measures}-(3), the equation \eqref{ss} holds almost surely, then in Theorem \ref{measure-sequence}, $\{N_i\}_{i\in \N}$ is $\{N\}_{N\in \N}$, that is,
      $$\lim_{N\to\infty}\frac{1}{N}\sum_{n=1}^{N}\d_{\widetilde{\sigma}^n\xi^\A_x}= \widetilde{\mu}^{(\infty)}_{\A,x}\cong (\widetilde{\mu_\infty})^{(\infty)}_{\A,\pi_\infty(x)}\times \big(\rho^s\times (\rho^{d-s})^\Z\big).$$

  \item By the same proof, the statement of Theorem \ref{measure-sequence} can be replaced by the following: for any increasing sequence $\{N_i\}_{i\in \N}$ of $\N$, there is an increasing subsequence $\{N'_i\}_{i\in \N}$ of $\{N_i\}_{i\in \N}$ such that in ${\mathcal M}(X^s\times (X^{d-s})^\Z)$,  for all $x\in X_0$,
  $$\frac{1}{N'_i}\sum_{n=1}^{N'_i}\d_{\widetilde{\sigma}^n\xi^\A_x}\rightarrow \widetilde{\mu}^{(\infty)}_{\A,x}, \quad i\to\infty.$$
\end{enumerate}
\end{rem}

\section{Convergence theorems}\label{section-applications}

In this section we will give some applications of the theory we developed in the previous sections,
namely 
we will show some mean and pointwise convergence results, which extend the corresponding ones proved in \cite{FranHost21}. The general ideas of the proofs are similar to the ones used in \cite{FranHost21}. The essential difficulty we face is that
several basic facts used in the proof of \cite{FranHost21} are only stated for a single term $f(T^n)$. To show our results we need to prove
the facts for the multiple terms $f_1(T^{a_1n})\cdots f_m(T^{a_mn})$ and $g_1(S^{p_1})\cdots g_d(S^{p_d})$, which are done either in the
previous subsections (for example Remark \ref{rem-mea2}) or in this section and the appendix (for example Theorem \ref{plus-nil}).

\medskip

\subsection{Equivalent conditions for pointwise convergence of multiple ergodic averages}\
\medskip

\begin{prop}\label{prop-pw}
Let $(X, \X, \mu, T)$ be an ergodic m.p.s. Then the following statements are equivalent.
\begin{enumerate}
  \item For all $d\in \N$, families $\A= \{p_1,  \ldots,  p_d \}$ of non-constant integral polynomials satisfying $(\spadesuit)$, for any $F\in C(X^s\times (X^{d-s})^{\Z})$,
\begin{equation*}
    \frac 1 N\sum_{n=0}^{N-1} F(\widetilde{\sigma}^n \xi_x^\A)
\end{equation*}
converges $\mu $ a.e. $ x\in X$.
  \item For all $d\in\N$, $\A=\{p_1, \ldots, p_{s}, p_{s+1}, \ldots, p_{d}\}$ satisfying $(\spadesuit)$, where $s\ge 0$, $p_i(n)=a_in$, $1\le i\le s$, $a_1,\ldots, a_s$ are distinct non-zero integers, and $\deg p_{i}\ge 2, s+1\le i\le d$, for all $f_1,f_2,\ldots, f_s\in L^\infty(X,\mu)$, and for all $l \in \N$ and all $f_j^{\otimes_{II}} = f_j^{(s+1)}\otimes f_j^{(s+2)} \otimes \cdots \otimes f_j^{(d)}$, (where $f^{(t)}_j\in L^\infty(X,\mu)$, $s+1\le t\le d, -l\le j\le l$),
\begin{equation*}
\begin{split}
  \frac{1}{N}\sum_{n=0}^{N-1} f_1(T^{a_1n}x)\cdots f_s(T^{a_sn}x) \prod_{j=-l}^l f_j^{\otimes_{II}} (T^{{p_{s+1}}(n+j)}x, \ldots, T^{p_d(n+j)}x)
\end{split}
\end{equation*}
converges almost surely.
  \item For all $d\in\N$, non-constant integral polynomials $p_1,\ldots, p_d$ and $f_1, \ldots, f_d \in L^{\infty}(X, \mu)$,
\begin{equation*}
    \frac 1 N\sum_{n=0}^{N-1}f_1(T^{p_1(n)}x)\cdots f_d(T^{p_d(n)}x)
\end{equation*}
converges almost surely.
\end{enumerate}

\end{prop}

\begin{proof}
Similar to the proof of Theorem \ref{thm-A2} one can prove that (1) is equivalent to (2). It is clear that (3) implies (2). It remains to  show (2) implies (3).

Without loss of generality, we assume that $p_1,p_2,\cdots, p_d$ are essentially distinct.
Let  $\A=\{p_1(n)=a_1 n, \ldots, p_{s}(n)=a_{s}n, p_{s+1}(n), \ldots, p_{d}(n)\}$. If $\A$ satisfies the condition ($\spadesuit$), then by (2) (in this case we take $l=0$), we have
\begin{equation*}
    \frac 1 N\sum_{n=0}^{N-1}f_1(T^{p_1(n)}x)\cdots f_d(T^{p_d(n)}x)
    =\frac 1 N\sum_{n=0}^{N-1}f_1(T^{a_1n}x)\cdots f_s(T^{a_sn}x)f_{s+1}(T^{p_{s+1}(n)}x)\cdots f_d(T^{p_d(n)}x)
\end{equation*}
converge almost surely.

If $\A$ does not satisfy the condition ($\spadesuit$), then
by Lemma \ref{ww=dem}, there are integral polynomials $\{q_1,\ldots, q_r\}$ satisfying the condition ($\spadesuit$) and $\{p_{s+1},\ldots, p_d\}\subseteq \{q_i^{[j]}=q_i(n+j)-q_i(j): 1\le i\le r, j\in \Z \}.$ Thus there is some $l\in \N$ such that
$$\{p_{s+1},\ldots, p_d\}\subseteq \{q_i^{[j]}=q_i(n+j)-q_i(j): 1\le i\le r, -l\le j\le l\}.$$
Then $\A'=\{a_1n, \ldots, a_sn, q_1(n),\ldots, q_r(n)\}$ satisfies the condition ($\spadesuit$).

Applying (2) to $\A'$ and $l$, we have
for  all $f_1,f_2,\ldots, f_s\in L^\infty(X,\mu)$, and for all $l \in \N$ and all $f^{(t)}_j\in L^\infty(X,\mu)$, $1\le t\le r, 1\le j\le l$,
\begin{equation*}
\frac{1}{N}\sum_{n=0}^{N-1} \prod_{i=1}^s f_i(T^{a_in}x)\cdot \prod _{j=1}^l f_j^{(1)}(T^{q_{1}(n+j)}x)\cdots f_j^{(r)}(T^{q_r(n+j)}x)
\end{equation*}
converges $\mu $ a.e. $ x\in X$. Note that
$$\{p_{s+1},\ldots, p_d\}\subseteq \{q_i^{[j]}=q_i(n+j)-q_i(j): 1\le i\le r, -l\le j\le l\}.$$
It follows that for all $f_1, \cdots, f_d \in L^{\infty}(X, \mu)$
\begin{equation*}
    \frac 1 N\sum_{n=0}^{N-1}f_1(T^{p_1(n)}x)\cdots f_d(T^{p_d(n)}x)
\end{equation*}
converge almost surely. The proof is complete.
\end{proof}

\subsection{Mean convergence}\
\medskip

 First we need some preparations.

\subsubsection{Pinsker factor} Every m.p.s. $(X,\X,\mu,T)$ admits a largest factor belonging to
the class of systems of entropy zero. This factor is called the {\em Pinsker factor} of $(X,\X,\mu,T)$
and is denoted by $\Pi(X,T)$.
It is well known that the Pinsker factor of the product of two m.p.s. is the product of their Pinsker factors (see for example \cite[Chapter 4, Theorem 14]{Parry} or \cite[Theorem 18.13.]{Glasner}), i.e. for m.p.s. $(X,\X,\mu,T)$ and $(Y,\Y,\nu,S)$, we have
$\Pi(X\times Y, T\times S)=\Pi(X,T)\times \Pi(Y,S).$
By this result, we have the following observation.

\begin{lem}\label{lem-pin}
Let $(X,\X,\mu,T)$ be a m.p.s. with zero entropy and let $(Y,\Y,\nu,S)$ be a Bornoulli system. Then  $\Pi(X\times Y, T\times S)=\X \times \{Y,\emptyset\}$, i.e.  $(X,\X,\mu,T)$ is the Pinsker factor of $(X\times Y,\X\times \Y, \mu\times \nu,T\times S)$.
\end{lem}

The following result is from \cite{Thouvenot} (see also \cite{rue, KKLR}).

\begin{thm}\label{Pinsker}
Let $\lambda$ be a joining of a m.p.s. $(X,\X,\mu,T)$ with zero entropy and a m.p.s. $(Y,\Y,\nu,S)$. Let $f\in L^2(X,\mu)$ and $g\in L^2(Y,\nu)$. If $\E_\nu (g|\Pi(Y,S))=0$, then
$$\int_{X\times Y} f(x)g(y)d\lambda (x,y)=0.$$
\end{thm}


\begin{lem}\label{lem-zero-entropy}
Let $(X,\X,\mu,T)$ be a m.p.s. with zero entropy and let $c_1,\ldots, c_m\in \Z\setminus \{0\}$. Let $\mu^{(m)}_{\vec{c}}$ be the Furstenberg joining respect to $\vec{c}=(c_1,\ldots, c_m)$. Then there is an $X^*\in \X$ with  $\mu(X^*)=1$ such that for all $x\in X^*$, $(X^m, \X^m, \mu^{(m)}_{\vec{c},x}, \tau_{\vec{c}})$ has zero entropy.
\end{lem}

\begin{proof}
Since $(X,\X,\mu,T)$ has zero entropy, for all $c\in \Z$, $(X,\X,\mu, T^c)$ has zero entropy.
Thus as a joining of $\{(X,\X,\mu, T^{c_i})\}_{i=1}^m$, the m.p.s. $(X^m,\X^m, \mu^{(m)}_{\vec{c}}, \tau_{\vec{c}})$ has zero entropy, i.e. $h_{\mu_{\vec{c}} ^{(m)}}(X^d, \tau_{\vec{c}})=0$.

By Theorem \ref{thm-HSY}, the disintegrations of $\mu_{\vec{c}}^{(m)}$ over $\mu$ is $\displaystyle \mu_{\vec{c}} ^{(m)}=\int_X \mu_{\vec{c}, x}^{(m)} d\mu(x),$ and for $\mu$ a.e. $x\in X$, $\mu^{(m)}_{\vec{c},x}$ is ergodic under $\tau_{\vec{c}}=T^{c_1}\times T^{c_2}\times \cdots \times T^{c_m}$.

Since $\displaystyle \mu_{\vec{c}} ^{(m)}=\int_X \mu_{\vec{c}, x}^{(m)} d\mu(x),$ we have (see \cite[Theorem 15.12]{Glasner})
\begin{equation}\label{}
  h_{\mu_{\vec{c}} ^{(m)}}(X^m, \tau_{\vec{c}})=\int_X  h_{\mu_{\vec{c},x} ^{(m)}}(X^m, \tau_{\vec{c}}) d \mu(x).
\end{equation}
As $h_{\mu_{\vec{c}} ^{(m)}}(X^m, \tau_{\vec{c}})=0$, it follows that there is an $X^*\in \X$ with  $\mu(X^*)=1$ such that for all $x\in X^*$, $h_{\mu_{\vec{c},x} ^{(m)}}(X^m, \tau_{\vec{c}})=0$.
The proof is complete.
\end{proof}

\subsubsection{A convergence theorem}

The following result is a convergence theorem appearing in \cite{CFH11}, which is a direct consequence of a generalization of Wiener-Wintner theorem proved by Host-Kra \cite{HK09} (see Theorem \ref{WW-Convergence}).

\begin{thm}\cite[Corollary 6.3]{CFH11}
Let $(X,\X,\mu, T)$ be a m.p.s. and $(\phi_n(x))_{n\in \Z}$ be a uniformly bounded sequence of $\mu$-measurable functions such that for $\mu$-almost every $x\in X$, the sequence $(\phi_n(x))_{n\in \Z}$ is a $k$-step nilsequence for some $k\in \N$. Then for $f\in L^\infty(X,\mu)$, the averages
$$\frac{1}{N-M}\sum_{n=M}^{N-1} \phi_n(x)f(T^{n}x)$$
converge in $L^2(X,\mu)$ as $N-M\to\infty$.

If in addition $\HK f\HK_{k+1}=0$, then $\displaystyle \frac{1}{N-M}\sum_{n=M}^{N-1} \phi_n(x)f(T^{n}x)$ converge to $0$ in $L^2(X,\mu)$.
\end{thm}

We need a generalization of the above theorem, which will be proved in the appendix (see Theorem \ref{thm-apendix} for a more general result).

\begin{thm}\label{plus-nil}
Let $(X,\X,\mu, T)$ be a m.p.s. and $k\in \N$. Let $(\phi_n(x))_{n\in \Z}$ be a uniformly bounded sequence of $\mu$-measurable functions such that, for $\mu$-almost every $x\in X$, the sequence $(\phi_n(x))_{n\in \Z}$ is a $k$-step nilsequence. Then for all $c_1,\ldots,c_d\in \Z$ and for all $f_1,\ldots, f_d\in L^\infty(X,\mu)$, the averages
$$\frac{1}{N-M}\sum_{n=M}^{N-1} \phi_n(x)f_1(T^{c_1n}x)f_2(T^{c_2n}x)\cdots f_d(T^{c_d n}x)$$
converge in $L^2(X,\mu)$ as $N-M\to\infty$.
\end{thm}

\subsubsection{}
Now we start to prove convergence theorems. First we show the following

\begin{prop}\label{prop-reduing}
Let $T,S$ be ergodic measure preserving transformations acting on a probability space $(X,\X,\mu)$ such that $(X,\X,\mu,T)$ has zero entropy. Let $c_1,\ldots, c_m\in \Z\setminus \{0\}$ and $p_1,\ldots,p_d$ be non-constant integral polynomials with $\deg {p_i}\ge 2, 1\le i\le d$ satisfying $(\spadesuit)$. Then for all $l\in \N$, for all $f_i, g_j^{(t)}\in L^\infty(X,\mu), 1\le i\le m, 1\le t\le d, -l\le j\le l$ with $\E_\mu(g_j^{(t)}|\ZZ_\infty(S))=0$ for some $t\in [1,d], j\in [-l,l]$, one has
\begin{equation}\label{h1}
  \lim_{N\to\infty} \frac{1}{N}\sum_{n=0}^{N-1} T^{c_1n}f_1\cdots T^{c_mn}f_m\cdot \prod_{j=-l}^l \Big(S^{p_1(n+j)}g_j^{(1)}\cdots S^{p_d(n+j)}g_j^{(d)}\Big)=0
\end{equation}
in $L^2(X,\mu)$.
\end{prop}

\begin{proof}
Let $\pi_{\infty}: (X,\X,\mu, S)\rightarrow (Z_{\infty}(S),\ZZ_\infty(S),\mu_\infty, S)$
be the factor map, where $Z_{\infty}(S)$ is the $\infty$-step pro-nilfactor of $(X,\X,\mu,S)$.
By Theorem \ref{Rohlin}, we may assume that there exist a probability pace $(U, {\mathcal U},\rho)$ and a measurable cocycle $\a: Z_\infty(S)\rightarrow {\rm Aut} (U,\rho)$ such that
$$(X,\X,\mu,S)=(Z_\infty(S) \times U, \ZZ_\infty(S) \times {\mathcal U}, \mu_\infty \times \rho, S_\a).$$
In this case $\displaystyle \E_\mu(f|\ZZ_\infty)(y,u)=\int_{U}f(y,u) d\rho(u)$ for all $f\in L^1(X,\mu)$.
As usual, we assume that $X, Z_\infty(S)$ and $U$ are all compact metric spaces.

\medskip

First we show that it suffices to show \eqref{h1} holds for all continuous functions. To see this, first note that in $L^1(X,\mu)$ the set $\{g\in C(X): \E_\mu(g|\ZZ_\infty(S))=0\}$ is dense in $\{f\in L^1(X,\mu): \E_\mu(f|\ZZ_\infty(S))=0\}$, i.e. for each $f\in L^1(X,\mu)$ with $\E_\mu(f|\ZZ_\infty(S))=0$ and $\ep>0$, there is some $g\in C(X)$ such that $\|f-g\|_{L^1(\mu)}\le\ep$ and $\E_\mu(g|\ZZ_\infty(S))=0$. In fact, for each $f\in L^1(X,\mu)$ with $\E_\mu(f|\ZZ_\infty(S))=0$ and $\ep>0$, choose $\widetilde{g}\in C(X)$ such that $\|f-\widetilde{g}\|_{L^1(\mu)}<\frac{\ep}{2}$. Let $g(y,u)=\widetilde{g}-\E_\mu(\widetilde{g}|\ZZ_\infty(S) ) =\widetilde{g}(y,u)-\int_U\widetilde{g}(y,u)d\rho(u)$. Then $g\in C(X)$, $\E_\mu(g|\ZZ_\infty(S))=0$, and
\begin{equation*}
\begin{split}
 \|f-g\|_{L^1(\mu)}&=\big\| (f-\E_\mu(f|\ZZ_\infty(S)))-(\widetilde{g}-\E_\mu(\widetilde{g}|\ZZ_\infty(S))) \big\|_{L^1(\mu)}\\
 & \le \|f-\widetilde{g}\|_{L^1(\mu)}+\| \E_\mu(f-\widetilde{g}|\ZZ_\infty(S)) \|_{L^1(\mu)}\le\ep.
\end{split}
\end{equation*}

Next similar to the proof of Theorem \ref{thm-A1}, one can show that if \eqref{h1} holds for all continuous functions, then \eqref{h1} holds for all bounded measurable functions. Thus it suffices to show \eqref{h1} holds for all continuous functions.

\medskip

Arguing by contradiction, suppose that the conclusion fails. Then there are some $l\in \N$, $f_i, g_j^{(t)}\in C(X), 1\le i\le m, 1\le t\le d, -l\le j\le l$ with $\E_\mu(g_j^{(t)}|\ZZ_\infty(S))=0$ for some $t\in [1,d], j\in [-l,l]$, one has
\begin{equation*}
  \lim_{N\to\infty} \frac{1}{N}\sum_{n=0}^{N-1} T^{c_1n}f_1\cdots T^{c_mn}f_m\cdot \prod_{j=-l}^l \Big(S^{p_1(n+j)}g_j^{(1)}\cdots S^{p_d(n+j)}g_j^{(d)}\Big)=0
\end{equation*}
does not hold in $L^2(X,\mu)$.
Then there exist $\ep>0$ and $N_i\to\infty$ such that
\begin{equation}\label{j1}
  \left\|\frac{1}{N_i}\sum_{n=0}^{N_i-1} T^{c_1n}f_1\cdots T^{c_mn}f_m\cdot \prod_{j=-l}^l \Big(S^{p_1(n+j)}g_j^{(1)}\cdots S^{p_d(n+j)}g_j^{(d)}\Big) \right\|_{L^2(\mu)}\ge \ep
\end{equation}
for every $i\in \N$.

By Remark \ref{rem-mea2}-(3), there is a subset $X_0\in \X$ with $\mu(X_0)=1$ and subsequence $\{N_i'\}_{i\in \N}$ of $\{N_i\}_{i\in \N}$ such that for all $x\in X_0$,
in ${\mathcal M}(X^m)$ and ${\mathcal M}((X^{d})^\Z)$
  $$\lim_{i\to\infty}\frac{1}{N_i'}\sum_{n=0}^{N_i'-1}\d_{\tau_{\vec{c}}^nx^{\otimes m}}=\mu^{(m)}_{\vec{c},x} \quad \text{and }\quad   \lim_{i\to\infty}\frac{1}{N_i'}\sum_{n=0}^{N_i'-1}\d_{{\sigma}^n\w^\A_x}= {\mu}^{(\infty)}_{\A,x}, $$
where $\w^\A_x=\big(S^{\vec{p}(n)}x^{\otimes d}\big)_{n\in \Z} \in (X^d)^\Z$, $S^{\vec{p}(n)}=S^{p_1(n)}\times \cdots \times S^{p_d(n)}$.

By Theorem \ref{measure-like}, $${\mu}^{(\infty)}_{\A,x}=({\mu_\infty})^{(\infty)}_{\A,\pi_\infty(x)}\times (\rho^{d})^\Z,$$
and $((X^d)^\Z, (\X^d)^\Z, {\mu}^{(\infty)}_{\A,x}, \sigma)$ is the product of a $\infty$-step pro-nilsystem
$$((Z_\infty(S)^d)^\Z, (\ZZ_\infty(S)^d)^\Z, {(\mu_\infty)}^{(\infty)}_{\A,x}, \sigma)$$ and a Bernoulli system $((U^d)^\Z, ( {\mathcal U}^d)^\Z, (\rho^d)^\Z,\sigma)$.

By Lemma \ref{lem-zero-entropy},  there is an $X^*\in \X$ with a full measure $\mu(X^*)=1$ such that for all $x\in X^*$, $(X^m, \X^m, \mu^{(m)}_{\vec{c},x}, \tau_{\vec{c}})$ has zero entropy.

By the proof of the last equality of \eqref{b1} in Theorem \ref{thm-ergodic-measures}, there is some $X'\in \X$ with $\mu(X')=1$ such that for each $x\in X'$ and for $({\mu_\infty})_{\A,\pi_\infty(x)}^{(\infty)}$-a.e.  $ {\bf z}\in (Z_\infty(S)^{d})^{\Z}$,
\begin{equation}\label{b2}
\prod_{j=-l}^l \big(\int_U g_j^{(1)} d\rho \cdot \cdots \cdot \int_U g_j^{(d)} d \rho\big)
  = \prod_{j=-l}^l \big(\E_\mu( g_j^{(1)} |\ZZ_\infty(S)) \cdot \cdots \cdot \E_\mu(g_j^{(d)}|\ZZ_\infty(S) )\big) .
\end{equation}

Now let $x\in X_0\cap X^*\cap X'$, and let $\lambda_x$ be a limit point of the following sequence in ${\mathcal M}(X^m\times (X^{d})^\Z)$
$$\Big\{\frac{1}{N_i'}\sum_{n=0}^{N_i'-1} \d_{(\tau_{\vec{c}}^nx^{\otimes m}, \sigma ^n \w_x^{\A})}\Big\}_{i\in \N}.$$
 That is, there is some subsequence $\{N_i''\}_{i\in \N}$ of $\{N_i'\}_{i\in \N}$ such that
\begin{equation}\label{j2}
  \lambda_x=\lim_{i\to\infty}\frac{1}{N_i''}\sum_{n=0}^{N_i''-1} \d_{(\tau_{\vec{c}}^nx^{\otimes m}, \sigma ^n \w_x^{\A})}, \quad \text{weak$^*$ in ${\mathcal M}(X^m\times (X^d)^{\Z})$}.
\end{equation}
Thus $\lambda_x$ is a joining of $(X^m, \X^m, \mu^{(m)}_{\vec{c},x}, \tau_{\vec{c}})$ and $((X^d)^\Z, (\X^d)^\Z, {\mu}^{(\infty)}_{\A,x}, \sigma)$.

Since $((X^d)^\Z, (\X^d)^\Z, {\mu}^{(\infty)}_{\A,x}, \sigma)$ is the product of an $\infty$-step pro-nilsystem
$$((Z_\infty(S)^d)^\Z, (\ZZ_\infty (S)^d)^\Z, {(\mu_\infty)}^{(\infty)}_{\A,x}, \sigma)$$ and a Bernoulli system $((U^d)^\Z,({\mathcal U}^d)^\Z, (\rho^d)^\Z,\sigma)$, by Lemma \ref{lem-pin} the $\infty$-step pro-nilfactor $((Z_\infty(S)^d)^\Z,  (\ZZ_\infty (S)^d)^\Z,  {(\mu_\infty)}^{(\infty)}_{\A,\pi_\infty(x)},  \sigma)$ is also
the Pinsker factor of $((X^d)^\Z, (\X^d)^\Z, {\mu}^{(\infty)}_{\A,x}, \sigma)$.

Since $\E_\mu(g_j^{(t)}|\ZZ_\infty(S))=0$ for some $t\in \{1,2,\ldots, d\}, j\in \{-l, -l+1,\ldots, l\}$ and ${\mu}^{(\infty)}_{\A,x}=({\mu_\infty})^{(\infty)}_{\A,\pi_\infty(x)}\times (\rho^{d})^\Z$, we have
\begin{equation}\label{j3}
\begin{split}
  &\quad \E_{ {\mu}^{(\infty)}_{\A,x}}\Big( \bigotimes_{j=-l}^l \big(g_j^{(1)}\otimes \cdots \otimes g_j^{(d)}\big) | (\ZZ_\infty(S)^d)^\Z \Big)\\
  &=\int_{(Z_\infty(S)^d)^\Z}\int_{(U^d)^\Z} \Big(\bigotimes_{j=-l}^l \big(g_j^{(1)}\otimes \cdots \otimes g_j^{(d)}\big) d (\rho^d)^\Z \Big) d {(\mu_\infty)}^{(\infty)}_{\A,\pi_\infty(x)}
  \\&= \int_{(Z_\infty(S)^d)^\Z} \Big(\prod_{j=-l}^l \big(\int_U g_j^{(1)} d\rho \cdot \cdots \cdot \int_U g_j^{(d)} d \rho\big) \Big) d {(\mu_\infty)}^{(\infty)}_{\A,\pi_\infty(x)}\\
\end{split}
\end{equation}
which is equal to
\begin{equation*}\label{j3}
\begin{split}
  \overset{\eqref{b2}}=\int_{(Z_\infty(S)^d)^\Z} \Big(\prod_{j=-l}^l \big(\E_\mu( g_j^{(1)} |\ZZ_\infty(S)) \cdot \cdots \cdot \E_\mu(g_j^{(d)}|\ZZ_\infty(S) )\big) \Big) d {(\mu_\infty)}^{(\infty)}_{\A,\pi_\infty(x)}=0.
\end{split}
\end{equation*}

Since $(X^m, \X^m, \mu^{(m)}_{\vec{c},x}, \tau_{\vec{c}})$ has zero entropy and $((Z_\infty(S)^d)^\Z, (\ZZ_\infty (S)^d)^\Z, {(\mu_\infty)}^{(\infty)}_{\A,\pi_\infty(x)}, \sigma)$ is the Pinsker factor of $((X^d)^\Z, (\X^d)^\Z, {\mu}^{(\infty)}_{\A,x}, \sigma)$, by Theorem \ref{Pinsker} and \eqref{j3}, one has
\begin{equation}\label{j4}
  \int_{X^m\times (X^d)^\Z} \bigotimes_{i=1}^m f_i \otimes \bigotimes_{j=-l}^l \big(g_j^{(1)}\otimes \cdots \otimes g_j^{(d)}\big) d \lambda_x =0.
\end{equation}
By \eqref{j2}, for all $x\in X_0\cap X^*\cap X'$,
\begin{equation*}
  \begin{split}
      & \quad \lim_{i\to\infty} \frac{1}{N_i''}\sum_{n=0}^{N_i''-1} f_1(T^{c_1n}x)\cdots f_m(T^{c_mn}x)\cdot \prod_{j=-l}^l\big( g_j^{(1)}(S^{p_1(n+j)}x)\cdots g_j^{(d)}(S^{p_d(n+j)}x) \big) \\
       & =\int_{X^m\times (X^d)^\Z} \bigotimes_{i=1}^m f_i \otimes \bigotimes_{j=-l}^l \big( g_j^{(1)}\otimes \cdots \otimes g_j^{(d)}\big) d \lambda_x \overset{\eqref{j4}} =0.
   \end{split}
\end{equation*}
Since $\{N_i''\}_{i\in \N}$ is a subsequence of $\{N_i\}_{i\in \N}$ and $\mu(X_0\cap X^*\cap X')=1$, the equation above contradicts with \eqref{j1}.
The proof is complete.
\end{proof}

Now it is time to prove the main result of this subsection.

\begin{thm}\label{thm-mean}
Let $T,S$ be ergodic measure preserving transformations acting on a probability space $(X,\X,\mu)$ such that $(X,\X,\mu,T)$ has zero entropy. Let $c_1,\ldots, c_m\in \Z\setminus \{0\}$ and $p_1,\ldots,p_d$ be non-constant integral polynomials with $\deg {p_i}\ge 2, 1\le i\le d$. Then for all $f_1,\ldots, f_m, g_1,\ldots, g_d\in L^\infty(X,\mu)$, the limit
\begin{equation}\label{}
  \lim_{N\to\infty} \frac{1}{N}\sum_{n=0}^{N-1} T^{c_1n}f_1\cdots T^{c_mn}f_m\cdot S^{p_1(n)}g_1\cdots S^{p_d(n)}g_d
\end{equation}
exists in $L^2(X,\mu)$.
\end{thm}

\begin{proof}
We assume that the family $\A=\{p_1,\ldots, p_d\}$ satisfies $(\spadesuit)$. If not, one can handle the proof using the analysis in the proof of Proposition \ref{prop-pw}.

Since $\A=\{p_1,\ldots, p_d\}$ satisfies $(\spadesuit)$, by Proposition \ref{prop-reduing} we have
$$\lim_{N\to\infty} \Big\|\frac{1}{N}\sum_{n=0}^{N-1} T^{c_1n}f_1\cdots T^{c_m n}f_m \cdot S^{p_1(n)}g_1\cdots S^{p_d(n)}g_d\Big\|_{L^2(\mu)}=0,$$
whenever $\E(g_j|\ZZ_\infty(S))=0$ for some $j\in \{1,2,\ldots, d\}$.
Thus we can restrict to the case where $g_j, 1\le j\le d$, is measurable with respect to $\ZZ_\infty(S)$ and using an $L^2(\mu)$ approximation argument we can assume that it is measurable with respect to the factor $\ZZ_k(S)$ for some $k\in \N$.

By Theorem \ref{thm-CFH}, for every $\ep>0$ there exists $\widetilde{g}_j\in L^\infty(X,\mu)$ such that $\widetilde{g}_j$ is $\ZZ_k(S)$-measurable and $\|g_j-\widetilde{g}_j\|_{L^1(X,\mu)}<\ep, 1\le j\le d$, and for $\mu$-almost all $x\in X$ the sequence
$(\widetilde{g}_j(S^nx))_{n\in \Z}$ is a $k$-step nilsequence.
By \cite[Proposition 3.14]{Leibman05-Isr} or \cite[Theorem 14.15]{HK18}, for $\mu$-almost all $x\in X$ the sequence $(\widetilde{g}_j(S^{p_j(n)}x))_{n\in \Z}$ is a $(k\deg{p_j})$-step nilsequence for $1\le j\le d$. Let $X_1\in \X$ be a full measure subset of $X$ for which
this property holds.
Thus $(\Psi_n(x))_{n\in \Z}=\Big(\prod_{j=1}^d\widetilde{g}_j(S^{p_j(n)}x)\Big)_{n\in \Z}$ is a nilsequence.
Therefore, it suffices to show the existence in $L^2(X,\mu)$ of the limit
\begin{equation}\label{}
  \lim_{N\to\infty} \frac{1}{N}\sum_{n=0}^{N-1}\Psi_n(x) f_1(T^{c_1n}x) f_2(T^{c_2n}x)\cdots f_m (T^{c_m n}x)
\end{equation}
This follows from Theorem \ref{plus-nil}. The proof is complete.
\end{proof}



\subsection{Pointwise convergence}\
\medskip

In this subsection we give some pointwise convergence results assuming some natural conjectures.

\subsubsection{The main result}

Before presenting the result we state two natural conjectures. First we recall the following well known result of
Bourgain. 

\begin{thm}[Bourgain's double recurrence theorem] \cite{B90}\label{Bouragain-double}
Let $(X,\X,\mu, T)$ be a m.p.s. and $c_1, c_2\in \Z$. Then for all $f,g\in L^\infty(X,\mu)$ the averages
$$\displaystyle \frac 1 N\sum_{n=0}^{N-1}f(T^{c_1n}x)g(T^{c_2n}x)$$
converge almost surely.
\end{thm}

In general, one has the following well-known conjecture.

\begin{conj}\label{conj1}
Let $(X,\X,\mu, T)$ be a m.p.s. and $d\in \N$. Then for all non-constant integral polynomials $p_1, \cdots, p_d$ and $f_1, \cdots, f_d \in L^{\infty}(X, \mu)$, the averages
\begin{equation*}
    \frac 1 N\sum_{n=0}^{N-1}f_1(T^{p_1(n)}x) f_2(T^{p_2(n)}x)\cdots f_d(T^{p_d(n)}x)
\end{equation*}
converge almost surely.
\end{conj}

For the recent progress along the conjecture, see \cite{KMT}. The following results is a generalization of Wiener-Wintner Theorem.

\begin{thm}\cite[Theorem 2.22]{HK09}\cite[Theorem 23.14]{HK18}\label{WW-Convergence}
Let $(X,\X,\mu,T)$ be a m.p.s. and $f\in L^\infty(X,\mu)$. Then there exits a subset $X_0\in \X$ with $\mu(X_0)=1$ such that the limit
$$\lim_{N\to\infty}\frac{1}{N}\sum_{n=0}^{N-1} \Psi(n)f(T^nx)$$
exits for all $x\in X_0$ and every nilsequence $(\Psi(n))_{n\in \Z}$.

Moreover, if $\E(f|\ZZ_\infty(T))=0$, then the for all $x\in X_0$, $\lim_{N\to\infty}\frac{1}{N}\sum_{n=0}^{N-1} \Psi(n)f(T^nx)=0$.
\end{thm}

It is natural to conjecture the following generalization of Theorem \ref{WW-Convergence}, which is also a pointwise version of Theorem \ref{plus-nil}.

\begin{conj}\label{conj2}
Let $(X,\X,\mu,T)$ be a m.p.s. and $a_1,a_2,\ldots, a_m\in \Z$ and $f_1, \ldots, f_m \in L^\infty(X,\mu)$. Then there exits a subset $X_0\in \X$ with $\mu(X_0)=1$ such that the limit
$$\lim_{N\to\infty}\frac{1}{N}\sum_{n=0}^{N-1} \Psi(n)f_1(T^{a_1n}x)\cdots f_m(T^{a_mn}x)$$
exits for all $x\in X_0$ and every nilsequence $(\Psi(n))_{n\in \Z}$.
\end{conj}

When $m=2$, see \cite{AssaniMoore18} for some related results about Conjecture \ref{conj2}.

\medskip

Assuming the Conjectures \ref{conj1} and \ref{conj2}, we can prove the following pointwise version of Theorem \ref{thm-mean}:

\begin{thm}\label{thm-pointwise}
Let $T,S$ be ergodic measure preserving transformations acting on a probability space $(X,\X,\mu)$ such that $(X,\X,\mu,T)$ has zero entropy. Let $a_1,\ldots, a_m\in \Z$, $p_1,\ldots,p_d$ be integral polynomials with $\deg {p_i}\ge 2, 1\le i\le d$. If Conjectures \ref{conj1} and \ref{conj2} hold, then for all $f_1,\ldots, f_m, g_1,\ldots, g_d\in L^\infty(X,\mu)$, the averages
\begin{equation}\label{pp1}
  \frac{1}{N}\sum_{n=0}^{N-1} f_1(T^{c_1n}x)\cdots f_m(T^{c_mn}x)\cdot S^{p_1(n)}g_1\cdots S^{p_d(n)}g_d
\end{equation}
converge almost surely.
\end{thm}

\subsubsection{Proof of Theorem \ref{thm-pointwise}}

To show the main result of this subsection we need the following

\begin{prop}\label{prop-pointwise-distal}
Let $T,S$ be ergodic measure preserving transformations acting on a probability space $(X,\X,\mu)$ and $T$ 
has zero entropy. Let $c_1,\ldots, c_m$ be distinct non-zero integers and let $p_1,\ldots,p_d$ be integral polynomials with $\deg {p_i}\ge 2, 1\le i\le d$ and satisfying $(\spadesuit)$. If Conjecture \ref{conj1} holds, then for all $l\in \N$, for all $f_i, g_j^{(t)}\in L^\infty(X,\mu), 1\le i\le m, 1\le t\le d, -l\le j\le l$ with $\E_\mu(g_j^{(t)}|\ZZ_\infty(S))=0$ for some $t\in [1,d], j\in [-l,l] $, one has
\begin{equation}\label{h2}
  \lim_{N\to\infty} \frac{1}{N}\sum_{n=0}^{N-1} T^{c_1n}f_1\cdots T^{c_mn}f_m\cdot \prod_{j=-l}^l \big( S^{p_1(n+j)}g_j^{(1)}\cdots S^{p_d(n+j)}g_j^{(d)}\big)=0
\end{equation}
almost surely.
\end{prop}

\begin{proof}
Let $\pi_{\infty}: (X,\X,\mu, S)\rightarrow (Z_{\infty}(S),\ZZ_\infty(S),\mu_\infty, S)$
be the factor map, where $Z_{\infty}(S)$ is the $\infty$-step pro-nilfactor of $(X,\X,\mu,S)$.
By Theorem \ref{Rohlin}, we may assume that there exist a probability pace $(U, {\mathcal U},\rho)$ and a measurable cocycle $\a: Z_\infty(S)\rightarrow {\rm Aut} (U,\rho)$ such that
$$(X,\X,\mu,S)=(Z_\infty(S) \times U, \ZZ_\infty(S) \times {\mathcal U}, \mu_\infty \times \rho, S_\a).$$
We assume that $X,\ Z_\infty(S)$ and $U$ are all compact metric spaces.

\medskip

By Proposition \ref{prop-DL}, one has that if \eqref{h2} holds for all continuous functions, then \eqref{h2} holds for all bounded measurable functions. Thus it suffices to show \eqref{h2} holds for all continuous functions.

\medskip

By Conjecture \ref{conj1} and Proposition \ref{prop-pw}, there is a subset $X_0\in \X$ with $\mu(X_0)=1$ such that for all $x\in X_0$, we have that in ${\mathcal M}(X^m)$ and $M((X^{d})^\Z)$
  $$\lim_{N\to\infty}\frac{1}{N}\sum_{n=0}^{N-1}\d_{\tau_{\vec{c}}^nx^{\otimes m}}=\mu^{(m)}_{\vec{c},x} \quad \text{and }\quad \lim_{N\to\infty}\frac{1}{N}\sum_{n=0}^{N-1}\d_{{\sigma}^n\w^\A_x}= {\mu}^{(\infty)}_{\A,x},$$
where $\w^\A_x=\big(S^{\vec{p}(n)}x^{\otimes d}\big)_{n\in \Z} \in (X^d)^\Z$, $S^{\vec{p}(n)}=S^{p_1(n)}\times \ldots \times S^{p_d(n)}$.

By Theorem \ref{measure-like},
${\mu}^{(\infty)}_{\A,x}=({\mu_\infty})^{(\infty)}_{\A,\pi_\infty(x)}\times (\rho^{d})^\Z,$
and $((X^d)^\Z, (\X^d)^\Z, {\mu}^{(\infty)}_{\A,x}, \sigma)$ is the product of an $\infty$-step pro-nilsystem  $((Z_\infty(S)^d)^\Z, (\ZZ_\infty(S)^d)^\Z, {(\mu_\infty)}^{(\infty)}_{\A,\pi_\infty(x)}, \sigma)$ and a Bernoulli system $((U^d)^\Z, ({\mathcal U}^d)^\Z, (\rho^d)^\Z,\sigma)$.

By Lemma \ref{lem-zero-entropy},  there is an $X^*\in \X$ with a full measure $\mu(X^*)=1$ such that for all $x\in X^*$, $(X^m, \X^m, \mu^{(m)}_{\vec{c},x}, \tau_{\vec{c}})$ has zero entropy, where  $\vec{c}=(c_1,c_2,\ldots, c_m)$.

By the proof of the last equality of \eqref{b1} in Theorem \ref{thm-ergodic-measures}, there is some $X'\in \X$ with $\mu(X')=1$ such that for each $x\in X'$ and for $({\mu_\infty})_{\A,\pi_\infty(x)}^{(\infty)}$-a.e. $ {\bf z}\in (Z_\infty(S)^{d})^{\Z}$,
\begin{equation*}
\prod_{j=-l}^l \big(\int_U g_j^{(1)} d\rho \cdot \cdots \cdot \int_U g_j^{(d)} d \rho\big)
  = \prod_{j=-l}^l \big(\E_\mu( g_j^{(1)} |\ZZ_\infty(S)) \cdot \cdots \cdot \E_\mu(g_j^{(d)}|\ZZ_\infty(S) )\big) .
\end{equation*}

Let $x\in X_0\cap X^*\cap X'$ and $\lambda_x$ be a limit point of the following sequence in ${\mathcal M}(X^m\times (X^{d})^\Z)$
$$\Big\{\frac{1}{N}\sum_{n=0}^{N-1} \d_{(\tau_{\vec{c}}^nx^{\otimes m}, \sigma ^n \w_x^{\A})}\Big\}_{N\in \N}.$$
Thus, 
there is some subsequence $\{N_i\}_{i\in \N}$ of $\N$ such that
\begin{equation}\label{jj2}
  \lambda_x=\lim_{i\to\infty}\frac{1}{N_i}\sum_{n=0}^{N_i-1} \d_{(\tau_{\vec{c}}^nx^{\otimes m}, \sigma ^n \w_x^{\A})}, \quad \text{weak$^*$ in ${\mathcal M}(X^m\times (X^d)^{\Z})$}.
\end{equation}
Thus $\lambda_x$ is a joining of $(X^m, \X^m, \mu^{(m)}_{\vec{a},x}, \tau_{\vec{a}})$ and $((X^d)^\Z, (\X^d)^\Z, {\mu}^{(\infty)}_{\A,x}, \sigma)$.

Since $((X^d)^\Z, (\X^d)^\Z, {\mu}^{(\infty)}_{\A,x}, \sigma)$ is the product of an $\infty$-step pro-nilsystem
$$((Z_\infty(S)^d)^\Z, (\ZZ_\infty (S)^d)^\Z, {(\mu_\infty)}^{(\infty)}_{\A,\pi_\infty(x)}, \sigma)$$ and a Bernoulli system $((U^d)^\Z, ({\mathcal U}^d)^\Z, (\rho^d)^\Z,\sigma)$, by Lemma \ref{lem-pin} the $\infty$-step pro-nilfactor $((Z_\infty(S)^d)^\Z, (\ZZ_\infty (S)^d)^\Z, {(\mu_\infty)}^{(\infty)}_{\A,x}, \sigma)$ is also
the Pinsker factor of $((X^d)^\Z, (\X^d)^\Z, {\mu}^{(\infty)}_{\A,x}, \sigma)$.

Since $\E_\mu(g_j^{(t)}|\ZZ_\infty(S))=0$ for some $t\in [1,d], j\in [-l,l]$, by \eqref{j3}
\begin{equation}\label{jj3}
 \E_{{\mu}^{(\infty)}_{\A,x}}\Big( \bigotimes_{j=-l}^l \big(g_j^{(1)}\otimes \cdots \otimes g_j^{(d)}\big) | (\ZZ_\infty(S)^d)^\Z \Big)=0.
\end{equation}
Since $(X^m, \X^m, \mu^{(m)}_{\vec{c},x}, \tau_{\vec{c}})$ has zero entropy and $((Z_\infty(S)^d)^\Z, (\ZZ_\infty (S)^d)^\Z, {(\mu_\infty)}^{(\infty)}_{\A,x}, \sigma)$ is the Pinsker factor of $((X^d)^\Z, (\X^d)^\Z, {\mu}^{(\infty)}_{\A,x}, \sigma)$, by Theorem \ref{Pinsker} and \eqref{jj3},
\begin{equation}\label{jj4}
  \int_{X^m\times (X^d)^\Z} \bigotimes_{i=1}^m f_i \otimes \bigotimes_{j=-l}^l \big(g_j^{(1)}\otimes \cdots \otimes g_j^{(d)}\big) d \lambda_x =0.
\end{equation}
By \eqref{jj2} and \eqref{jj4}, for all $x\in X_0\cap X^*\cap X'$,
\begin{equation*}
  \begin{split}
      & \quad \lim_{i\to\infty} \frac{1}{N_i}\sum_{n=0}^{N_i-1} T^{c_1n}f_1\cdots T^{c_mn}f_m\cdot \prod_{j=-l}^l \big( S^{p_1(n+j)}g_j^{(1)}\cdots S^{p_d(n+j)}g_j^{(d)}\big)  \\
       & =\int_{X^s\times (X^d)^\Z} \bigotimes_{i=1}^m f_i \otimes \bigotimes_{j=-l}^l\big( g_j^{(1)}\otimes \cdots \otimes g_j^{(d)}\big) d \lambda_x =0.
   \end{split}
\end{equation*}
Since $\{N_i\}_{i\in \N}$ is an arbitrary subsequence of $\N$, one has
\begin{equation}\label{}
  \lim_{N\to\infty} \frac{1}{N}\sum_{n=0}^{N-1} T^{c_1n}f_1\cdots T^{c_mn}f_m\cdot \prod_{j=-l}^l \big( S^{p_1(n+j)}g_j^{(1)}\cdots S^{p_d(n+j)}g_j^{(d)}\big)=0
\end{equation}
almost surely.
The proof is complete.
\end{proof}

Similar to the proof of Proposition \ref{prop-pw}, we can remove the condition  $(\spadesuit)$ in Proposition \ref{prop-pointwise-distal}.

\begin{cor}\label{cor-pw}
Let $T,S$ be ergodic measure preserving transformations acting on a probability space $(X,\X,\mu)$ such that $(X,\X,\mu,T)$ has zero entropy. Let $c_1,\ldots,c_m$ be distinct non-zero integers and let $p_1,\ldots,p_d$ be integral polynomials with $\deg {p_i}\ge 2, 1\le i\le d$. If Conjecture \ref{conj1} holds, then for all $l\in \N$, for all $f_i, g_j\in L^\infty(X,\mu), 1\le i\le m, 1\le j\le d$ with $\E_\mu(g_j|\ZZ_\infty(S))=0$ for some $j\in \{1,2,\ldots, d\}$, one has
\begin{equation}\label{}
  \lim_{N\to\infty} \frac{1}{N}\sum_{n=0}^{N-1} T^{c_1n}f_1\cdots T^{c_m n}f_m\cdot S^{p_1(n)}g_1\cdots S^{p_d(n)}g_d=0
\end{equation}
almost surely.
\end{cor}

Finally, we can prove Theorem \ref{thm-pointwise}.

\begin{proof}
By Corollary \ref{cor-pw}, in order to prove pointwise convergence
of the averages \eqref{pp1}, we can assume that all functions $g_1,\ldots, g_d$ are measurable with respect to $\ZZ_\infty(S)$.

By Proposition \ref{prop-DL}, for $f_1,\ldots, f_m$ fixed, the family of functions $g_i$ such that the convergence \eqref{pp1} holds almost
everywhere is closed in $L^2(X,\mu)$. Therefore, in order to prove the existence almost everywhere
of the limit \eqref{pp1} for $g_1,\ldots, g_d$ measurable with respect to $\ZZ_\infty(S)$, it suffices to restrict to
the case where $g_1,\ldots, g_d$ are measurable with respect to $\ZZ_k(S)$ for some $k\in \N$. As showed in the proof of Theorem \ref{thm-mean}, there is a full measure subset $X_1$ of $X$ for which
$(\Psi_n(x))_{n\in \Z}=\Big(\prod_{j=1}^d{g}_j(S^{p_j(n)}x)\Big)_{n\in \Z}$ is a nilsequence.

By Theorem \ref{WW-Convergence} there exits a subset $X_0\in \X$ with $\mu(X_0)=1$ such that the limit
$$\lim_{N\to\infty}\frac{1}{N}\sum_{n=0}^{N-1} \Psi(n)f_1(T^{a_1n}x)\cdots f_m(T^{a_mn}x)$$
exits for all $x\in X_0$ and every nilsequence $(\Phi(n))_{n\in \Z}$. Now let $X_2=X_0\cap X_1$. Then $\mu(X_2)=1$, and for all $x\in X_2$, the limit
\begin{equation}\label{pp2}
\begin{split}
  &\quad \lim_{N\to\infty} \frac{1}{N}\sum_{n=0}^{N-1}\Psi_n(x) f_1(T^{a_1n}x)\cdots f_m(T^{a_mn}x) \\ & = \lim_{N\to\infty}\frac{1}{N}\sum_{n=0}^{N-1} f_1(T^{a_1n}x)\cdots f_m(T^{a_mn}x)\cdot g_1(S^{p_1(n)}x)\cdots g_d(S^{p_d(n)}x)
\end{split}
\end{equation}
exists.
The proof is complete.
\end{proof}

\section{Relations with Furstenberg systems}\label{section-Furst-systems}

In this section, we study the relations of $\mu^{(\infty)}_\A$ with Furstenberg systems of sequences, and give a positive answer to Problem 1 in \cite{Fran22}.

\subsection{Furstenberg systems of sequences}\
\medskip

First we introduce the notion of Furstenberg systems of sequences from \cite{FranHost18, FranHost21, Fran22}.

\subsubsection{}
Let $\{N_k\}_{k\in \N}$ be a sequence of $\N$ with $N_k\to\infty$ and $I$ be a
compact interval in $\R$. We say that the sequence $z=(z_n)_{n\in \Z}=(z(n))_{n\in \Z}$ with values in $I$
{\em admits correlations} on $\{N_k\}_{k\in \N}$, if the limits
\begin{equation*}
\lim_{k\to\infty} \frac{1}{N_k} \sum_{n=1}^{N_k} \prod_{j=1}^s z(n+n_j)
\end{equation*}
exist for all $s\in \N$ and all $n_1,\ldots, n_s\in\Z$ (not necessarily distinct).
Note that for a given sequence $z=(z_n)_{n\in \Z}$  with values in $I$ (i.e. $z: \Z\rightarrow I$), every sequence of $\{N_k\}_{k\in \N}$ has a subsequence $\{N'_k\}_{k\in \N}$ such
that the sequence $z$ admits correlations on $\{N'_k\}_{k\in \N}$.

\subsubsection{Furstenberg system associated with $z$}
Let $I$ be a
compact interval in $\R$.
If a sequence $z: \Z\rightarrow I$ admits correlations on a given sequence of positive imtegers, then by the correspondence principle of Furstenberg \cite{F77, F} we may associate
a measure preserving system that captures the statistical properties of this sequence.

Let $\Omega=I^\Z$ and $\sigma: \Omega\rightarrow \Omega$ be the shift: $(\sigma\w)(n)=\w(n+1), \forall n\in \Z$. If a sequence $z: \Z\rightarrow I$ admits correlations on $\{N_k\}_{k\in \N}$, then for all $f\in C(\Omega)$ the following limit exist
$$\lim_{k\to\infty} \frac{1}{N_k} \sum_{n=1}^{N_k} f(\sigma^n z).$$
Hence the following limit exists in ${\mathcal M} (\Omega)$:
$$\nu=\lim_{k\to\infty} \frac{1}{N_k} \sum_{n=1}^{N_k}\d_{\sigma^n z}$$
and we say that the point $z$ is {\em generic for $\nu$ along $\{N_k\}_{k\in \N}$}.

\begin{de}
Let $I$ be a compact interval of $\R$ and let $z: \Z\rightarrow I$ be a sequence that
admits correlations on $\{N_k\}_{k\in \N}$, and $(\Omega, \nu, \sigma)$ as above.
\begin{itemize}
  \item $(\Omega, \nu, \sigma)$ is called the {\em Furstenberg system associated with $z$ on $\{N_k\}_{k\in \N}$}.
  \item Let $F_0\in C(\Omega)$ be defined by $F_0(\w)=\w(0), \forall \w \in \Omega$, i.e. the $0^{\rm th}$-coordinate projection. Then $F_0(\sigma^n z)=z(n), \forall n\in \Z$ and
      \begin{equation}\label{}
        \lim_{k\to\infty} \frac{1}{N_k} \sum_{n=1}^{N_k} \prod_{j=1}^s z(n+n_j)=\int_{\Omega} \prod_{j=1}^s \sigma^{n_j}F_0 d\nu
      \end{equation}
      for all $s\in \N$ and all $n_1,\ldots, n_s\in \Z$. This identity is referred to as the {\em Furstenberg
correspondence principle}.
\end{itemize}
\end{de}

\begin{rem} We have
\begin{enumerate}
  \item A sequence $z: \Z\rightarrow I$ may have several non-isomorphic Furstenberg systems
depending on which sequence of $\{N_k\}_{k\in \N}$ we use in the evaluation of its correlations. We call any such system a Furstenberg system of $z$.
  \item We say that the sequence $z$ has a {\em unique Furstenberg system}, if $z$ admits correlations
on $\{N\}_{N\in \N}$, or equivalently,
$$\nu=\lim_{N\to\infty} \frac{1}{N} \sum_{n=1}^{N}\d_{\sigma^n z}$$
exists.
\end{enumerate}
\end{rem}

\subsubsection{Problems about Furstenberg systems}

The following is the Problem 1 in \cite{Fran22} (see also Remark 3.2 in \cite{FranHost21}).

\begin{prob}\label{problem1}
Let $(X,\X,\mu, T)$ be an ergodic m.p.s., and let $p$ be a non-linear integral polynomial, and $f\in L^\infty(X,\mu)$. Show that for almost every $x\in X$ the sequence $\big(f(T^{p(n)})x\big)_{n\in \Z}$ has a unique Furstenberg system that is ergodic and isomorphic to a direct product of an infinite-step pro-nilsystem
and a Bernoulli system.
\end{prob}

As mentioned in \cite{Fran22}, proving uniqueness of the Furstenberg
system seems very hard as this amounts to proving a pointwise convergence result for multiple ergodic averages that currently seems out of reach. So as a first step for
an unconditional result, one probably has to compromise with the following problem.

\medskip

\begin{prob}\label{problem2}
Let $(X,\X,\mu, T)$ be an ergodic m.p.s., and let $p$ be a non-linear integral polynomial, and $f\in L^\infty(X,\mu)$. Show that for any strictly increasing sequence of positive integers $\{N_k\}_{k\in \N}$ there is a subsequence $\{N'_k\}_{k\in \N}$ such that for almost every $x\in X$ the the Furstenberg systems of $\big(f(T^{p(n)})x\big)_{n\in \Z}$ along $\{N_k'\}_{k\in \N}$ are ergodic and isomorphic to direct products of infinite-step pro-nilsystems and Bernoulli systems.
\end{prob}

\subsection{A corollary of Thouvenot's and Host-Kra's theorem}\
\medskip

To answer Problems above, we need to use the following theorems to prove that a factor of a direct product of an ergodic pro-nilsystem with a Bernoulli system is still of the same type. 

Ornstein showed that Bernoulli systems are closed under taking factors \cite{Ornstein70}. Thouvenot extended this result as follows

\begin{thm}[Thouvenot]\cite{Thouvenot}\label{thou}
The factor of a direct product of an ergodic entropy zero transformation with a Bernoulli system is still a direct product of ergodic entropy zero transformation with a Bernoulli system.
\end{thm}

Another well known fact is the following

\begin{thm}[Host-Kra]\cite[Proposition 4.11]{HK05}\label{thm-nil-factor} 
 A factor of $d$-step pro-nilsystem ($d\in \N$) is a $d$-step pro-nilsystem.
\end{thm}


Using Thouvenot's  and Host-Kra's result  we can prove the following result:

\begin{thm}\label{thm-Thouvenot}
The factor of a direct product of an ergodic pro-nilsystem with a Bernoulli system is still a direct product of an ergodic pro-nilsystem with a Bernoulli system.
\end{thm}

\begin{proof}
Let $(X,\X,\mu,T)$ be an ergodic pro-nilsystem and let $(Y,\Y,\nu, S)$ be a Bernoulli system. Let $\pi: (X\times Y,\X\times \Y,\mu\times \nu, T\times S)\rightarrow (Z,\ZZ,\varrho, H)$ be a factor map. Note that $h_\mu(X,T)=0$. By  Theorem \ref{thou}, there is an ergodic m.p.s. $(X',X',\mu',T')$ with $h_{\mu'}(X',T')=0$ and a Bernoulli system $(Y',\Y'\nu',S')$ such that
$$(Z,\ZZ,\varrho, H) = (X'\times Y',\X'\times \Y',\mu'\times \nu', T'\times S').$$
It is well known that the Pinsker factor of the product of two ergodic m.p.s. is the product of their Pinsker factors (see for example \cite[Theorem 18.13.]{Glasner}). Thus
$$\Pi(X\times Y, T\times S)=\Pi(X,T)\times \Pi(Y,S), \quad \Pi(X'\times Y', T'\times S')=\Pi(X',T')\times \Pi(Y',S').$$
Since $\Pi(X,T)=\X$, $\Pi(X',T')=\X'$, and $(Y,\Y,\nu, S)$ and $(Y',\Y',\nu',S')$ are  Bornoulli systems,  it follows $\Pi(X\times Y, T\times S)=\X \times \{Y,\emptyset\}$ and $\Pi(X'\times Y', T'\times S')=\X'\times \{Y',\emptyset\}$. Thus $(X',\X',\mu',T')$ is a factor of $(X,\X,\mu, T)$. By Theorem \ref{thm-nil-factor}, $(X',\X',\mu',T')$ is a pro-nilsystem.
\end{proof}

\subsection{A lemma about topological models}\
\medskip

One can find the following result in \cite[Page 222]{F77}. Since in \cite{F77} no detailed proof was given, for completeness we give a proof in the appendix \ref{appendix2} (see Theorem \ref{top-model2}).

\begin{thm}\label{top-model}
Let $(X,\X,\mu,T)$ be a m.p.s. and let $f\in L^\infty(X,\mu)$. Then there is a t.d.s. $(X',T')$, a $T'$-invariant measure $\mu'$ and an isomorphism $\psi: (X,\X,\mu,T)\rightarrow (X',\B(X'),\mu',T')$
such that $f'= f\circ \psi^{-1} \pmod{\mu'}$, where $f'\in C(X')$.
\end{thm}

\subsection{Relation between $\infty$-joining and Furstenberg systems of sequences}\
\medskip

Now 
we prove the following result, which gives a positive answer to Problem \ref{problem2}. Moreover, the remark after the result gives
a positive answer to Problem \ref{problem1}.

\begin{thm}\label{thm-answer-Fran}
Let $(X,\X,\mu, T)$ be an ergodic m.p.s., and let $p$ be a non-linear integral polynomial with $p(0)=0$, and $f\in L^\infty(X,\mu)$. Then for any strictly increasing sequence of positive integers $\{N_k\}_{k\in \N}$ there is a subsequence $\{N'_k\}_{k\in \N}$  such that for almost every $x\in X$ the Furstenberg systems of $\big(f(T^{p(n)})x\big)_{n\in \Z}$ along $\{N'_k\}_{k\in \N}$ are ergodic and isomorphic to direct products of infinite-step pro-nilsystems and Bernoulli systems.
\end{thm}

\begin{proof}
By Theorem \ref{top-model}, we may assume that $(X,T)$ is a t.d.s. and $f$ is real-valued and continuous. Let $\A=\{p\}$. By Theorem \ref{measure-like}, Theorem \ref{measure-sequence}, Remark \ref{rem-mea}, and Remark \ref{rem-mea2}, for any strictly increasing sequence of positive integers $\{N_k\}_{k\in \N}$ there is a subsequence $\{N'_k\}_{k\in \N}$ and a subset $X_0\in \X$ with $\mu(X_0)=1$ such that for all $x\in X_0$,
  $$\lim_{k\to\infty}\frac{1}{N'_k}\sum_{n=1}^{N'_k}\d_{{\sigma}^n\w^\A_x}= {\mu}^{(\infty)}_{\A,x}, \quad w^*\  \text{in}\ M(X^\Z),$$
and $(X^\Z, \X^\Z, {\mu}^{(\infty)}_{\A,x}, \sigma)$ is isomorphic to the product of an $\infty$-step pro-nilsystem and a Bernoulli system.

Now let $f: X\rightarrow I$, where $I$ is a compact interval of $\R$. Let $z_x: \Z\rightarrow I$ be $z_x(n)=f(T^{p(n)}x)$. For $x\in X_0$, let
$$\phi: X^\Z \rightarrow \Omega=I^\Z, \quad (x_n)_{n\in \Z}\mapsto (f(x_n))_{n\in \Z}.$$
In particular, $\phi (\w_x^\A)=\phi \Big((T^{p(n)}x)_{n\in \Z}\Big)=(f(T^{p(n)}x))_{n\in \Z}=z_x$.

Since $f$ is continuous, we have that $\phi$ is continuous. Then
$$\phi_*: {\mathcal M}(X^\Z)\rightarrow {\mathcal M}(\Omega), \theta \mapsto \theta\circ \phi^{-1}$$
is continuous. Note that
$$\frac{1}{N'_k}\sum_{n=1}^{N'_k}\d_{{\sigma}^n\w^\A_x}\overset{\phi_*}\mapsto \frac{1}{N'_k}\sum_{n=1}^{N'_k}\d_{{\sigma}^nz_x}.$$
Since $\displaystyle \lim_{k\to\infty}\frac{1}{N'_k}\sum_{n=1}^{N'_k}\d_{{\sigma}^n\w^\A_x}= {\mu}^{(\infty)}_{\A,x}$ and $\phi_*$ is continuous, we have that $\displaystyle \lim_{k\to\infty}\frac{1}{N'_k}\sum_{n=1}^{N'_k}\d_{{\sigma}^nz_x}$ exists and denote it by $\nu_x$.
By the definition of Furstenberg systems, $(\Omega, \nu_x, \sigma)$ is the Furstenberg system associated with $z$ on $\{N'_k\}_{k\in \N}$. And
$$\phi: (X^\Z, \mu^{(\infty)}_{\A,x},\sigma)\rightarrow (\Omega, \nu_x, \sigma)$$
is a factor map. By Theorem \ref{thm-Thouvenot}, $(\Omega, \nu_x, \sigma)$ is a direct product of an ergodic pro-nilsystem with a Bernoulli shift. The proof is complete.
\end{proof}

\begin{rem}
Let $(X,\X,\mu, T)$ be an ergodic m.p.s., and let $p$ be a non-linear integral polynomial, and $f\in L^\infty(X,\mu)$.
By Remark \ref{rem-mea}, if one can prove any statement in Proposition \ref{prop-pw}, then in Theorem \ref{measure-like}, $\{N_i\}_{i\in \N}$ is $\{N\}_{N\in \N}$, that is,
      $$\lim_{N\to\infty}\frac{1}{N}\sum_{n=1}^{N}\d_{\widetilde{\sigma}^n\xi^\A_x}= \widetilde{\mu}^{(\infty)}_{\A,x}.$$
In this case for almost every $x\in X$ the sequence $\big(f(T^{p(n)})x\big)_{n\in \Z}$ has a unique Furstenberg system. Then by the same analysis above, one can show that it is ergodic and isomorphic to a direct product of an infinite-step pro-nilsystem and a Bernoulli system, which will give the positive answer to Problem \ref{problem1}.
\end{rem}

\section{Questions}\label{section-ques}

In this section, we give some remarks and more questions.

\subsection{$\mu^{(\infty)}_\A$ for general scheme $\A$}\
\medskip

Let $\phi_1, \phi_2,\ldots, \phi_d: \Z\rightarrow \Z$ be a family of functions taking integer values at the
integers, where $d\in \N$. Then in the same sprit of Section \ref{section-furstenberg-joining}, we can define a measure $\mu^{(\infty)}_\A$ with respect to $\A=\{\phi_1,\phi_2,\ldots, \phi_d\}$.

Let $\vec{\phi}=(\phi_1, \phi_2, \cdots, \phi_d)$ and let $T^{\vec{\phi}(n)}: X^d\rightarrow X^d$
\begin{equation}\label{}
  T^{\vec{\phi}(n)}\Big((x_1, x_2, \ldots, x_d)\Big)=(T^{\phi_1(n)}x_1, T^{\phi_2(n)}x_2,\ldots, T^{\phi_d(n)}x_d).
\end{equation}
For each $x\in X$, define
\begin{equation*}\label{}
  \w_x^\A\triangleq (T^{\vec{\phi}(n)}x^{\otimes d})_{n\in \Z}=(\ldots, T^{\vec{\phi}(-1)}(x^{\otimes d}), \underset{\bullet}{T^{\vec{\phi}(0)}(x^{\otimes d})},T^{\vec{\phi}(1)}(x^{\otimes d}),T^{\vec{\phi}(2)}(x^{\otimes d}), \cdots)\in (X^d)^{\Z}.
\end{equation*}
Let $\mu_\A ^{(\infty)}$ be the measure on $(X^d)^{\Z}$ such that there is some sequence $\{N_i\}_{i\in \N}$ with
\begin{equation*}
  \begin{split}
    \int_{(X^d)^{\Z}} F({\bf x}) d\mu_\A^{(\infty)}({\bf x})&
    = \lim_{i\rightarrow +\infty} \frac{1}{N_i}\sum_{n=0}^{N_i-1} \int_X F( \sigma^n(\w_x^\A)) d\mu(x)\\& =\lim_{i\rightarrow +\infty} \frac{1}{N_i}\sum_{n=0}^{N_i-1} \int_X F\big( (T^{\vec{\phi}(n+j)}x^{\otimes d})_{j\in \Z}\big) d\mu(x)
    \end{split}
\end{equation*}
for all $F\in C((X^d)^{\Z})$, equivalently, for all $l \in \N$ and all $f_j^\otimes = f_j^{(1)}\otimes f_j^{(2)} \otimes \cdots \otimes f_j^{(d)} \in C(X^d), -l \le j\le l$,
\begin{equation*}
\begin{split}
  & \quad \int_{(X^d)^{\Z}} \Big( \bigotimes_{j=-l}^l f_j^\otimes \Big)({\bf x})d\mu_\A^{(\infty)}({\bf x})\\ &=\lim_{i\rightarrow +\infty} \frac{1}{N_i}\sum_{n=0}^{N_i-1}\int_X \prod_{j=-l}^l f_j^\otimes (T^{\vec{\phi}(n+j)}x^{\otimes d}) d\mu(x) \\
  &= \lim_{i\rightarrow +\infty} \frac{1}{N_i}\sum_{n=0}^{N_i-1}\int_X \prod_{j=-l}^l f_j^{(1)}(T^{\phi_1(n+j}x) f_j^{(2)} (T^{\phi_2(n+j)}x) \cdots f_j^{(d)}(T^{\phi_d(n+j)}x) d\mu(x).
\end{split}
\end{equation*}

\begin{rem}
In general, for $\A=\{\phi_1,\phi_2,\ldots, \phi_d\}$ the measure $\mu^{(\infty)}_\A$ depends on the sequence $\{N_i\}_{i\in \N}$.
\end{rem}

\begin{ques}\label{ques1}
For what kind of $\A$, for $\mu$ a.e. $x\in X$, the m.p.s. $((X^d)^\Z, (\X^d)^\Z, {\mu}^{(\infty)}_{\A,x}, \sigma)$ is the product of a $\infty$-step  pro-nilsystem  and a Bernoulli system, where $\displaystyle \mu^{(\infty)}_\A=\int_X\mu^{(\infty)}_{\A,x} d\mu(x)$ is the ergodic decomposition  of $\mu_\A^{(\infty)}$ under $\sigma$.
\end{ques}

\subsection{}

Let ${\mathbb P}=\{\mathfrak{p}_1<\mathfrak{p}_2<\cdots \}$ be the set of primes.

\begin{ques}\label{ques2}
Let $(X,\X,\mu,T)$ be a m.p.s. Then for all non-constant integral polynomials $q_1,\ldots, q_d$ and for all $f_1, \ldots, f_d \in L^\infty(X,\mu)$, do the averages
$$\frac{1}{N}\sum_{n=0}^{N-1} f_1(T^{q_1(\mathfrak{p}_{n})}x)\cdots f_d(T^{q_d(\mathfrak{p}_{n+d-1})}x)$$
converge in $L^2(X,\mu)$? a.s.?

In particular, do the averages
$$\frac{1}{N}\sum_{n=0}^{N-1} f_1(T^{\mathfrak{p}_{n}}x)\cdots f_d(T^{\mathfrak{p}_{n+d-1}}x)$$
converge in $L^2(X,\mu)$? a.s.?
\end{ques}

If Question \ref{ques2} has a positive answer, then for $\A=\{\phi_1(n)=q_1(\mathfrak{p}_n), \ldots, \phi_d(n)=q_d(\mathfrak{p}_n)\}$, one can define the corresponding m.p.s. $((X^d)^\Z, (\X^d)^\Z, {\mu}^{(\infty)}_{\A}, \sigma)$.
It is interesting to know its structure and its applications.

\appendix

\section{From continuous functions to measurable functions}\label{appendix3}

In the paper, we meet the following situation frequently: we first build some convergence results 
for continuous functions, and then pass them to bounded measurable ones. The proof of this procedure is non-trivial but standard. In this appendix, we give an example to explain the way to do this for completeness. Moreover, we also prove some results mentioned in the previous text.

\subsection{From continuous functions to measurable ones}

First we need the following simple but basic fact:

\begin{lem}\label{lem-product-1}
Let $\{a_i\}, \{b_i\}\subseteq \C$. Then
\begin{equation*}\label{}
\prod_{i=1}^k a_i-\prod_{i=1}^k b_i=(a_1-b_1)b_2\ldots b_k+
a_1(a_2-b_2)b_3\ldots b_k +\ldots +a_1\cdots a_{k-1}(a_k-b_k).
\end{equation*}
\end{lem}

By this fact we have
\begin{lem}\label{lem-product}
Let $f_i,f_i'\in L^\infty(X_i,\X_i,\mu_i)$, $i=1,\ldots, k$ and assume that all $|f_i|, |f'_i|$ bounded by $M>0$. Then for any standard measure $\lambda$ of $\{(X_i,\X_i,\mu_i)\}_{i=1}^k$,
\begin{equation*}
\int_{\prod_{i=1}^k X_i} \left|f_1\otimes f_2\otimes \cdots \otimes f_k-f'_1\otimes f_2'\otimes \cdots \otimes f'_k\right|d\lambda \le M^{k-1} \sum_{i=1}^k \int_{X_i} |f_i-f'_i| d\mu_i.
\end{equation*}
\end{lem}

\begin{thm}\label{thm-A1}
Let $T_1,T_2, \cdots, T_d$ be invertible measure preserving transformations which act on a Lebesgue space $(X,\X,\mu)$ ($X$ is a compact metric space), and let $a_1, a_2,\cdots, a_d: \Z\rightarrow \Z$ be maps. If there is a family $\{\lambda_x\}_{x\in X}$ of Borel probability measures on $X^d$ such that
\begin{itemize}
 \item $\displaystyle \lambda=\int_X \lambda_x d \mu(x)$ is a Borel standard probability measure on $X^d$.
  \item for all $g_1,\ldots, g_d\in C(X)$, the limit
\begin{equation*}
 \lim_{N\to\infty} \frac{1}{N}\sum_{n=0}^{N-1} g_1(T_1^{a_1(n)}x)g_2(T_2^{a_2(n)}x)\cdots g_d(T_2^{a_d(n)}x)=\int_{X^d} g_1\otimes g_2\otimes \cdots \otimes g_d d\lambda_x
\end{equation*}
holds in $L^2(X,\mu)$,
\end{itemize}
then for all $f_1,\ldots, f_d\in L^\infty(X,\mu)$,
\begin{enumerate}
 \item for $\mu$-a.e. $x\in X$, $\bigotimes_{j=1}^d f_j\in L^\infty(X^d,\lambda_x)$,
  \item the limit
\begin{equation}\label{AA1}
 \lim_{N\to\infty} \frac{1}{N}\sum_{n=0}^{N-1} f_1(T_1^{a_1(n)}x)f_2(T_2^{a_2(n)}x)\cdots f_d(T_2^{a_d(n)}x)=\int_{X^d} f_1\otimes f_2\otimes \cdots \otimes f_d d\lambda_x
\end{equation}
holds in $L^2(X,\mu)$.
\end{enumerate}
\end{thm}

\begin{proof}
Let $f_{1}, f_{2}, \ldots, f_d\in L^\infty(X,\mu)$ and let $\ep>0$.
Without loss of generality, we assume that for each $1\le j\le d$,
$\|f_j\|_\infty\le 1$. Choose continuous functions $g_j$ such that
$\|g_j\|_\infty\le 1$ and $\|f_j-g_j\|_{L^1(\mu)}\le \|f_j-g_j\|_{L^2(\mu)} < \ep$ for all $1\le j\le
d$.

\medskip

First we show (1). Let $A_j=\{x\in X: |f_j(x)|\le \|f_j\|_\infty\le 1\}$ for $j\in \{1,\ldots, d\}$. Then $\mu(A_j)=1$ for all $j\in \{1,\ldots, d\}$. Note that
$$\prod_{j=1}^dA_j\subseteq \{(x_1,\ldots,x_d)\in X^d: \bigotimes_{j=l}^d f_j(x_1,\ldots,x_d)=\prod_{j=1}^d f_j(x_j)\le 1 \}.$$
Since $\lambda$ is a standard measure on $X^d$, for each $j\in \{1,\ldots,d\}$ $\lambda(P_j^{-1}(A_j))=\mu(A_j)=1$, where $P_j: X^d\rightarrow X$ is the projection to $j$-th coordinate. Thus $\lambda(\prod_{j=1}^dA_j)=\lambda(\bigcap_{j=1}^d P_j^{-1}(A_j))=1 $. In particular, $\big\|\bigotimes_{j=l}^d f_j\big \|_{L^\infty(X^d,\lambda)}\le 1$
 and
$\bigotimes_{j=l}^d f_j\in L^\infty(X^d,\lambda)$.
And by
\begin{equation*}
  \int_X \lambda_x(\prod_{j=1}^dA_j)d\mu(x)=\lambda(\prod_{j=1}^dA_j)=1,
\end{equation*}
we have for $\mu$-a.e. $x\in X$, $\lambda_x(\prod_{j=1}^dA_j)=1$. It follows that for $\mu$-a.e. $x\in X$, $\big\|\bigotimes_{j=l}^d f_j\big \|_{L^\infty(X^d,\lambda_x)}\le 1$ and
$\bigotimes_{j=1}^d f_j\in L^\infty(X^d,\lambda_x)$.

Now we show (2). We have
{\small \begin{equation}\label{yxd1}
\begin{split}
&  \left \| \frac{1}{N}\sum_{n=0}^{N-1} \prod_{j=1}^d f_j(T_j^{a_j(n)}x) -\int_{X^{d}} \bigotimes_{j=1}^d f_jd\lambda_x  \right \|_{L^2(\mu)}\\
& \le \left \|
\frac{1}{N}\sum_{n=0}^{N-1} \prod_{j=1}^d f_j(T_j^{a_j(n)}x)-\frac{1}{N}\sum_{n=0}^{N-1} \prod_{j=1}^d g_j(T_j^{a_j(n)}x) \right \|_{L^2(\mu)} + \\ & \quad \left \| \frac{1}{N}\sum_{n=0}^{N-1}  \prod_{j=1}^d g_j(T_j^{a_j(n)}x)-\int_{X^{d}} \bigotimes_{j=1}^d g_j d\lambda_x   \right \|_{L^2(\mu)}+\left \| \int_{X^{d}} \bigotimes_{j=1}^d g_j d\lambda_x -\int_{X^{d}}  \bigotimes_{j=1}^d f_j d\lambda_x \right \|_{L^2(\mu)}  .
\end{split}
\end{equation}}

Since $\mu$ is $T_j$-invariant, for  $1\le j\le d$ and $N\in \N$, we have
\begin{equation*}\label{yxd4}
\begin{split}
   &\quad  \left\| \frac {1}{N} \sum_{n=0}^{N-1} \Big | g_j(T_j^{a_j(n)}x)-f_j(T_j^{a_j(n)}x)\Big | \right\|_{L^2(\mu)} \le \frac {1}{N} \sum_{n=0}^{N-1}\left\|  g_j(T_j^{a_j(n)}x)-f_j(T_j^{a_j(n)}x) \right\|_{L^2(\mu)} \\ &= \frac {1}{N} \sum_{n=0}^{N-1}\Big( \int_X \Big ( g_j(T_j^{a_j(n)}x)-f_j(T_j^{a_j(n)}x)\Big )^2 d\mu(x)\Big)^{1/2}\\
   &=\frac {1}{N} \sum_{n=0}^{N-1}\Big( \int_X \Big ( g_j(x)-f_j(x)\Big )^2 d\mu(x)\Big)^{1/2}=
\|g_j- f_j\|_{L^2(\mu)}.
\end{split}
\end{equation*}
Hence by Lemma \ref{lem-product-1},
\begin{equation}\label{yxd5}
\begin{split}
&  \quad \limsup_{N\to \infty} \Big \| \frac{1}{N}\sum_{n=0}^{N-1}\prod_{j=1}^d g_j(T_j^{a_j(n)}x) -
\frac{1}{N}\sum_{n=0}^{N-1} \prod_{j=1}^d f_j(T_j^{a_j(n)}x) \Big \|_{L^2(\mu)}\\
& \le \sum_{j=1}^d \Big( \lim_{N\to\infty } \Big\| \frac {1}{N} \sum_{n=0}^{N-1} \Big | g_j(T_j^{a_j(n)}x)-f_j(T_j^{a_j(n)}x)\Big | \Big \|_{L^2(\mu)} \Big ) \\
& = \sum_{j=1}^d \|g_j-f_j\|_{L^2(\mu)}\le d \ep .
\end{split}
\end{equation}

Since $g_{l}, g_{2}, \ldots g_d\in C(X)$, by the assumption of the theorem,  we have
\begin{equation}\label{yxd3}
\lim_{N\to\infty} \Big\| \int_{X^{d}} \bigotimes_{j=1}^d g_j d\lambda_x -
\frac{1}{N}\sum_{n=0}^{N-1}  \prod_{j=1}^d g_j(T_j^{a_j(n)}x)  \Big \|_{L^2(\mu)}=0.
\end{equation}

Let $$\widetilde{\lambda}=\lambda\times_X\lambda=\int_X \lambda_x \times \lambda_x d \mu(x).$$
Sine $\lambda$ is standard on $X^d$, the measure $\widetilde{\lambda}$ is standard on $X^d\times X^d$.
Then
\begin{equation*}
  \begin{split}
   & \Big \| \int_{X^d} \bigotimes_{j=1}^d f_j d\lambda_x-
\int_{X^d} \bigotimes_{j=1}^d g_j d\lambda_x \Big \|_{L^2(\mu)}^2\\
= & \int_X\Big( \int_{X^d} \big(\bigotimes_{j=1}^d f_j - \bigotimes_{j=1}^d g_j\big) d\lambda_x \Big)^2 d\mu(x)\\
=& \int_X\Big( \int_{X^d\times X^d} \big(\bigotimes_{j=1}^d f_j - \bigotimes_{j=1}^d g_j\big)\otimes \big(\bigotimes_{j=1}^d f_j - \bigotimes_{j=1}^d g_j\big) d\lambda_x\times \lambda_x \Big) d\mu(x)\\
=& \int_{X^d\times X^d} \big(\bigotimes_{j=1}^d f_j - \bigotimes_{j=1}^d g_j\big)\otimes \big(\bigotimes_{j=1}^d f_j - \bigotimes_{j=1}^d g_j\big) d\widetilde{\lambda}.
   \end{split}
\end{equation*}

Because $\lambda$ is a standard measure, by Lemma \ref{lem-product} we have
\begin{equation*}
  \begin{split}
   & \int_{X^d\times X^d} \big(\bigotimes_{j=1}^d f_j - \bigotimes_{j=1}^d g_j\big)\otimes \big(\bigotimes_{j=1}^d f_j - \bigotimes_{j=1}^d g_j\big) d\widetilde{\lambda}\\
   \le & \int_{X^d\times X^d} \big|\bigotimes_{j=1}^d f_j - \bigotimes_{j=1}^d g_j\big | \otimes \Big(\big|\bigotimes_{j=1}^d f_j\big|+ \big| \bigotimes_{j=1}^d g_j\big|\Big) d\widetilde{\lambda}\\
   \le & 2 \int_{X^d\times X^d} \big|\bigotimes_{j=1}^d f_j - \bigotimes_{j=1}^d g_j\big | \otimes \bigotimes_{j=1}^d 1 d\widetilde{\lambda}=  2 \int_{X^d} \big|\bigotimes_{j=1}^d f_j - \bigotimes_{j=1}^d g_j\big | d{\lambda}\\
   \le & 2 \sum_{j=1}^d\int_X |f_j-g_j| d\mu\le 2\sum_{j=1}^d \|f_j-g_j\|_{L^2(\mu)}\le 2d\ep.
   \end{split}
\end{equation*}

It follows that
\begin{equation}\label{yxd2}
\left \| \int_{X^d} \bigotimes_{j=1}^d f_j d\lambda_x-
\int_{X^d} \bigotimes_{j=1}^d g_j d\lambda_x \right \|_{L^2(\mu)}
\le \sqrt{2d\ep} .
\end{equation}

So combining (\ref{yxd1})-(\ref{yxd2}), we have
$$ \limsup_{N\to\infty } \left \| \int_{X^{d}} \bigotimes_{j=1}^d f_jd\lambda_x- \frac{1}{N}\sum_{n=0}^{N-1} \prod_{j=1}^d f_j(T_j^{a_j(n)}x)  \right \|_{L^2(\mu)} \le d\ep +  \sqrt{2d\ep}. $$
Since $\ep$ is arbitrary, the proof is complete.
\end{proof}

\begin{rem}
In Theorem \ref{thm-A1}, if we only want to show that  $\displaystyle \lim_{N\to\infty} \frac{1}{N}\sum_{n=0}^{N-1} T_1^{a_1(n)}g_1\cdots T_2^{a_d(n)}g_d$ holding in $L^2(\mu)$ for continuous functions implies that it holds for bounded measure functions, then the proof will be much simpler.
\end{rem}

\begin{de}
Let $(X,\X,\mu, T)$ be m.p.s., and let $a: \Z\rightarrow \Z$ be a map. We say that $T$ {\em satisfies the condition (B) w.r.t. $a$ }, if for all $p>1$ there exists a constant $C(p, a)>0$ such that for all $f\in L^p(X, \mu)$,
\begin{equation*}
\Big \| \sup_{N>0} \big| \frac{1}{N}\sum_{n=0}^{N-1}T^{a(n)}f\big|\Big \|_{L^p(\mu)}\le C(p,a) \|f\|_{L^p(\mu)}. \tag{B}
\end{equation*}
\end{de}

By Theorem \ref{Bouragain}, any m.p.s. $(X,\X,\mu,T)$ satisfies the condition (B) w.r.t. $p$, where $p: \Z\rightarrow \Z$ is a polynomial.

\begin{thm}\label{thm-A2}
Let $T_1,T_2, \cdots, T_d$ be invertible measure preserving transformations which act on a Lebesgue space $(X,\X,\mu)$ ($X$ is a compact metric space), and let $a_1, a_2,\cdots, a_d: \Z\rightarrow \Z$ be maps. Assume that for each $i\in \{1,2,\ldots, d\}$, $T_i$ satisfies the condition (B) w.r.t. $a_i$.
If there is a family $\{\lambda_x\}_{x\in X}$ of Borel probability measures on $X^d$ such that
\begin{itemize}
 \item $\displaystyle \lambda=\int_X \lambda_x \mu(x)$ is a Borel standard probability measures on $X^d$.
  \item for all $g_1,\ldots, g_d\in C(X)$, the limit
\begin{equation*}
 \lim_{N\to\infty} \frac{1}{N}\sum_{n=0}^{N-1} g_1(T_1^{a_1(n)}x)g_2(T_2^{a_2(n)}x)\cdots g_d(T_2^{a_d(n)}x)=\int_{X^d} g_1\otimes g_2\otimes \cdots \otimes g_d d\lambda_x
\end{equation*}
holds for $\mu$-a.e. $x\in X$,
\end{itemize}
then for all $f_1,\ldots, f_d\in L^\infty(X,\mu)$,  the limit
\begin{equation}\label{AA2}
 \lim_{N\to\infty} \frac{1}{N}\sum_{n=0}^{N-1} f_1(T_1^{a_1(n)}x)f_2(T_2^{a_2(n)}x)\cdots f_d(T_2^{a_d(n)}x)=\int_{X^d} f_1\otimes f_2\otimes \cdots \otimes f_d d\lambda_x
\end{equation}
exists for $\mu$-a.e. $x\in X$.
\end{thm}

\begin{proof}
Let $f_{1}, f_{2}, \ldots, f_d\in L^\infty(X,\mu)$. By Theorem \ref{thm-A1}, for $\mu$-a.e. $x\in X$, $\bigotimes_{j=1}^d f_j\in L^\infty(X^d,\lambda_x)$. First
we show that $$\lim_{N\to\infty} \frac{1}{N}\sum_{n=0}^{N-1} f_1(T_1^{a_1(n)}x)f_2(T_2^{a_2(n)}x)\cdots f_d(T_2^{a_d(n)}x)$$ exists  for $\mu$-a.e. $x\in X$. The proof is standard, we follow the proof of \cite[Corollary 2.2]{DL1996}.  After that we show \eqref{AA2} holds.

Let $(f_1,\ldots,f_d)$ be the limit of $\Big\{(f_1^{(k)}, \ldots, f_d^{(k)})\Big\}_{k\in \N}$ in $L^{2d}(X, \mu) \times \cdots \times L^{2d}(X, \mu)$, for all $k\in \N$, $f_1^{(k)}, \ldots, f_d^{(k)}\in C(X)$. Thus by the assumption, we have
\begin{equation*}
 \lim_{N\to\infty} \frac{1}{N}\sum_{n=0}^{N-1} f^{(k)}_1(T_1^{a_1(n)}x)f^{(k)}_2(T_2^{a_2(n)}x)\cdots f^{(k)}_d(T_2^{a_d(n)}x)
\end{equation*}
exists for $\mu$-a.e. $x\in X$.
For a $d$-tuple $(g_1,\ldots,g_d)$, let
\begin{equation*}
\begin{split}
 R(g_1,\ldots, g_d;x)& =\limsup_{N,M\to \infty}\Big|\frac{1}{N}\sum_{n=0}^{N-1} g_1(T_1^{a_1(n)}x)g_2(T_2^{a_2(n)}x)\cdots g_d(T_2^{a_d(n)}x)\\
 & \quad \quad \quad \quad \quad \quad- \frac{1}{M}\sum_{n=0}^{M-1} g_1(T_1^{a_1(n)}x)g_2(T_2^{a_2(n)}x)\cdots g_d(T_2^{a_d(n)}x)\Big|.
\end{split}
\end{equation*}
To prove the result, it suffices to establish that
\begin{equation}\label{a1}
  \|R(f_1,\ldots,f_d; x)\|_{L^2(\mu)}=0.
\end{equation}

\medskip

\noindent {\bf Claim:} There is some constant $C=C(d, a_1,\ldots, a_d)>0$ such that for all $d$-tuple $(g_1,\ldots,g_d)\in L^{2d}(X,\mu)\times \cdots \times L^{2d}(X,\mu)$, one has
\begin{equation*}
  \|R(g_1,\ldots,g_d;\cdot)\|_{L^2(\mu)}\le C\prod_{i=1}^d \|g_i\|_{L^{2d}(\mu)}.
\end{equation*}

\medskip

\noindent {\em Proof of Claim:} By the H\"{o}lder inequality \footnote{Here we use the following H\"{o}lder inequality: Let $p_1,\ldots,p_n>1$ with $\frac {1}{p_1}+\cdots +\frac {1}{ p_n} =1$ and let $\{a_i^{(j)}\}_{i=1}^m \subseteq \R, 1\le j\le n$. Then
$$\sum_{i=1}^m a_i^{(1)}a_i^{(2)}\cdots a_i^{(n)}\le (\sum_{i=1}^m |a_i^{(1)}|^{p_1})^{\frac{1}{p_1}} (\sum_{i=1}^m |a_i^{(2)}|^{p_2})^{\frac{1}{p_2}}\cdots  (\sum_{i=1}^m |a_i^{(n)}|^{p_n})^{\frac{1}{p_n}}.$$}
\begin{equation*}
\begin{split}
 R(g_1,\ldots, g_d;x)\le 2\sup_{N>0} \Big| \frac{1}{N}\sum_{n=0}^{N-1} \prod_{i=1}^d g_i(T_i^{a_i(n)}x)\Big|
\le 2 \prod_{i=1}^d\Big( \sup_{N>0} \big( \frac{1}{N}\sum_{n=0}^{N-1} |g_i|^{d}(T_i^{a_i(n)}x) \big) \Big)^{\frac{1}{d}}.
\end{split}
\end{equation*}
Again by the H\"{o}lder inequality \footnote{Here we use the following H\"{o}lder inequality: Let $p, p_1,\ldots,p_n\ge 1$ with $\frac {1}{p_1}+\cdots +\frac {1}{ p_n} =\frac{1}{p}$ and let $f_1\in L^{p_1}(X,\mu),\ldots , f_n\in L^{p_n}(X,\mu)$. Then $f_1f_2\cdots f_n\in L^p(X,\mu)$ and
$ \| f_1f_2\cdots f_n\|_{L^p(X,\mu)}\le \|f_1\|_{L^{p_1}(X,\mu)} \|f_2\|_{L^{p_2}(X,\mu)} \cdots \|f_n\|_{L^{p_n}(X,\mu)} ,$ that is,
$$\Big(\int_X|f_1f_2\cdots f_n|^pd\mu\Big)^{\frac{1}{p}}\le \Big(\int_X|f_1|^{p_1}\Big)^{\frac{1}{p_1}} \Big(\int_X|f_2|^{p_2}\Big)^{\frac{1}{p_2}}\cdots \Big(\int_X|f_n|^{p_n}\Big)^{\frac{1}{p_n}}$$} and the condition (B),
\begin{equation*}
\begin{split}
\| R(g_1,\ldots, g_d; x)\|_{L^2(\mu)}
 \le &  2\Big\| \prod_{i=1}^d\Big( \sup_{N>0} \big( \frac{1}{N}\sum_{n=0}^{N-1} T_i^{a_i(n)}|g_i|^{d} \big) \Big)^{\frac{1}{d}}\Big\|_{L^2(\mu)}\\
 \le &  2 \prod_{i=1}^d\Big\| \sup_{N>0} \big( \frac{1}{N}\sum_{n=0}^{N-1} T_i^{a_i(n)}|g_i|^{d} \big) \Big\|^{\frac{1}{d}}_{L^{2}(\mu)}\\
 \le & 2 \prod_{i=1}^d C(2,a_i)\Big\| |g_i|^{d} \Big\|^{\frac{1}{d}}_{L^{2}(\mu)}=2 \prod_{i=1}^d C(2,a_i)\big\| g_i \big\|_{L^{2d}(\mu)}.
\end{split}
\end{equation*}
This ends the proof of Claim.
\hfill $\square$

\medskip

Note that for $j\in \{1,\ldots,d\}$,
$$R(g_1,\ldots,g_j+g'_j,\ldots, g_d;x)\le R(g_1,\ldots,g_j,\ldots, g_d;x)+R(g_1,\ldots,g'_j,\ldots, g_d;x).$$
Combining this with the assumption $R(f^{(k)}_1,\ldots,f^{(k)}_d;x)=0$, for all $k\in \N$, one has that for almost all $x\in X$,
\begin{equation*}
\begin{split}
& R(f_1,\ldots,f_{d-1}, f_d;x)
\\ \le &   R(f_1,\ldots, f_{d-1},f_{d}-f^{(k)}_{d};x)+ R(f_1,\ldots,f_{d-1},f^{(k)}_{d};x)\\
\le &  R(f_1,\ldots, f_{d-1},f_{d}-f^{(k)}_{d};x)+ R(f_1,\ldots, f_{d-1}-f^{(k)}_{d-1},f_d^{(k)};x)+R(f_1,\ldots,f^{(k)}_{d-1},f^{(k)}_{d};x)\\
\le & \ldots \\
\le & \sum_{i=1}^d R(f_1,\ldots, f_{i-1},f_i-f_i^{(k)},f_{i+1}^{(k)},\ldots,f^{(k)}_d;x)+R(f^{(k)}_1,f^{(k)}_{2},\ldots,f^{(k)}_{d};x)\\
\le & \sum_{i=1}^d R(f_1,\ldots, f_{i-1},f_i-f_i^{(k)},f_{i+1}^{(k)},\ldots,f^{(k)}_d;x).
\end{split}
\end{equation*}
By the Claim, we have
\begin{equation*}
\begin{split}
& \| R(f_1,\ldots,f_{d-1}, f_d;x)\|_{L^2(\mu)}
\\ \le & \sum_{i=1}^d \| R(f_1,\ldots, f_{i-1},f_i-f_i^{(k)},f_{i+1}^{(k)},\ldots,f^{(k)}_d;x)\|_{L^2(\mu)} \\
\le &  \sum_{i=1}^d  C \| f_1 \|_{L^{2d}(\mu)} \cdots \| f_{i-1} \|_{L^{2d}(\mu)} \| f_i- f_i^{(k)} \|_{L^{2d}(\mu)} \| f_{i+1}^{(k)} \|_{L^{2d}(\mu)} \cdots \| f^{(k)}_d \|_{L^{2d}(\mu)}\\
\rightarrow & 0 , \quad k\to\infty .
\end{split}
\end{equation*}
Then we have $(\ref{a1})$, and hence
$$\lim_{N\to\infty} \frac{1}{N}\sum_{n=0}^{N-1} f_1(T_1^{a_1(n)}x)f_2(T_2^{a_2(n)}x)\cdots f_d(T_2^{a_d(n)}x)$$ holds  for $\mu$-a.e. $x\in X$.
By Theorem \ref{thm-A1},
\begin{equation*}
 \lim_{N\to\infty} \Big\| \frac{1}{N}\sum_{n=0}^{N-1} f_1(T_1^{a_1(n)}x)f_2(T_2^{a_2(n)}x)\cdots f_d(T_2^{a_d(n)}x)-\int_{X^d} f_1\otimes f_2\otimes \cdots \otimes f_d d\lambda_x\|_{L^2(\mu)}=0.
\end{equation*}
Thus the limit
\begin{equation*}
 \lim_{N\to\infty} \frac{1}{N}\sum_{n=0}^{N-1} f_1(T_1^{a_1(n)}x)f_2(T_2^{a_2(n)}x)\cdots f_d(T_2^{a_d(n)}x)=\int_{X^d} f_1\otimes f_2\otimes \cdots \otimes f_d d\lambda_x
\end{equation*}
exists for $\mu$-a.e. $x\in X$. The proof is complete.
\end{proof}

\subsection{Proof of Proposition \ref{prop-mesurable}}

In Theorem \ref{thm-A1} and Theorem \ref{thm-A2}, we deal with the standard measure on the product spaces of finitely many systems. But in the paper, we mostly deal with the standard measure on the product spaces of infinitely many systems. The proofs of related results are similar. We take an example in this subsection.

The proof of Proposition \ref{prop-mesurable} is similar but simpler to the proof of Theorem \ref{thm-A1}.

\begin{proof}[Proof of Proposition \ref{prop-mesurable}]
We only show a special case when $\A=\{p\}$, and the general case is similar. In this case $d=1$. So by the definition, for all $l\in \N$ and all $f_{-l},\ldots, f_l\in C(X)$,
\begin{equation*}
  \begin{split}
  & \quad \int_{X^{\Z}} \Big( \bigotimes_{j=-l}^l f_j\Big)({\bf x})d\mu_\A^{(\infty)}({\bf x})= \int_{X^{\Z}} \big(f_{-l}\otimes f_{-l+1} \otimes \cdots \otimes f_l\big)({\bf x}) d\mu_\A^{(\infty)}({\bf x})\\ &=\lim_{N\to \infty} \frac{1}{N}\sum_{n=0}^{N-1}\int_X \prod_{j=-l}^l f_j(T^{p(n+j)}x) d\mu(x),
\end{split}
\end{equation*}
where ${\bf x}=(x_n)_{n\in \Z}\in X^\Z$. We show that this equality still holds if replace functions in $C(X)$ by functions in $L^\infty(X,\mu)$.

\medskip

Fix $l\in \N$ and $f_{-l}, f_{-l+1}, \ldots, f_l\in L^\infty(X,\mu)$ and let $\ep>0$. Without loss of generality, we assume that for all $-l\le j\le l$,
$\|f_j\|_\infty\le 1$. Choose continuous functions $g_j$ such that
$\|g_j\|_\infty\le 1$ and $\|f_j-g_j\|_{L^1(X,\mu)}<\frac{\ep}{2l+1}$ for all $-l\le j\le
l$. Since $\mu^{(\infty)}_\A$ is a standard measure, it is easy to see that $\bigotimes_{j=-l}^l f_j\in L^\infty(X^\Z,\mu^{(\infty)}_\A)$.
We have
\begin{equation}\label{yexd1}
\begin{split}
& \quad  \left | \int_{X^{\Z}} \Big( \bigotimes_{j=-l}^l f_j\Big)({\bf x})d\mu_\A^{(\infty)}({\bf x})- \frac{1}{N}\sum_{n=0}^{N-1}\int_X \prod_{j=-l}^l f_j(T^{p(n+j)}x) d\mu(x) \right |\\
& \le  \left | \int_{X^{\Z}} \Big( \bigotimes_{j=-l}^l f_j\Big)({\bf x})d\mu_\A^{(\infty)}({\bf x}) -\int_{X^{\Z}} \Big( \bigotimes_{j=-l}^l g_j\Big)({\bf x})d\mu_\A^{(\infty)}({\bf x}) \right |\\
&+\left | \int_{X^{\Z}} \Big( \bigotimes_{j=-l}^l g_j\Big)({\bf x})d\mu_\A^{(\infty)}({\bf x}) -
\frac{1}{N}\sum_{n=0}^{N-1}\int_X \prod_{j=-l}^lg_j(T^{p(n+j)}x) d\mu(x)  \right | \\ & +\left | \frac{1}{N}\sum_{n=0}^{N-1}\int_X \prod_{j=-l}^lg_j(T^{p(n+j)}x) d\mu(x) -
\frac{1}{N}\sum_{n=0}^{N-1}\int_X \prod_{j=-l}^lf_j(T^{p(n+j)}x) d\mu(x) \right |.
\end{split}
\end{equation}

Because $\mu^{(\infty)}_\A$ is a standard measure, by Lemma \ref{lem-product} we have
\begin{equation}\label{yexd4}
\left | \int_{X^\Z} \bigotimes_{j=-l}^l f_j d\mu^{(\infty)}_\A -
\int_{X^\Z} \bigotimes_{j=-l}^l g_j d\mu^{(\infty)}_\A \right |
\le \sum_{j=-l}^l \int_{X} |f_j-g_j| d\mu \le \ep.
\end{equation}
Since $g_{-l}, g_{-l+1}, \ldots g_l\in C(X)$, by the definition of $\mu^{(\infty)}_\A$, we have
\begin{equation}\label{xdye3}
\lim_{N\to\infty} \left | \int_{X^{\Z}} \Big( \bigotimes_{j=-l}^l g_j\Big)({\bf x})d\mu_\A^{(\infty)}({\bf x}) -
\frac{1}{N}\sum_{n=0}^{N-1}\int_X \prod_{j=-l}^lg_j(T^{p(n+j)}x) d\mu(x)  \right |=0.
\end{equation}

Since $\mu$ is $T$-invariant, for  $-l\le j\le l$ and $N\in \N$, we have
\begin{equation*}\label{need}
\begin{split}
   &\quad  \int_X \frac {1}{N} \sum_{n=0}^{N-1} \Big | g_j(T^{p(n+j)}x)-f_j(T^{p(n+j)}x)\Big | d\mu(x)\\ &=\frac {1}{N} \sum_{n=0}^{N-1} \int_X \Big | g_j(T^{p(n+j)}x)-f_j(T^{p(n+j)}x)\Big | d\mu(x)=
\|g_j- f_j\|_{L^1(X,\mu)}.
\end{split}
\end{equation*}
Hence by Lemma \ref{lem-product-1},
\begin{equation}\label{yexd2}
\begin{split}
&  \limsup_{N\to \infty} \left | \frac{1}{N}\sum_{n=0}^{N-1}\int_X \prod_{j=-l}^lg_j(T^{p(n+j)}x) d\mu(x) -
\frac{1}{N}\sum_{n=0}^{N-1}\int_X \prod_{j=-l}^l f_j(T^{p(n+j)}x) d\mu(x) \right |\\
& \le \sum_{j=-l}^l \Big[ \lim_{N\to\infty } \int_X \frac {1}{N} \sum_{n=0}^{N-1} \Big | g_j(T^{p(n+j)}x)-f_j(T^{p(n+j)}x)\Big | d\mu (x)\Big ]\\
& = \sum_{j=-l}^l \|f_j-g_j\|_{L^1(X,\mu)}\le \ep .
\end{split}
\end{equation}

So combining (\ref{yexd1})-(\ref{yexd2}), we have
$$ \limsup_{N\to\infty } \left | \int_{X^{\Z}} \Big( \bigotimes_{j=-l}^l f_j\Big)({\bf x})d\mu_\A^{(\infty)}({\bf x})- \frac{1}{N}\sum_{n=0}^{N-1}\int_X \prod_{j=-l}^lf_j(T^{p(n+j)}x) d\mu(x) \right | \le 2\ep . $$
Since $\ep$ is arbitrary, the proof is complete.
\end{proof}

\section{A weighted ergodic average and proof of Theorem \ref{plus-nil}}\label{appendix1}

First we need some notions and results in \cite{HK09}. In this section, by an {\em interval}, we mean an interval in $\Z$. If $I$ is an interval, $|I|$ denotes its length. Throughout this section, ${\bf I}$ denotes a sequence $(I_N)_{N\in \N}$ of intervals, and we implicitly assume that the lengths of the intervals tend to $\infty$ as $N\to\infty$.

Let $k\in \N$. Let $a=(a_n)_{n\in \Z}$ be a bounded sequence of real numbers
and ${\bf I}=(I_N)_{N\in \N}$ be a sequence of intervals whose lengths $|I_N|$ tend to infinity. We say that this sequence of intervals is {\em $k$-adapted to the sequence $a=(a_n)_{n\in \Z}$}, if for every $\underline{h}=(h_1,\ldots, h_k)\in \Z^k$,
the limit
\begin{equation}\label{}
c_{\underline{h}}({\bf I}, a)\triangleq \lim_{N\to\infty} \frac{1}{|I_N|} \sum_{n\in I_N}\prod_{\ep\in \{0,1\}^k}a_{n+h_1\ep_1+\cdots +h_k\ep_k}=\lim_{N\to\infty} \frac{1}{|I_N|} \sum_{n\in I_N}\prod_{\ep\in \{0,1\}^k}a_{n+\ep\cdot \underline{h}}
\end{equation}
exists. Note that every sequence of intervals whose lengths tend to infinity admits a subsequence
which is adapted to the sequence $a=(a_n)_{n\in \N}$.

\begin{de}
Suppose that ${\bf I}=(I_N)_{N\in \N}$ is $k$-adapted to $a=(a_n)_{n\in \Z}$. We define
\begin{equation}\label{app1}
  \HK a\HK_{{\bf I}, k}\triangleq \left(\lim_{H\to\infty}\frac{1}{H^k}\sum_{\underline{h}\in [0,H-1]^k}c_{\underline{h}}({\bf I}, a)\right)^{1/2^k}.
\end{equation}
\end{de}

\begin{rem}\label{rem-A1}
Suppose that ${\bf I}=(I_N)_{N\in \N}$ is $k$-adapted to $a=(a_n)_{n\in \Z}$. Using a variant of Furstenberg's correspondence principle, one can show that there exist a t.d.s. $(X,T)$, a transitive point $w$, a continuous function $f$ on $X$, and a $T$-invariant Borel probability measure $\mu$ on $X$ such that $a_n=f(T^n w)$ for all $n\in \Z$ and for $\underline{h}=(h_1,\ldots, h_k)\in \Z^k$,
\begin{equation}\label{}
c_{\underline{h}}({\bf I}, a)=\int_X\prod_{\ep\in \{0,1\}^k}f(T^{\ep\cdot \underline{h}}x) d\mu(x).
\end{equation}
See Subsection 4.2 of \cite{HK09} or \cite[Proposition 22.1]{HK18} for details about this fact. Thus
\begin{equation*}
  \begin{split}
    \HK a\HK_{{\bf I}, k}&= \left(\lim_{H\to\infty}\frac{1}{H^k}\sum_{\underline{h}\in [0,H-1]^k}c_{\underline{h}}({\bf I}, a)\right)^{1/2^k}\\
    &= \left(\lim_{H\to\infty}\frac{1}{H^k}\sum_{\underline{h}\in [0,H-1]^k}\int_X\prod_{\ep\in \{0,1\}^k}f(T^{\ep\cdot \underline{h}}x) d\mu(x)\right)^{1/2^k}\\
    &= \HK f \HK_k.
   \end{split}
\end{equation*}
It follows that $\HK a\HK_{{\bf I}, k}$ is well-defined. We need an equivalent definition of $\HK a\HK_{{\bf I}, k}$. By Appendix A of \cite{Bergelson06} or \cite[Theorem 8.26]{HK18}, we have
\begin{equation*}
  \HK f\HK _k=\left( \lim_{H_{1}\to\infty} \frac{1}{H_{1}}\sum_{h_{1}=1}^{H_{1}}\cdots \lim_{H_k\to\infty} \frac{1}{H_k} \sum_{h_k=1}^{H_k} \int_X\prod_{\ep\in \{0,1\}^k}f(T^{h_1\ep_1+\cdots +h_k\ep_k}x) d\mu(x) \right)^{1/2^k}.
\end{equation*}
It follows that
\begin{equation}\label{app00}
  \HK a\HK_{{\bf I}, k}= \left(\lim_{H_{1}\to\infty} \frac{1}{H_{1}}\sum_{h_{1}=1}^{H_{1}}\cdots \lim_{H_k\to\infty} \frac{1}{H_k}\sum_{h_k=1}^{H_k}c_{\underline{h}}({\bf I}, a)\right)^{1/2^k}.
\end{equation}
Note that by symmetry, the order of $H_1,\ldots, H_k$ is not important in the left side of the equation \eqref{app00}.
We will use this equation later.
\end{rem}


\begin{thm}\cite[Corollary 2.14]{HK09}\label{thm-HK09}
Let $a=(a_n)_{n\in \N}$ be a bounded sequence of real numbers,
and ${\bf I}=(I_N)_{N\in \N}$ be a sequence of intervals that is $k$-adapted to this sequence for some $k\ge 2$. Suppose that $\HK a\HK_{{\bf I}, k}=0$.
Then for every bounded $(k-1)$-step nilsequence $(u_n)_{n\in \Z}$ we have
\begin{equation*}
  \lim_{N\to\infty} \frac{1}{|I_N|} \sum_{n\in I_n} a_nu_n=0.
\end{equation*}
\end{thm}

\medskip

We will prove a result more general than Theorem \ref{plus-nil}.

\begin{thm}\label{thm-apendix}
Let $(X,\X,\mu, T)$ be a m.p.s. and $k\ge 2$ be an integer. Let $(\phi_n(x))_{n\in \Z}$ be a uniformly bounded sequence of $\mu$-measurable functions such that, for $\mu$-almost every $x\in X$, the sequence $(\phi_n(x))_{n\in \Z}$ is a $(k-1)$-step nilsequence. Then for all  integral polynomials  $p_1,\ldots,p_d$ and for all $f_1,\ldots, f_d\in L^\infty(X,\mu)$, the averages
\begin{equation}\label{aaa}
  \frac{1}{N-M}\sum_{n=M}^{N-1} \phi_n(x)f_1(T^{p_1(n)}x)f_2(T^{p_2(n)}x)\cdots f_d(T^{p_d (n)}x)
\end{equation}
converge in $L^2(X,\mu)$ as $N-M \to\infty$.

Moreover, if the integral polynomials $p_1,\ldots, p_d$ are  essentially distinct and $\E(f_j|\ZZ_\infty(T))=0$ for some $j\in \{1,2 ,\ldots, d\}$, then the averages \eqref{aaa} converge to $0$  in $L^2(X,\mu)$.
\end{thm}

\begin{proof} Without loss of generality, assume that $f_1,\ldots,f_d\in L^\infty(X,\mu)$ with $\|f_j\|_\infty\le 1$ for $j=1,\ldots,d$. And we assume that integral polynomials $p_1,\ldots, p_d$ are  essentially distinct and non-constant with $p_1(0)=\cdots =p_d(0)=0$.

\medskip
\noindent {\bf Step 1.} {\em For all non-zero distinct $c_1,\ldots,c_d\in \Z$ and for all $f_1,\ldots, f_d\in L^\infty(X,\mu)$, if $\HK f_j\HK_{k+d}=0$ for some $j\in \{1,\ldots, d\}$, then the averages
\begin{equation}\label{app2}
  \frac{1}{N-M}\sum_{n=M}^{N-1} \phi_n(x)f_1(T^{c_1n}x)f_2(T^{c_2n}x)\cdots f_d(T^{c_d n}x)
\end{equation}
converge to 0 in $L^2(X,\mu)$ as $N-M \to\infty$.}
\medskip


In order to  prove that \eqref{app2} converges to $0$, it suffices to show that every sequence of intervals ${\bf I}=(I_N)_{N\in \N}$ 
whose lengths $|I_N|$ tend to infinity, admits a subsequence ${\bf I}'=(I_N')_{N\in \N}$ such that
\begin{equation}\label{app3}
  \frac{1}{|I_N'|}\sum_{n\in I_N'} \phi_n(x)f_1(T^{c_1n}x) f_2(T^{c_2 n}x)\cdots f_d(T^{c_d n}x)\longrightarrow 0,\ N\to\infty, \ \text{in $L^2(X,\mu)$} .
\end{equation}

For $x\in X$, we write $a(x)=(a_n(x))_{n\in \N}$ for the sequence defined by
\begin{equation*}
  a_n(x)= f_1(T^{c_1n}x)f_2(T^{c_2n}x)\cdots f_d(T^{c_d n}x).
\end{equation*}
For every $\underline{h}=(h_1,\ldots, h_k)\in \Z^k$ , we study the averages
$$\frac{1}{|I_N|}\sum_{n\in I_N} \prod_{\ep\in \{0,1\}^k}a_{n+h_1\ep_1+\cdots +h_k\ep_k}(x).$$
Note that
\begin{equation}\label{app4}
\begin{split}
  a_{n+h_1\ep_1+\cdots +h_k\ep_k}(x)&=\Big(\prod_{\ep\in \{0,1\}^k}T^{c_1(h_1\ep_1+\cdots+h_k\ep_k)}f_1\Big)(T^{c_1n}x)\ldots \\
  &\quad \quad \quad \quad \quad\Big(\prod_{\ep\in \{0,1\}^k}T^{c_d(h_1\ep_1+\cdots+h_k\ep_k)}f_d\Big)(T^{c_dn}x).
  \end{split}
\end{equation}
Thus by Lemma \ref{lem-AP-HK}\footnote{In Lemma \ref{lem-AP-HK}, $I_N=[0, N-1]$. It holds for any ${\bf I}=(I_N)_{N\in \N}$,  a sequence of intervals whose lengths $|I_N|$ tend to infinity \cite[Proposition 21.7]{HK18}.}
\begin{equation}\label{app5}
  \begin{split}
 &\limsup_{N\to\infty}\Big \|\frac{1}{|I_N|}\sum_{n\in I_N} \prod_{\ep\in \{0,1\}^k}a_{n+h_1\ep_1+\cdots +h_k\ep_k}(x)\Big\|_{L^2(X,\mu)}\\
 &\quad \quad \le C\min_{1\le j\le d}\interleave \Big(\prod_{\ep\in \{0,1\}^k}T^{c_j(h_1\ep_1+\cdots+h_k\ep_k)}f_j\Big)\interleave _d,
 \end{split}
\end{equation}
where $C$ is a positive constant only depending on $c_1,\ldots, c_d$.
By \cite[Theorem 1.1]{HK05} the averages
$$\frac{1}{|I_N|}\sum_{n\in I_N} \prod_{\ep\in \{0,1\}^k}a_{n+h_1\ep_1+\cdots +h_k\ep_k}(x)$$
converge in $L^2(X,\mu)$. As a consequence, a subsequence of this sequence of averages converges
$\mu$-almost everywhere. This subsequence depends on the parameter $\underline{h}$, but since there are only countably many such parameters, by a diagonal argument we can find a subsequence ${\bf I}'=(I_N')_{N\in \N}$ such that for $\mu$-almost every $x\in X$ the limit
\begin{equation}\label{}
  c_{\underline{h}}({\bf I}', a(x))=\lim_{N\to\infty}\frac{1}{|I_N'|} \sum_{n\in I_N} \prod_{\ep\in \{0,1\}^k}a_{n+h_1\ep_1+\cdots +h_k\ep_k}(x)
\end{equation}
exists for every choice of $\underline{h}=(h_1,\ldots, h_k)\in \Z^k$. This means that, for $\mu$-almost every $x\in X$, the sequence of intervals ${\bf I}'$ is $k$-adapted to the sequence $a(x)=(a_n(x))_{n\in \N}$.
By \eqref{app5},
\begin{equation}\label{app6}
 \big \| c_{\underline{h}}({\bf I}', a(x)) \big\|_{L^2(X,\mu)}\le C\min_{1\le j\le d}\interleave \Big(\prod_{\ep\in \{0,1\}^k}T^{c_j(h_1\ep_1+\cdots+h_k\ep_k)}f_j\Big)\interleave _d.
\end{equation}
Let $\Phi_H=[1,H_1]\times \cdots \times [1,H_k]$, $H_1,\ldots, H_k\in \N$. Assume that $\HK f_j\HK_{k+d}=0$ for some $j\in \{1,2,\ldots, d\}$. Then by \eqref{app6}
\begin{equation}\label{app7}
\begin{split}
  &\quad \big \| \frac{1}{|\Phi_H|}\sum_{\underline{h}\in \Phi_H}c_{\underline{h}}({\bf I}', a(x)) \big\|_{L^2(X,\mu)}\le \frac{1}{|\Phi_H|}\sum_{\underline{h}\in \Phi_H}  \big \| c_{\underline{h}}({\bf I}', a(x)) \big\|_{L^2(X,\mu)}\\ & \le \frac{1}{|\Phi_H|}\sum_{\underline{h}\in \Phi_H}C \interleave \Big(\prod_{\ep\in \{0,1\}^k}T^{c_j(h_1\ep_1+\cdots+h_k\ep_k)}f_j\Big)\interleave _d\\
  &= C\cdot \frac{1}{H_1}\sum_{h_1=1}^{H_1}\ldots \frac{1}{H_k}\sum_{h_k=1}^{H_k} \HK\Big(\prod_{\ep\in \{0,1\}^k}T^{c_j(h_1\ep_1+\cdots+h_k\ep_k)}f_j\Big)\HK_d.
\end{split}
\end{equation}
Now we have
{\small \begin{equation*}
\begin{split}
 &\limsup_{H_k\to\infty} \frac{1}{H_k}\sum_{h_k=1}^{H_k} \HK\Big(\prod_{\ep\in \{0,1\}^{k}}T^{c_j(h_1\ep_1+\cdots+h_k\ep_k)}f_j\Big)\HK_d\\
\le &  \limsup_{H_k\to\infty} \Big(\frac{1}{H_k}\sum_{h_k=1}^{H_k} \HK\Big(\prod_{\ep\in \{0,1\}^{k}}T^{c_j(h_1\ep_1+\cdots+h_k\ep_k)}f_j\Big)\HK_d^{2^d}\Big)^{1/2^d}\\
 = & \limsup_{H_k\to\infty} \Big(\frac{1}{H_k}\sum_{h_k=1}^{H_k} \HK(\prod_{\ep\in \{0,1\}^{k-1}}T^{c_j(h_1\ep_1+\cdots+h_{k-1}\ep_{k-1})}f_j)T^{c_jh_k}\cdot\\
  &\quad \quad \quad \quad \quad \quad \quad (\prod_{\ep\in \{0,1\}^{k-1}}T^{c_j(h_1\ep_1+\cdots+h_{k-1}\ep_{k-1})}f_j)\HK_d^{2^d}\Big)^{1/2^d} \\
 \le & \limsup_{H_k\to\infty}|c_j|^{\frac {1}{2^d}} \Big(\frac{1}{|c_j|H_k}\sum_{l=1}^{|c_j|H_k} \HK(\prod_{\ep\in \{0,1\}^{k-1}}T^{c_j(h_1\ep_1+\cdots+h_{k-1}\ep_{k-1})}f_j)T^{l}\cdot\\
 &\quad \quad \quad \quad \quad \quad \quad(\prod_{\ep\in \{0,1\}^{k-1}}T^{c_j(h_1\ep_1+\cdots+h_{k-1}\ep_{k-1})}f_j)\HK_d^{2^d}\Big)^{\frac{1}{2^d}}\\
 = & |c_j|^{1/2^d} \HK\Big(\prod_{\ep\in \{0,1\}^{k-1}}T^{c_j(h_1\ep_1+\cdots+h_{k-1}\ep_{k-1})}f_j\Big)\HK_{d+1}^2,
\end{split}
\end{equation*}}
where the inequality holds by using the inequality $(\frac{1}{N}\sum_{n=1}^N B_n^{k_1})^{1/k_1}\le (\frac{1}{N}\sum_{n=1}^N B_n^{k_2})^{1/k_2}$, for all $B_1,\ldots, B_N>0$ and $k_1\le k_2\in \N$, the last equality is by Lemma \ref{lemmaE1}.

Similarly, we have
\begin{equation*}
  \begin{split}
 &\limsup_{H_{k-1}\to\infty} \frac{1}{H_{k-1}}\sum_{h_{k-1}=1}^{H_{k-1}} \limsup_{H_k\to\infty} \frac{1}{H_k}\sum_{h_k=1}^{H_k} \HK\Big(\prod_{\ep\in \{0,1\}^{k}}T^{c_j(h_1\ep_1+\cdots+h_k\ep_k)}f_j\Big)\HK_d\\
 &\quad \le |c_j|^{1/2^d} \limsup_{H_{k-1}\to\infty} \frac{1}{H_{k-1}}\sum_{h_{k-1}=1}^{H_{k-1}} \HK\Big(\prod_{\ep\in \{0,1\}^{k-1}}T^{c_j(h_1\ep_1+\cdots+h_{k-1}\ep_{k-1})}f_j\Big)\HK_{d+1}^2\\
 &\quad \le |c_j|^{1/2^d} \limsup_{H_{k-1}\to\infty} \left(\Big(\frac{1}{H_{k-1}}\sum_{h_{k-1}=1}^{H_{k-1}} \HK\Big(\prod_{\ep\in \{0,1\}^{k-1}}T^{c_j(h_1\ep_1+\cdots + h_{k-1}\ep_{k-1})}f_j\Big)\HK_{d+1}^{2^{d+1}}\Big)^{1/2^{d+1}}\right)^2\\
 &\quad \le |c_j|^{1/2^{d}} \Big (|c_j|^{1/2^{d+1}}\HK\Big(\prod_{\ep\in \{0,1\}^{k-2}}T^{c_j(h_1\ep_1+\cdots+h_{k-2}\ep_{k-2})}f_j\Big)\HK_{d+2}^{2}\Big)^2\\
 &\quad =|c_j|^{1/2^{d-1}}  \HK\Big(\prod_{\ep\in \{0,1\}^{k-2}}T^{c_j(h_1\ep_1+\cdots+h_{k-2}\ep_{k-2})}f_j\Big)\HK_{d+2}^{2^2}.
\end{split}
\end{equation*}
Repeating this progress several times, we have
\begin{equation*}
  \limsup_{H_{1}\to\infty} \frac{1}{H_{1}}\sum_{h_{1}=1}^{H_{1}}\cdots \limsup_{H_k\to\infty} \frac{1}{H_k}\sum_{h_k=1}^{H_k} \HK\Big(\prod_{\ep\in \{0,1\}^{k}}T^{a_j(h_1\ep_1+\cdots+h_k\ep_k)}f_j\Big)\HK_d
 \le |c_j|^{\frac{1}{2^{d-k+1}}} \HK f_j\HK_{d+k}^{2^k}.
\end{equation*}
Combining with \eqref{app7}, we have
\begin{equation}\label{app8}
\begin{split}
  &\quad \limsup_{H_1\to\infty}\cdots \limsup_{H_k\to\infty} \big \| \frac{1}{|\Phi_H|}\sum_{\underline{h}\in \Phi_H}c_{\underline{h}}({\bf I}', a(x)) \big\|_{L^2(X,\mu)}\\
  &\le \limsup_{H_{1}\to\infty} \frac{1}{H_{1}}\sum_{h_{1}=1}^{H_{1}}\cdots \limsup_{H_k\to\infty} \frac{1}{H_k}\sum_{h_k=1}^{H_k} C \HK\Big(\prod_{\ep\in \{0,1\}^{k}}T^{c_j(h_1\ep_1+\cdots+h_k\ep_k)}f_j\Big)\HK_d\\ &\le C|c_j|^{\frac{1}{2^{d-k+1}}} \cdot \HK f_j\HK_{d+k}^{2^k}.
\end{split}
\end{equation}
By Remark \ref{rem-A1},
$$\HK a\HK_{{\bf I}', k}= \left(\lim_{H_{k}\to\infty} \frac{1}{H_{k}}\sum_{h_{k}=1}^{H_{k}}\cdots \lim_{H_1\to\infty} \frac{1}{H_1}\sum_{h_1=1}^{H_1}c_{\underline{h}}({\bf I}', a)\right)^{1/2^k}.$$
Thus by H\"{o}lder inequality and by \eqref{app8}, we get
\begin{equation*}\label{}
\begin{split}
  &\quad \left\| \HK a\HK_{{\bf I}', k}\right\|^{2^k}_{L^2(X,\mu)}= \left(\int_X \HK a\HK_{{\bf I}', k}^2 d\mu \right)^{2^{k-1}}\le \int_X \HK a\HK_{{\bf I}', k}^{2^{k} }d\mu\\
  &= \int_X \left(\lim_{H_{k}\to\infty} \frac{1}{H_{k}}\sum_{h_{k}=1}^{H_{k}}\cdots \lim_{H_1\to\infty} \frac{1}{H_1}\sum_{h_1=1}^{H_1}c_{\underline{h}}({\bf I}', a)\right) d\mu \\
  &= \lim_{H_1\to\infty}\ldots \lim_{H_k\to\infty} \int_X \left(\frac{1}{H_{1}}\sum_{h_{1}=1}^{H_{1}}\cdots \frac{1}{H_k}\sum_{h_k=1}^{H_k}c_{\underline{h}}({\bf I}', a)\right) d\mu
  \\
  &\le
  \limsup_{H_1\to\infty}\cdots \limsup_{H_k\to\infty} \big \| \frac{1}{|\Phi_H|}\sum_{\underline{h}\in \Phi_H}c_{\underline{h}}({\bf I}', a(x)) \big\|_{L^2(X,\mu)}\\ &\overset{\eqref{app8}}\le C|c_j|^{\frac{1}{2^{d-k+1}}} \cdot \HK f_j\HK_{d+k}^{2^k}.
\end{split}
\end{equation*}
Hence we have
\begin{equation}\label{app9}
  \Big\|\HK a(x)\HK_{{\bf I}', k} \Big\|_{L^2(X,\mu)}\le C^{\frac{1}{2^k}} |c_j|^{\frac{1}{2^{d+1}}}\cdot \HK f_j\HK_{d+k}=0.
\end{equation}
Thus for $\mu$-a.e. $x\in X$, $\HK a(x)\HK_{{\bf I}', k}=0$. By Theorem \ref{thm-HK09},
we have
$$\frac{1}{|I_N'|}\sum_{n\in I_N'}\phi_n(x)a_n(x)\to 0, N\to\infty, \ \mu\text{-almost everywhere},$$
and \eqref{app3} is proved. This completes the proof of Step 1.

\medskip
\noindent {\bf Step 2.} {\em For all all non-zero distinct $a_1,\ldots,a_d\in \Z$ and for all $f_1,\ldots, f_d\in L^\infty(X,\mu)$, the averages
$$\frac{1}{N-M}\sum_{n=M}^{N-1} \phi_n(x)f_1(T^{a_1n}x)f_2(T^{a_2n}x)\cdots f_d(T^{a_d n}x)$$
converge in $L^2(X,\mu)$ as $N-M \to\infty$.}
\medskip

By Step 1, for all non-zero distinct $a_1,\ldots,a_d\in \Z$ and for all $f_1,\ldots, f_d\in L^\infty(X,\mu)$, if $\HK f_j\HK_{k+d}=0$ for some $j\in \{1,\ldots, d\}$, then the averages
\begin{equation*}
  \frac{1}{N-M}\sum_{n=M}^{N-1} \phi_n(x)f_1(T^{a_1n}x)f_2(T^{a_2n}x)\cdots f_d(T^{a_d n}x)
\end{equation*}
converge to 0 in $L^2(X,\mu)$.
Thus we can restrict to the case where $f_j, 1\le j\le d$, is measurable with respect to $\ZZ_m$ for some $m\in \N$.

By Remark \ref{rem-CFH}-(2), for every $\ep>0$ there exists $\widetilde{f}_j\in L^\infty(X,\mu)$ such that $\widetilde{f}_j$ is $\ZZ_m$-measurable and $\|f_j-\widetilde{f}_j\|_{L^1(X,\mu)}<\ep, 1\le j\le d$, and for $\mu$-almost all $x\in X$ the sequence
$(\widetilde{f}_j(T^nx))_{n\in \Z}$ is an $m$-step nilsequence.
Thus $(\Psi_n(x))_{n\in \Z}=\Big(\prod_{j=1}^d\widetilde{f}_j(T^{a_jn} x)\Big)_{n\in \Z}$ is a nilsequence.
Therefore, it suffices to show the existence in $L^2(X,\mu)$ of the limit
\begin{equation*}\label{}
  \lim_{N\to\infty} \frac{1}{N}\sum_{n=0}^{N-1}\phi_n(x)\Psi_n(x).
\end{equation*}
This follows from Theorem \ref{thm-ParryLeibman} since both are nilsequences. The proof of Step 2 is complete.

\medskip
\noindent {\bf Step 3.} {\em For all non-zero distinct $a_1,\ldots,a_d\in \Z$ and essentially distinct non-linear integral polynomials $p_1,\ldots, p_s$ for all $f_1,\ldots, f_d, g_1,\ldots, g_s\in L^\infty(X,\mu)$, there is some $m\in \N$ such that if $\HK g_j\HK_{m}=0$ for some $j\in \{1,\ldots, s\}$, then the averages
\begin{equation}\label{ab1}
\frac{1}{N-M}\sum_{n=M}^{N-1} \phi_n(x)f_1(T^{a_1n}x)\cdots f_d(T^{a_d n}x)\cdot g_1(T^{p_1(n)}x)\cdots g_s(T^{p_s(n)}x)\rightarrow 0
\end{equation}
in $L^2(X,\mu)$ as $N-M \to\infty$.}
\medskip

In order to  prove \eqref{ab1}, it suffices to show that every sequence of intervals ${\bf I}=(I_N)_{N\in \N}$ whose lengths $|I_N|$ tend to infinity, admits a subsequence ${\bf I}'=(I_N')_{N\in \N}$ such that
\begin{equation}\label{ab2}
  \frac{1}{|I_N'|}\sum_{n\in I_N'} \phi_n(x)f_1(T^{a_1n}x)\cdots f_d(T^{a_d n}x)\cdot g_1(T^{p_1(n)}x)\cdots g_s(T^{p_s(n)}x)\rightarrow 0,\ N\to\infty
\end{equation}
in $L^2(X,\mu)$.

For $x\in X$, we write $b(x)=(b_n(x))_{n\in \N}$ for the sequence defined by
\begin{equation*}
  b_n(x)= f_1(T^{a_1n}x)\cdots f_d(T^{a_d n}x)\cdot g_1(T^{p_1(n)}x)\cdots g_s(T^{p_s(n)}x) .
\end{equation*}
For every $\underline{h}=(h_1,\ldots, h_k)\in \Z^k$ , we study the averages
\begin{equation}\label{ab3}
\frac{1}{|I_N|}\sum_{n\in I_N} \prod_{\ep\in \{0,1\}^k}b_{n+h_1\ep_1+\cdots +h_k\ep_k}(x).
\end{equation}

Let $D=\{\underline{h}=(h_1,\ldots, h_k)\in \Z^k :$ $p_j(n+h_1\ep_1+\cdots +h_k\ep_k), 1\le j\le s, \ep\in \{0,1\}^k$ are essentially distinct integral polynomials $ \}\subset \Z^k$. By \cite[Lemma 6.1]{HSY22-1}, if $p_1, p_2$ are two essentially distinct integral polynomials of degree $\ge 2$, then there is a positive constant $L$ depending on $p_1,p_2$ such that whenever $k_1,k_2\in \mathbb{Z}$ with $|k_1-k_2|\ge L$, the integral polynomials $p_{i_1}(n+k_{j_1})-p_{i_1}(k_{j_1})$ and $p_{i_2}(n+k_{j_2})-p_{i_2}(k_{j_2})$ in one variable $n$ are essentially distinct for each $(i_1,j_1)\neq (i_2,j_2)\in \{1,2\}\times \{1,2\}$. By this fact, it is easy to see that $D$ has zero density in $\Z^k$.

For all $\underline{h}\in D$, since $p_j(n+h_1\ep_1+\cdots +h_k\ep_k), 1\le j\le s, \ep\in \{0,1\}^k$ are essentially distinct integral polynomials with degree $\ge 2$, by Theorem \ref{thm-Leibman} there is some integer $m$ (depending only on the weight of polynomials $\{p_1,\ldots, p_s\}$) such that
\eqref{ab3} converges to $0$ in $L^2(X,\mu)$.

Similar to the proof of Step 1, there exists a subsequence ${\bf I}'=(I_N')_{N\in \N}$ such that for $\mu$-almost every $x\in X$ the limit
\begin{equation}\label{}
  c_{\underline{h}}({\bf I}', b(x))=\lim_{N\to\infty}\frac{1}{|I_N'|} \sum_{n\in I_N} \prod_{\ep\in \{0,1\}^k}b_{n+h_1\ep_1+\cdots +h_k\ep_k}(x)=0
\end{equation}
for every choice of $\underline{h}=(h_1,\ldots, h_k)\in D$. This means that, for $\mu$-almost every $x\in X$, the sequence of intervals ${\bf I}'$ is $k$-adapted to the sequence $b(x)=(b_n(x))_{n\in \N}$ and $c_{\underline{h}}({\bf I}', b(x))=0$ for every $\underline{h}=(h_1,\ldots, h_k)\in D$.
Therefore for $\mu$-a.e. $x\in X$
$$\HK b(x)\HK_{{\bf I}', k}= \left(\lim_{H\to\infty}\frac{1}{H^k}\sum_{\underline{h}\in [0,H-1]^k}c_{\underline{h}}({\bf I}', b(x))\right)^{1/2^k}=0,$$
since $D$ has zero density in $\Z^k$.
By Theorem \ref{thm-HK09},
we have
$$\frac{1}{|I_N'|}\sum_{n\in I_N'}\phi_n(x)b_n(x)\to 0, N\to\infty, \ \mu\text{-almost everywhere},$$
and \eqref{ab2} is proved. This completes the proof of Step 3.

\medskip
\noindent {\bf Step 4.} {\em For all non-zero distinct $a_1,\ldots,a_d\in \Z$ and essentially distinct non-linear integral polynomials $p_1,\ldots, p_s$ for all $f_1,\ldots, f_d, g_1,\ldots, g_s\in L^\infty(X,\mu)$, the averages
\begin{equation}\label{}
\frac{1}{N-M}\sum_{n=M}^{N-1} \phi_n(x)f_1(T^{a_1n}x)\cdots f_d(T^{a_d n}x)\cdot g_1(T^{p_1(n)}x)\cdots g_s(T^{p_s(n)}x)
\end{equation}
converge in $L^2(X,\mu)$ as $N-M \to\infty$.}
\medskip

By Step 3, for all non-zero distinct $a_1,\ldots,a_d\in \Z$ and essentially distinct non-linear integral polynomials $p_1,\ldots, p_s$ for all $f_1,\ldots, f_d, g_1,\ldots, g_s\in L^\infty(X,\mu)$, there is some $m\in \N$ such that if $\HK g_j\HK_{m}=0$ for some $j\in \{1,\ldots, s\}$, then the averages
\begin{equation}\label{}
\frac{1}{N-M}\sum_{n=M}^{N-1} \phi_n(x)f_1(T^{a_1n}x)\cdots f_d(T^{a_d n}x)\cdot g_1(T^{p_1(n)}x)\cdots g_s(T^{p_s(n)}x)
\end{equation}
converge to 0 in $L^2(X,\mu)$.
Thus we can restrict to the case where $g_j, 1\le j\le s$, is measurable with respect to $\ZZ_m$ for some $m\in \N$.

By Remark \ref{rem-CFH}-(2), for every $\ep>0$ there exists $\widetilde{g}_j\in L^\infty(X,\mu)$ such that $\widetilde{g}_j$ is $\ZZ_m$-measurable and $\|g_j-\widetilde{g}_j\|_{L^1(X,\mu)}<\ep, 1\le j\le s$, and for $\mu$-almost all $x\in X$ the sequence
$(\widetilde{g}_j(T^nx))_{n\in \Z}$ is an $m$-step nilsequence.
By \cite[Proposition 3.14]{Leibman05-Isr} or \cite[Theorem 14.15]{HK18}, for $\mu$-almost all $x\in X$ the sequence $(\widetilde{g}_j(S^{p_j(n)}x))_{n\in \Z}$ is a $(m\deg{p_j})$-step nilsequence for $1\le j\le s$. Let $X_1\in \X$ be a full measure subset of $X$ for which this property holds.
Thus $(\Psi_n(x))_{n\in \Z}=\Big(\prod_{j=1}^s\widetilde{g}_j(T^{a_jn} x)\Big)_{n\in \Z}$ is a nilsequence.
Therefore, it suffices to show the existence in $L^2(X,\mu)$ of the limit
\begin{equation*}\label{}
  \lim_{N\to\infty} \frac{1}{N}\sum_{n=0}^{N-1}\phi_n(x)\Psi_n(x)f_1(T^{a_1n}x)\cdots f_d(T^{a_d n}x)
\end{equation*}
This follows from Step 2 since $\phi_n(x)\Psi_n(x)$ is a nilsequence. The whole proof  is complete.
\end{proof}

\section{A lemma about topological models}\label{appendix2}

The proof of the following result was outlined in \cite{F77}. So for completeness, we give a proof here, following the proof of \cite[Theorem 5.15]{F} and \cite[Theorem 2.15]{Glasner}.
Theorem \ref{top-model2} is a little stronger than Theorem \ref{top-model}.

\begin{thm}\label{top-model2}
Let $(X,\X,\mu,T)$ be a m.p.s. and let $\{f_k\}_{k\in \N}\subseteq L^\infty(X,\mu)$. Then there is a t.d.s. $(X',T')$, a $T'$-invariant measure $\mu'$ and an isomorphism $\psi: (X,\X,\mu,T)\rightarrow (X',\B(X'),\mu',T')$
such that for all $k\in \N$, $f'_k= f_k\circ \psi^{-1} \pmod{\mu'}$, where $f_k'\in C(X')$.
\end{thm}

\begin{proof}
Without loss of generality, we assume that $f_k$ is real-valued function for all $k\in \N$. For each $k\in \N$, since $f_k\in L^\infty(X,\mu)$, there is an interval $[a_k,b_k]\subseteq \R$ and a measurable function $g_k: X\rightarrow [a_k,b_k]$ such that $f_k(x)=g_k(x)$ for $\mu$-a.e. $x\in X$. Let $\{I^{(k)}_i\}_{i\in \N}$ be a countable base of $[a_k,b_k]$. Let $B^{(k)}_i=g_k^{-1}(I^{(k)}_i), i\in \N$. Then $\bigcup_{k\in \N}\{B^{(k)}_i\}_{i\in \N}$ is a countable subset of $\X$.

Now choose a sequence $\{A_n\}_{n\in \N}$ of distinct elements of $\X$ such that:
 \begin{enumerate}
   \item The sequence $\{A_n\}_{n\in \N}$ is dense in the measure algebra $(\X,\mu)$, and separate points of $X$ (i.e. for every $x\neq x'$ there is an $A_n$ with $x\in A_n$ but $x'\not\in A_n$);
   \item $\{A_n\}_{n\in \N}$ is invariant under $T$, i.e. there is a permutation $\theta: \N\rightarrow \N$ such that $T^{-1}A_n=A_{\theta(n)}$;
   \item $\bigcup_{k\in \N}\{B^{(k)}_i\}_{i\in \N} \subseteq \{A_n\}_{n\in \N}$.
 \end{enumerate}

Let $\A$ be the countable algebra generated by $\{A_n\}_{n\in \N}$. Now put $X'=\{0,1\}^\N$, $\X'=\B(X')$ and $A_n'=\{\xi\in X': \xi(n)=\xi_n=1\}$. For every $N\in \N$, let $\mu_N$ be the probability measure defined on the finite space $\{0,1\}^N$ by
$$\mu_N(A'_{n_1}\cap A'_{n_2}\cap \cdots \cap A'_{n_s})=\mu(A_{n_1}\cap A_{n-2}\cap \cdots \cap A_{n_s}), 1\le n_1<n_2<\ldots <n_s\le N.$$
Let $\mu'_N$ be any extension of $\mu_N$ to a Borel probability measure on $X'$. Let $\mu_{N_j}'\to \mu', j\to\infty$ be a weak$^*$ convergence subsequence. The measure $\mu'$ satisfies the equalities above for all finite sequences $n_1<n_2<\ldots <n_s$, i.e.
$$\mu'(A'_{n_1}\cap A'_{n_2}\cap \cdots \cap A'_{n_s})=\mu(A_{n_1}\cap A_{n-2}\cap \cdots \cap A_{n_s}), \ \forall 1\le n_1<n_2<\ldots <n_s.$$
The map $\Psi(A_n')=A_n, \forall n\in\N$ can be extended to an isomorphism of measure algebra $\Psi: (\A',\mu')\rightarrow (\A,\mu)$, where $\A'$ is the sub-algebra of $\X'=\B(X')$ determined by $\{A_n'\}_{n\in \N}$. The permutation $\theta$ defines a homeomorphism $T': X'\rightarrow X', \xi\mapsto T'\xi$,
where $$(T'\xi)_n=\xi_{\theta(n)},\ \forall n\in \N .$$
Since $T^{-1}A_n=A_{\theta(n)}$, by the definition of $T'$ it follows that $(T')^{-1}A'_n=A'_{\theta(n)}$. Thus we have that $T': X'\rightarrow X'$ is measure-preserving, and $(X',\X',\mu',T')$ is a m.p.s.

Next define a map
$$\psi: X\rightarrow X', \ \big(\psi(x)\big)_n=1_{A_n}(x).$$
Since the sequence $\{A_n\}_{n\in \N}$ separates points of $X$, $\psi$ is a Borel 1-1 map of $X$ into $X'$ with $\Psi=\psi^{-1}$. Let $X_0'=\psi(X)$. By Lusin-Souslin Theorem \cite[Corollary 15.2]{Kechris}, $X_0'$
is a Borel subset of $X'$ and $\psi$ is a Borel isomorphism of $X$ with $X_0'$. Note that $\mu(X)=\mu'(X_0')=1$ and $(\psi\circ T)(x)=(T'\circ \psi)(x), \forall x\in X$. Thus $\psi: (X,\X,\mu,T)\rightarrow (X',\X',\mu',T')$ is an isomorphism.

Now for each $k\in \N$, $f'_k=g_k\circ\psi^{-1}: X'\rightarrow [a_k, b_k]$, we have that for each $i\in \N$,
$$(f_k')^{-1}(I^{(k)}_i)=(g_k\circ\psi^{-1})^{-1}(I^{(k)}_i)=\psi(g_k^{-1}(I^{(k)}_i))=\psi(B^{(k)}_i)\in \{A_n'\}_{n\in \N}$$
is an open subset of $X'$. Since $\{I^{(k)}_i\}_{i\in \N}$ be a countable base of $[a_k,b_k]$, $f_k'$ is continuous. Note that each $k\in \N$, $f_k(x)=g_k(x)$ for $\mu$-a.e. $x\in X$, and it follows that $f'_k(x')= f_k\circ \psi^{-1}(x')$ for $ \mu'$-a.e. $x'\in X'$.
The proof is complete.
\end{proof}


\end{document}